\documentclass[11pt,reqno]{amsart}
\usepackage{amsmath}
\usepackage{amsfonts}
\usepackage{amssymb}
\usepackage{graphicx}
\usepackage{amsthm,graphicx,color,yfonts}
\usepackage{pdfsync}
\usepackage{epstopdf}%
\usepackage{pifont}
\usepackage{bbm}
\usepackage{a4wide}

\usepackage[colorlinks=true]{hyperref}
\hypersetup{linkcolor=red,citecolor=blue,filecolor=dullmagenta,urlcolor=blue} 

\usepackage{wrapfig}
\usepackage{tikz}
\usetikzlibrary{arrows,calc,decorations.pathreplacing}
\definecolor{light-gray1}{gray}{0.90}
\definecolor{light-gray2}{gray}{0.80}
\definecolor{light-gray3}{gray}{0.60}

\newcommand{\wt}{\widetilde}

\newcommand{\vp}{\varphi}

\newcommand{\R} {\mathbb R}
\newcommand{\cuad}{{\sqcap\kern-.68em\sqcup}}

\newcommand{\ve}{\varepsilon}

\newcommand{\be}{\begin{equation}}
\newcommand{\ee}{\end{equation}}

\newcommand{\la}{\lambda}

\definecolor{darkgreen}{rgb}{0.2,0.7,0.1}

\newcommand{\sech}{\mathop{\mbox{\normalfont sech}}\nolimits}

\newcommand{\px}{\partial_x}

\newcommand{\nlop}{(1-\partial_x^2)^{-1}}
\def\wh#1{\widehat{#1}}

\newcommand{\al}{\alpha}
\newcommand{\bt}{\beta}

\def\bm{\left( \begin{array}{cc}}
\def\endm{\end{array}\right)}

\providecommand{\norm}[1]{\left\| #1 \right\|}

\newcommand{\ba}{\begin{equation*}}
\newcommand{\ea}{\begin{equation*}}
\newcommand{\bea}{\begin{eqnarray}}
\newcommand{\eea}{\end{eqnarray}}
\newcommand{\bee}{\begin{eqnarray*}}
\newcommand{\eee}{\end{eqnarray*}}
\newcommand{\ben}{\begin{enumerate}}
\newcommand{\een}{\end{enumerate}}

\numberwithin{equation}{section}

\newtheorem{theorem}{Theorem}[section]
\newtheorem*{theorem*}{Theorem}
\newtheorem{proposition}{Proposition}[section]
\newtheorem{corollary}{Corollary}[section]
\newtheorem{lemma}{Lemma}[section]
\newtheorem{definition}{Definition}[section]

\theoremstyle{remark}
\newtheorem{remark}{Remark}[section]

\title[Improved decay]{Asymptotic dynamics for the small data weakly dispersive one-dimensional Hamiltonian ABCD system}

\author{Chulkwang Kwak}
\address{Facultad de Matem\'aticas, Pontificia Universidad Cat\'olica de Chile and Institute of Pure and Applied Mathematics, Chonbuk National University}
\email{chkwak@mat.uc.cl}
\thanks{C. K. is supported by FONDECYT Postdoctorado 2017 Proyecto No. 3170067.}

\author{Claudio Mu\~noz}
\address{CNRS and Departamento de Ingenier\'{\i}a Matem\'atica and Centro
de Modelamiento Matem\'atico (UMI 2807 CNRS), Universidad de Chile, Casilla
170 Correo 3, Santiago, Chile.}
\email{cmunoz@dim.uchile.cl}
\thanks{C. M. work was partly funded by Chilean research grants FONDECYT  1150202, project France-Chile ECOS-Sud C18E06 and CMM Conicyt PIA AFB170001.}
\thanks{One of us (C.M.) would like to thank the Applied Mathematics Department of University of Granada, Spain, where part of this work was completed.}

\subjclass{35Q35,35Q51}

\begin{document}

\begin{abstract}
Consider the Hamiltonian $abcd$ system in one dimension, with data posed in the energy space $H^1\times H^1$. This model, introduced by Bona, Chen and Saut \cite{BCS1,BCS2}, is a well-known physical generalization of the classical Boussinesq equations \cite{Bous}.  The Hamiltonian case corresponds to the regime where $a,c<0$ and $b=d>0$. Under this regime, small solutions in the energy space are globally defined. A first proof of decay for this $2\times 2$ system was given in \cite{KMPP2018}, in a strongly dispersive regime, i.e. under essentially the conditions 
\[
b=d > \frac29, \quad a,c<-\frac1{18}.
\]
Additionally, decay was obtained inside a proper subset of the light cone $(-|t|,|t|)$. In this paper, we improve \cite{KMPP2018} in three directions. First, we enlarge the set of parameters $(a,b,c,d)$ for which decay to zero is the only available option, considering now the so-called weakly dispersive regime $a,c\sim 0$: we prove decay if now
\[
b=d > \frac3{16}, \quad a,c<-\frac1{48}.
\]
This result is sharp in the case where $a=c$, since for $a,c$ bigger, some $abcd$ linear waves of \emph{nonzero frequency} do have \emph{zero group velocity}. Second, we sharply enlarge the interval of decay to consider the whole light cone, that is to say, any interval of the form $|x|\sim |v|t$, for any $|v|<1$. This result rules out, among other things, the existence of nonzero speed solitary waves in the regime where decay is present. Finally, we prove decay to zero of small $abcd$ solutions in exterior regions $|x|\gg |t|$, also discarding super-luminical small solitary waves. These three results are obtained by performing new improved virial estimates for which better decay properties are deduced.
\end{abstract}

\maketitle

\section{Introduction}

\subsection{Setting} Consider the initial value problem (IVP) for the one-dimensional $abcd$ system:
\begin{equation}\label{boussinesq}
\begin{aligned}
&(1- b\,\partial_x^2)\partial_t \eta  + \partial_x \! \left( a\, \partial_x^2 u +u + u \eta \right) =0, \quad (t,x)\in \R\times\R, \\
&(1- d\,\partial_x^2)\partial_t u  + \partial_x \! \left( c\, \partial_x^2 \eta + \eta  + \frac12 u^2 \right) =0,\\
&  \eta(t=0)=\eta_0, \quad u(t=0) =u_0.
\end{aligned}
\end{equation}
Here, $\eta=\eta(t,x)$ and $u=u(t,x)$ are real-valued unknowns, and $(\eta_0,u_0)$ are given initial data in $(H^1\times H^1)(\R)$.

\medskip

The $abcd$ system \eqref{boussinesq} was introduced in 2002 in a foundational paper by Bona, Chen and Saut \cite{BCS1,BCS2}, as an improved version of the original Boussinesq equations \cite{Bous} obtained from the water waves equation in the long wave regime, and it is by now a canonical shallow water model. The constants $a,b,c,d$ must follow the conditions \cite{BCS1}
\be\label{Conds}
a,~c<0, \quad a +b =\frac1{2}\left(\theta^2-\frac13\right),\quad c+d =\frac1{2} (1-\theta^2)\geq 0,
\ee
for some $\theta\in [0,1]$ (this case is referred as the regime \emph{without surface tension}). The \emph{physical perturbation parameters} under which the long wave expansion is performed in the water waves model are
\[
\al := \frac Ah \ll 1, \quad \bt := \frac{h^2}{\ell^2} \ll 1, \quad \al \sim \bt,
\]
and where $A$ and $\ell$ are typical waves amplitude and wavelength respectively, and $h$ is the constant depth.

\medskip

Note that in the case $\eta=const.u$, with $b$ and $d$ properly chosen, and after suitable rescaling and time translation $x\mapsto x-t$, \eqref{boussinesq} becomes the classical Benjamin-Bona-Mahony (BBM) equation \cite{BBM}. In that sense, the $abcd$ system is a physically motivated improvement and generalization of the BBM equation, and it is also believed to be nonintegrable as BBM. Equations \eqref{boussinesq} are part of a hierarchy of Boussinesq models, including second order systems, which were obtained in \cite{BCS1} and \cite{BCL}. See these papers for a more accurate description of the physical relevance of \eqref{boussinesq}.

\medskip

A rigorous justification for the $abcd$ model from the free surface Euler equations (as well as extensions to higher dimensions) is given by Bona, Colin and Lannes \cite{BCL}, see also Alvarez-Samaniego and Lannes \cite{ASL} for improved results. Since then, these models have been extensively studied in the literature, see e.g. \cite{BLS,BCL,LPS,Saut} and references therein for a detailed account of known results. 

\medskip

As for the low regularity Cauchy problem associated to \eqref{boussinesq} and its generalizations to higher dimensions, Saut et. al. \cite{SX,SWX} studied in great detail the long time existence problem by focusing in the small physical parameter $\ve$ appearing in the asymptotic expansions. They showed well-posedness (on a time interval of order $1/\ve$) for \eqref{boussinesq}. We also refer to Burtea's work \cite{Burtea} (compared to \cite{SWX}) for an additional improvement of \cite{SX} under some directions, including most of ``generic" cases and low-regularity regime (in comparison with \cite{SX}) for which the well-posedness holds true. Previous results by Schonbek \cite{Schonbek} and Amick \cite{Amick} considered the case $a=c=d=0$, $b=\frac13$, a version of the original Boussinesq system, proving global well-posedness under a non cavitation condition, and via parabolic regularization. Linares, Pilod and Saut \cite{LPS} considered existence in higher dimensions for time of order $O(\ve^{-1/2})$, in the KdV-KdV regime $(b=d=0)$. On the other hand, ill-posedness results and blow-up criteria (for the case $b=1$, $a=c=0$), are proved in \cite{CL}, following the ideas in Bona-Tzvetkov \cite{BT}.

\medskip

It is well-known that \eqref{boussinesq} admits (big) solitary waves in certain regimes of $a,b,c,d$, see \cite{BCL0} and references therein for details. Those solitary waves are globally defined solutions in the energy space with no decay. Below, we will discuss further this topic.

\medskip

Whenever $b=d$, the system \eqref{boussinesq} is {\bf Hamiltonian}, and globally well-posed in the energy space $H^1\times H^1$ \cite{BCS2}, at least for small data, thanks to the conservation of the energy
\be\label{Energy}
\begin{aligned}
E[u ,\eta ](t):=&~{} \frac12\int \left( -a (\partial_x u)^2 -c (\partial_x \eta)^2  + u^2+ \eta^2 + u^2\eta \right)(t,x)dx,\\
&~{} \hbox{with } a,c<0.
\end{aligned}
\ee
This will be the setting that we shall assume in this paper: globally defined small solutions for the Hamiltonian case. Notice that is unknown whether or not large or non Hamiltonian solutions of \eqref{boussinesq} may develop singularities in finite time. In the global existence setting, the case where $a,c\sim 0$ is essentially the {\bf weakly dispersive case}, and on of the main subjects of this paper.

\subsection{Scattering and decay} Assume that $(\eta,u)\in C(\R,(H^1\times H^1)(\R))$ is a globally defined solution to \eqref{boussinesq}. Can one find its asymptotic behavior as $t\to\pm\infty$, in terms of scattering? By scattering, we mean the existence of final states $(u^\pm,\eta^\pm)\in H^1\times H^1$ such that 
\be\label{Scat}
\lim_{t\to \pm\infty} \| (u,\eta)(t) - S(t)(u^\pm,\eta^\pm)\|_{H^1\times H^1}=0,
\ee
and $S(t)$ is the associated linear flow in \eqref{boussinesq}. This question is far from trivial, because of the following reasons: 

\begin{itemize}
\item First of all, the nonlinearities are too weak (just quadratic) for getting \eqref{Scat}, and modified scattering is expected.
\smallskip
\item Second, the low dimension (=1) makes the associated linear decay the worst possible, of order $O(t^{-1/3})$ for a wide range of parameters $(a,b,c,d)$ \cite{Munoz_Rivas}.
\smallskip
\item Third, there is the presence of (non decaying) solitary waves, with zero and nonzero speeds \cite{BCL0}. Therefore, \eqref{Scat} cannot hold in that case.
\smallskip
\item Finally, \eqref{boussinesq} is a $2\times2$ system, which implies that scattering techniques are harder to apply than usual \cite{Munoz_Rivas}. 
\end{itemize}

The first known ``scattering'' result for \eqref{boussinesq}, in the alternative form of an explicit proof of decay to zero in time-depending regions of space, was proved in \cite{KMPP2018}. In order to state this result, first recall the following definition, also important for the main results of this paper.

\begin{definition}[Intervals of decay] Let $v\in [-1,1]$ be any fixed number (speed), and $t\in \R$ such that $|t|\geq 2$. For each $(t,v)$, we define the interval of decay $J_v(t)\subset\R$ as
\be\label{Jv}
J_v(t):= \left( v|t| - \frac{ |t|}{\log^2 |t|}, v|t|+ \frac{|t|}{\log^2 |t|}\right).
\ee 
\end{definition}

Essentially, $J_v(t)\sim \{|x|\sim v|t|\}$. One has

\begin{theorem}[Decay for the strongly dispersive Hamiltonian $abcd$ system, see \cite{KMPP2018}]\label{Thm2}
Let $(u,\eta)\in C(\R, H^1\times H^1)$ be a global, small solution of \eqref{boussinesq}-\eqref{Conds}, such that for some $\ve>0$ small
\be\label{Smallness}
\|(u_0,\eta_0)\|_{H^1\times H^1}< \ve.
\ee
Let $v=0$, $|t|\geq 2$ and $J_0(t)$ be as in \eqref{Jv}. Assume additionally that $(a,c)$ are dispersion-like parameters, see \eqref{dispersion_like}. Then, there is strong decay in $J_0(t)$:
\be\label{Conclusion_0}
\lim_{t \to \pm\infty}   \|(u,\eta)(t)\|_{(H^1\times H^1)(J_0(t))} =0.
\ee
\end{theorem}

By \emph{dispersion-like parameters}, we mean the following \cite{KMPP2018}: $(a,b,c)$ satisfying \eqref{Conds} are such that\footnote{Recall that since $abcd$ is Hamiltonian, one has $b=d$.}
\be\label{dispersion_like}
3b(a+c) + 2b^2  < 8ac. 
\ee
This condition means that the negative $(a,c)$ are both not sufficiently close to zero, or in other words, \eqref{boussinesq} has enough dispersion to ensure decay to zero of its small solutions. In \cite{KMPP2018} it was additionally proved that the parameters of the Hamiltonian $abcd$ system (see the physical restrictions \eqref{Conds}) must obey the condition 
\be\label{primera}
b>\frac16\sim 0.17,
\ee
and that 
\be\label{segunda}
\hbox{\eqref{dispersion_like}\quad  holds if \quad $b>\frac29\sim 0.22$.}
\ee
This essentially says that Theorem \ref{Thm2} only concerns with a strongly dispersive regime, so the regime $\frac16 < b \leq \frac29$, which describes the weakly dispersive regime, was left open. Indeed, for instance in the case $a=c$, \eqref{dispersion_like} in addition to \eqref{Conds} (in particular, $a = c = \frac16 - b$) implies 
\be\label{Strong}
a=c <-\frac1{18}.
\ee
Consequently, if $b<\frac29$, then $(a,c)$ are above these last values. 
Let us be more precise. A very useful parametrization of the set of parameters $(a,b,c)$ for which \eqref{primera} and \eqref{segunda} can be understood is given in terms of two parameters, $b$ and $\nu$ \cite{KMPP2018}:
\begin{equation}\label{R0}
\begin{aligned}
(a,b,c) = &~ \left(-\frac{\nu}{2} + \frac13 -b,\; b,\; \frac{\nu}{2} - b \right), \quad (\nu,b)\in \mathcal R_0,\\
\mathcal R_0:= & ~ \left\{ (\nu,b) ~ : ~ \nu \in [0,1] \cap \left(\frac23 - 2b, 2b \right), ~b>\frac16 \right\}.
\end{aligned}
\ee
Therefore, the Hamiltonian $abcd$ system is valid only in the set $\mathcal R_0$ described in Figure \ref{Fig:0} (left). On the other hand, the  set $\mathcal R \subseteq \mathcal R_0$ of pairs $(\nu,b)$ under which \eqref{dispersion_like} is valid is given in Figure \ref{Fig:0} (right).

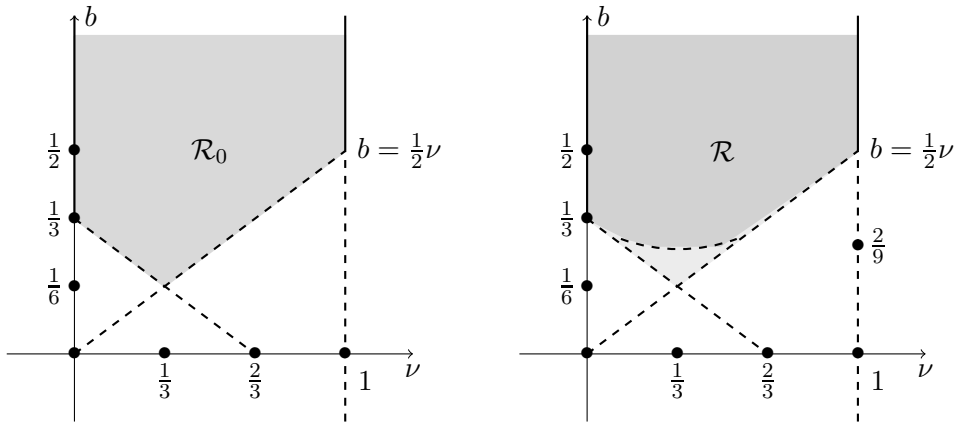
\begin{figure}[h!]
\begin{center}
\begin{tikzpicture}[scale=0.9]
\filldraw[thick, color=lightgray!60] (0,4.7)--(0,2) -- (4/3,1) --(4,3) -- (4,4.7) -- (0,4.7);
\draw[thick,dashed] (4,-1) -- (4,3);
\draw[thick] (4,3) -- (4,5);
\draw[thick,dashed] (0,0) -- (4,3);
\draw[thick,dashed] (0,2)--(8/3,0);
\draw[thick] (0,2) -- (0,5);
\draw[->] (-1,0) -- (5,0) node[below] {$\nu$};
\draw[->] (0,-1) -- (0,5) node[right] {$b$};
\node at (0,0){$\bullet$};
\node at (4,0){$\bullet$};
\node at (4.3,-0.4){$1$};
\node at (4/3,0){$\bullet$};
\node at (8/3,-0.4){$\frac23$};
\node at (4/3,-0.4){$\frac13$};
\node at (8/3,0){$\bullet$};
\node at (0,2){$\bullet$};
\node at (-0.3,2){$\frac 13$};
\node at (0,1){$\bullet$};
\node at (-0.3,1){$\frac16$};
\node at (0,3){$\bullet$};
\node at (-0.3,3){$\frac12$};
\node at (4.8,3){$b=\frac12\nu$};
\node at (2,3){$ \mathcal{R}_0$};
\end{tikzpicture}
\qquad
\begin{tikzpicture}[scale=0.9]
\filldraw[thick, color=lightgray!30] (0,4.7)--(0,2) -- (4/3,1) --(4,3) -- (4,4.7) -- (0,4.7);
\filldraw[thick, color=lightgray!70] (0,4.7)--(0,2) -- (0.6,1.7) --(1,1.6) --(4/3,1.6) --(1.97,1.6) --(4,3) -- (4,4.7) -- (0,4.7);
\draw[thick,dashed] (4,-1) -- (4,3);
\draw[thick] (4,3) -- (4,5);
\draw[thick,dashed] (0,0) -- (4,3);
\draw[thick,dashed] (0,2)--(8/3,0);
\draw[thick] (0,2) -- (0,5);
\draw[->] (-1,0) -- (5,0) node[below] {$\nu$};
\draw[->] (0,-1) -- (0,5) node[right] {$b$};
\node at (0,0){$\bullet$};
\node at (4,0){$\bullet$};
\node at (4.3,-0.4){$1$};
\node at (4/3,0){$\bullet$};
\node at (8/3,-0.4){$\frac23$};
\node at (4/3,-0.4){$\frac13$};
\node at (8/3,0){$\bullet$};
\node at (0,2){$\bullet$};
\node at (-0.3,2){$\frac 13$};
\node at (0,1){$\bullet$};
\node at (-0.3,1){$\frac16$};
\node at (0,3){$\bullet$};
\node at (4,1.6){$\bullet$};
\node at (4.3,1.6){$\frac29$};
\node at (-0.3,3){$\frac12$};
\node at (4.8,3){$b=\frac12\nu$};
\draw[thick,dashed] (0.5,1.7) arc (250:290:2.5);
\node at (2,3){$ \mathcal R$};
\end{tikzpicture}
\end{center}
\caption{(\emph{Left}). The set $\mathcal R_0$ under which the Hamiltonian $abcd$ system (without surface tension) makes sense. Note that each point has associated values $(a,c)$ via formula  \eqref{R0}, and the set of admissible values makes sense only if $b>\frac16$. Also, note that at $(\nu,b)=(\frac13,\frac16)$, one has $(a,c)=(0,0)$ (no dispersion in \eqref{boussinesq}). (\emph{Right}). The set $\mathcal R$ under which \eqref{dispersion_like} holds, and Theorem \ref{Thm2} is valid. Note that $b=\frac29$ represents the bottom of this set. Finally note that at $(\nu,b)=(\frac13,\frac29)$, one has $(a,c)=(-\frac{1}{18},-\frac{1}{18}).$ Figure taken from \cite{KMPP2018}.}\label{Fig:0}
\end{figure}

\medskip

Having explained the importance of strongly dispersive parameters $(a,b,c,d)$ in Theorem \ref{Thm2}, let us come back to the statement of Theorem \ref{Thm2}.

\begin{remark}
Note that since \eqref{boussinesq} is Hamiltonian, decay to zero in the whole energy space $H^1\times H^1$ would imply that the solution is identically zero. In that sense, \eqref{Conclusion_0} is sharp.
\end{remark}

\begin{remark}
Since the data is only assumed in the energy space, \eqref{Conclusion_0} does not give an explicit rate of decay for the solution; it is expected that a better description of the rate of decay could be obtained by assuming better spatial decay properties for the initial data, as it is usually required in scattering techniques.   
\end{remark}

\begin{remark}
An equivalent, but also useful description of the set of $(a,b,c)$-parameters{\color{black}, for which the parameters satisfy \eqref{Conds}}, and Theorem \ref{Thm2} is valid, is given in Fig. \ref{Fig:5}. In this configuration, $(a,c)$ are in functions of the remaining parameter $b$.
\end{remark}

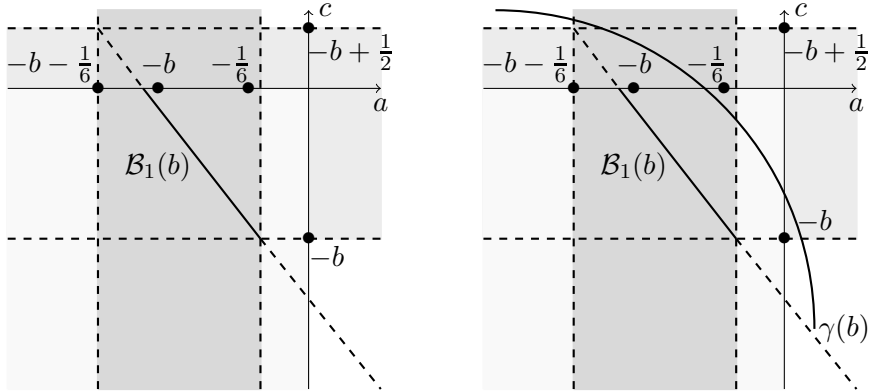
\begin{figure}[h!]
\begin{center}
\begin{tikzpicture}[scale=0.8]
\filldraw[thick, color=lightgray!30] (-1,1.5)--(5.2,1.5) -- (5.2,5) --(-1,5) -- (-1,1.5);
\filldraw[thick, color=lightgray!10] (-1,-1)--(-1,4) -- (4,4) --(4,-1) -- (-1,-1);
\filldraw[thick, color=lightgray!60] (0.5,-1)--(0.5,5.3) -- (3.2,5.3) --(3.2,-1) -- (0.5,-1);
\draw[thick, color=black] (1.25,4) -- (3.2,1.5);
\draw[thick,dashed] (0.5,5) -- (1.25,4);
\draw[thick,dashed] (3.2,1.5) -- (5.2,-1);
\draw[thick,dashed] (0.5,-1)--(0.5,5.3);
\draw[thick,dashed] (3.2,-1)--(3.2,5.3);
\draw[thick,dashed] (-1,1.5)--(5.2,1.5);
\draw[thick,dashed] (-1,5)--(5.2,5);
\draw[->] (-1,4) -- (5.2,4) node[below] {$a$};
\draw[->] (4,-1) -- (4,5.3) node[right] {$c$};
\node at (3,4){$\bullet$};
\node at (2.7,4.4){$-\frac16$};
\node at (1.5,4){$\bullet$};
\node at (1.5,4.4){$-b$};
\node at (0.5,4){$\bullet$};
\node at (-0.3,4.4){$-b -\frac16$};
\node at (4,1.5){$\bullet$};
\node at (4.3,1.2){$-b$};
\node at (4,5){$\bullet$};
\node at (4.7,4.6){$-b+\frac12$};
\node at (1.5,2.7){$ \mathcal{B}_1(b)$};
\end{tikzpicture}
\qquad
\begin{tikzpicture}[scale=0.8]
\filldraw[thick, color=lightgray!30] (-1,1.5)--(5.2,1.5) -- (5.2,5) --(-1,5) -- (-1,1.5);
\filldraw[thick, color=lightgray!10] (-1,-1)--(-1,4) -- (4,4) --(4,-1) -- (-1,-1);
\filldraw[thick, color=lightgray!60] (0.5,-1)--(0.5,5.3) -- (3.2,5.3) --(3.2,-1) -- (0.5,-1);
\draw[thick, color=black] (1.25,4) -- (3.2,1.5);
\draw[thick,dashed] (0.5,5) -- (1.25,4);
\draw[thick,dashed] (3.2,1.5) -- (5.2,-1);
\draw[thick,dashed] (0.5,-1)--(0.5,5.3);
\draw[thick,dashed] (3.2,-1)--(3.2,5.3);
\draw[thick,dashed] (-1,1.5)--(5.2,1.5);
\draw[thick,dashed] (-1,5)--(5.2,5);
\draw[->] (-1,4) -- (5.2,4) node[below] {$a$};
\draw[->] (4,-1) -- (4,5.3) node[right] {$c$};
\node at (3,4){$\bullet$};
\node at (2.7,4.4){$-\frac16$};
\node at (1.5,4){$\bullet$};
\node at (1.5,4.4){$-b$};
\node at (0.5,4){$\bullet$};
\node at (-0.3,4.4){$-b -\frac16$};
\node at (4,1.5){$\bullet$};
\node at (4.5,1.8){$-b$};
\node at (4,5){$\bullet$};
\node at (4.7,4.6){$-b+\frac12$};
\node at (1.5,2.7){$ \mathcal{B}_1(b)$};
\draw[thick] (4.5,0) arc (0:90:5.3);
\node at (5,0){$\gamma(b)$};
\end{tikzpicture}
\end{center}
\caption{(\emph{Left.}) Schematic representation of the admissible pairs $(a,c)$ (in terms of $b$) under which $abcd$ is valid ($b>\frac16$). Here, the continuous line $\mathcal{B}_1(b)$ represents this set. Note that $(a,c)$ cannot be arbitrarily small. (\emph{Right.}) The curve $\gamma(b)$ represents now the borderline for the validity of Theorem \ref{Thm2}: below this curve \eqref{dispersion_like} holds true. In other words, for $b$ sufficiently small (below 2/9), the curve $\gamma(b)$ is completely \emph{below} $\mathcal{B}_1(b)$, and Theorem \ref{Thm2} fails. See \cite{KMPP2018} for further details.} \label{Fig:5}
\end{figure}

The proof of Theorem \ref{Thm2} (see \cite{KMPP2018}) was based essentially in the use of three particular virial identities adapted to the $acbd$ dynamics, in the spirit of previous works by Martel and Merle \cite{MM,MM1,MM2}, Merle and Rapha\"el \cite{MR}, and recently by Kowalczyk, Martel and the second author \cite{KMM1,KMM2,KMM3} (see also \cite{AM1,MPP} for similar results). Let us be more precise: a well-cooked virial functional $\mathcal I_{0}(t):= \mathcal I_0[u,\eta](t)$ is said to control the dynamics of decay as $t\to +\infty$ if one is able to show that $\mathcal I_0(t)$ is bounded for all time, and 
\[
\frac{d}{dt}\mathcal I_0(t) \gtrsim \|(u,\eta)(t)\|_{(H^1\times H^1)_{loc}}^m,\quad m>0.
\]
This implies that at least for a sequence of times $t_n\to+\infty$, one has decay along compact sets in space: $\lim_{n\to+\infty}\|(u,\eta)(t_n)\|_{(H^1\times H^1)_{loc}}=0$. This elementary but no less important principle is also the key ingredient for each proof proposed in this paper, where we further improve the virials in \cite{KMPP2018}, and construct other two additional virial functionals. The virial method has the advantage of allowing --otherwise impossible to attain by today-- supercritical scattering nonlinearities, and only data in the energy space is necessary. Its drawback is the lack of an explicit rate of decay for the solution, unless one assumes more on the data. 

\bigskip

\section{New results}\label{sec:new results}  At a first sight, at least three important questions were left open in \cite{KMPP2018}: in the small data regime, 
\begin{itemize}
\item Can one improve \eqref{dispersion_like} to include all Hamiltonian $abcd$ systems, namely all negative $(a,c)$ satisfying \eqref{Conds}?  This amounts to consider the weakly dispersive regime $a,c\sim 0$, where dispersion is very weak compared with nonlinear terms. In this regime, $abcd$ becomes closer to the hyperbolic regime \cite{Schonbek,Amick,CL}.
\smallskip
\item Can one improve the interval of decay $J_0(t)$ to include the whole light cone $[-t,t]$? This question is related to the existence of moving solitary waves \cite{BCL0,CNS1,CNS2,Olivera}: a proof of decay in moving regions will preclude the existence of small solitary waves in those regimes. Theorem \ref{Thm2} only rules out small solitary waves of zero speed.
\smallskip
\item  A description of the dynamics in the exterior region $[-t,t]^c$ was left completely open. We do not know if super-luminical (speeds greater than 1) solitary waves may exist. 
\end{itemize}

The purpose of this paper is to give positive answers to the previous questions. Some of them are sharp, but others are still far from optimal. Since we will use plenty of times some particular concepts, in order to encompass them we prefer to introduce the following standard notation:

\begin{definition}[Interval of energy decay and energy decay region]
Let $(u,\eta)$ be a global $H^1\times H^1$ solution of the Hamiltonian $abcd$ \eqref{boussinesq}-\eqref{Conds}, and let $t\geq 1$. We say that an interval $I(t)\subseteq\R$ is an \emph{interval of energy decay} if 
\be\label{Fundamental_0}
\liminf_{t\to+\infty} |I(t)|>0, \qquad \hbox{and} \qquad  \lim_{t\to \infty} \|(u,\eta)(t)\|_{(H^1\times H^1)(I(t))} =0.
\ee
Given a nonempty interval of energy decay $I(t)$, we define the \emph{energy decay} region, denoted by $ED$, as the space-time set of $(t,x)\in [1,\infty) \times I(t)$. A completely analogous definition is available for times $t\to-\infty$.
\end{definition}

\medskip

As an example, $J_0(t)=J_{v=0}(t)$ defined in \eqref{Jv} corresponds to an energy decay interval for the Hamiltonian $abcd$ system, provided $(a,b,c) $ are dispersion-like parameters, in view of Theorem \ref{Thm2}. Note also that $I(t)=\R$ implies $(u,\eta)\equiv (0,0)$.

\begin{remark}
The notion of energy decay just introduced may differ from the standard notion of \emph{dispersive decay}, or decay of the $L^\infty$ norm. Indeed, note that from Theorem \ref{Thm2} one has \eqref{Fundamental_0} with $I(t)=J_0(t)$. Therefore, in this case, and by Sobolev embedding in one dimension, for any compact open interval $I_0$,
\be\label{Fundamental_0_a}
 \lim_{t\to \infty} \|(u,\eta)(t)\|_{(L^\infty\times L^\infty)(I_0)} =0.
\ee
However, it may be the case that, for some space interval $I(t)$, one has \eqref{Fundamental_0_a}, but instead
\be\label{Fundamental_0_b}
 \limsup_{t\to \infty} \|(u,\eta)(t)\|_{(H^1\times H^1)(I(t))} >0.
\ee
We will come back to this question in the next pages.
\end{remark}

\medskip

Consider now the linear flow associated to the Hamiltonian $abcd$ \eqref{boussinesq}, namely
\begin{equation}\label{boussinesq_lineal}
\begin{aligned}
&(1- b\,\partial_x^2)\partial_t \eta  + \partial_x \! \left( a\, \partial_x^2 u +u \right) =0, \quad (t,x)\in \R\times\R, \\
&(1- b\,\partial_x^2)\partial_t u  + \partial_x \! \left( c\, \partial_x^2 \eta + \eta   \right) =0.
\end{aligned}
\end{equation}
For this problem, recall the well-known notion of plane wave.

\begin{definition}[Plane wave and plane wave region]\label{def:PW} Let $k,w,A\in\R$, and $(a,b,c)$ satisfying \eqref{Conds}. We say that $(u_{pw},\eta_{pw})=(e^{i(kx-wt)},Ae^{i(kx-wt)})$  is a \emph{plane wave} for \eqref{boussinesq_lineal} if there exist $A=A(k)=A(k;a,b,c)$ and $w=w(k)=w(k;a,b,c)$ such that $(u_{pw},\eta_{pw})$ solve \eqref{boussinesq_lineal}. The quantity
$w'(k)$ (if exists) is called the \emph{group velocity} associated to \eqref{boussinesq_lineal}, and the set of rays 
\[
\{ (t, w'(k)t) \in [1,\infty) \times \R ~:~  k\in\R \},
\]
is denoted by $PW$, or plane wave region.
\end{definition}

In the case of \eqref{boussinesq_lineal}, one has
\begin{equation}\label{dispersion relation_0}
w(k) = \frac{\pm |k| }{(1+ bk^2)}(1-ak^2)^{1/2}(1-ck^2)^{1/2},
\end{equation}
\[
A(k)= \frac{k(1-ak^2)}{w(k)(1+bk^2)} = \frac{w(k)(1+bk^2)}{k(1-ck^2)},
\]
and
\begin{equation}\label{group velocity_0}
|w'(k)|= \frac{|abck^6 +3ack^4 -(b+2a+2c)k^2 +1|}{(1+bk^2)^2 (1-ak^2)^{1/2}(1-ck^2)^{1/2}}.
\end{equation}
Note the complexity of the group velocity, which, first of all, combines all the constants $(a,b,c)$, even in the Hamiltonian case $b=d$. 

\medskip

Our first result relates the $PW$ region with the $ED$ region, as approximate duals in some sense. We classify the decay and plane wave regions in terms of the most appropriate parameter $b$, in the special case $a=c$. Note that $b>\frac16$, see \eqref{primera}. Essentially, we will have the following decoupling of the half-line $b>\frac16$:

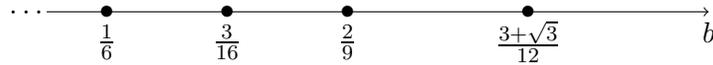
\begin{figure}[h!]
\begin{center}
\begin{tikzpicture}[scale=0.8]
\draw[->] (-6,3) -- (5,3) node[below] {$b$};
\node at (-6.3,3){$\dots$};
\node at (-5,3){$\bullet$};
\node at (-3,3){$\bullet$};
\node at (-1,3){$\bullet$};
\node at (2,3){$\bullet$};
\node at (-5,2.5){$\frac16$};
\node at (-3,2.5){$\frac3{16}$};
\node at (-1,2.5){$\frac29$};
\node at (2,2.5){$\frac{3+\sqrt{3}}{12}$};
\end{tikzpicture}
\end{center}
\end{figure}
Having in mind this last figure, and $J_v(t)$ in \eqref{Jv}, we will prove in this paper the following result.

\begin{theorem}[Description of the Hamiltonian $abcd$ system small data dynamics, case $a=c$]\label{TH0}
Assume that $a=c$ in \eqref{boussinesq}-\eqref{Conds}. Let $\epsilon,\delta>0$ be arbitrarily parameters, and $(u,\eta)$ be a sufficiently small, globally defined solution of \eqref{boussinesq}, its initial data $H^1\times H^1$ norm depending on $\epsilon,\delta$ if necessary. Then, the following holds. 

\smallskip
Assume that either
\ben
\item $\frac16 < b \le \frac{3}{16}$, and $I(t):=\big(-\infty, -(1+\epsilon)t\big)\cup \big((1+\delta)t,\infty\big)$ (Fig. \ref{Fig:11a} left);
\smallskip
\item $\frac{3}{16} < b \le \frac{2}{9}$,  and 
\[
I(t):=\big(-\infty, -(1+\epsilon)t\big)\cup J_0(t) \cup \big((1+\delta)t,\infty\big)
\]
(Fig. \ref{Fig:11a} right);
\smallskip
\item $\frac{2}{9} < b \le \frac{3+\sqrt{3}}{12} $, $|v| < 1-\frac{2}{9b}$ and  
\[
I(t):= \big(-\infty, -(1+\epsilon)t \big)\cup J_v(t) \cup \big((1+\delta)t,\infty \big)
\]
(Fig. \ref{Fig:11b} above);
\smallskip
\item $  b>\frac{3+\sqrt{3}}{12} $, $\sigma_0:=\sigma_0(b):= \frac{2(b-\frac16)(b-\frac18)}{b(b-\frac1{12})} >0$, $\sigma > \sigma_0$, $|v| < 1-\frac{2}{9b}$ and  
\[
I(t):= (-\infty, - \sigma t )\cup J_v(t) \cup (\sigma t,\infty)
\]
(Fig. \ref{Fig:11b} below);
\een
then $(u,\eta)$ decay in $ED= [1,\infty)\times I(t)$:
\be\label{Modelo}
\lim_{t\to \infty}\|(u,\eta)\|_{H^1\times H^1(I(t))} =0.
\ee
\end{theorem}

\bigskip

Some comments are necessary.

\begin{remark}[Extended decay]
In Items (3) and (4) of Theorem \ref{TH0}, the interval $J_v(t)$ can be extended to the union of a finite number of $J_{v_j}(t)$ (if the $v_j$'s are mutually distinct), see Remark \ref{rem:finite number} below.
\end{remark}

\begin{remark}[About Plane-Wave regions and the duality ED-PW]\label{rem:ED-PW}
Complementing Theorem \ref{TH0} (in the sense that energy decay and plane wave sets seem to be always disjoint), the plane wave regions $PW$ can be easily determined from \eqref{group velocity_0}, and are presented in Figs. \ref{Fig:11a} and \ref{Fig:11b}. Note that the $PW$ regions are determined using the linear part of the equation \eqref{boussinesq_lineal}, and \emph{do not necessarily represent} part of the long time behavior of the nonlinear $abcd$ system. For this reason, we have not explicitly included them in Theorem \ref{TH0}. An analogous phenomenon can be seen in the case of the generalized KdV (gKdV) equation \cite{MM, MM1, MM2}: precisely the left portion ($x < -\alpha t$) and the right portion ($x > \beta t$), $\alpha, \beta > 0$ correspond to the $PW$ and $ED$ regions, respectively. See also a similar phenomenon for generalized BBM (gBBM)\cite{ElDika, ElDika2, ElDika_Martel, KM2018}. Here, $(-\frac18 t, t)$ and its complementary regions are the $PW$ and $ED$ regions, respectively.
\end{remark}

\begin{remark}
The study on the decay property (under suitable assumptions on the initial data) near the boundary region of $PW$ and $ED$ for gKdV ($x \sim 0$) and gBBM ($x \sim t$) has been recently studied in \cite{MP2018} and \cite{ACKM2018}, respectively. In the gBBM case, we have no information yet about decay inside the region $x \sim -\frac18 t$.
\end{remark}

\begin{remark}
In contrast with Remark \ref{rem:ED-PW}, in this paper there is in general no complete classification in terms of $PW$, $ED$ and the solitonic region (the region in space time which contains solitary waves). For instance, the solitonic region for gKdV is identical to the region $ED$. In contrast, for the case of gBBM with even nonlinearities, the intersection between the solitonic region and the $ED$+$PW$ region (in the left portion) is non-empty, see \cite{ACKM2018} for more details.
\end{remark}

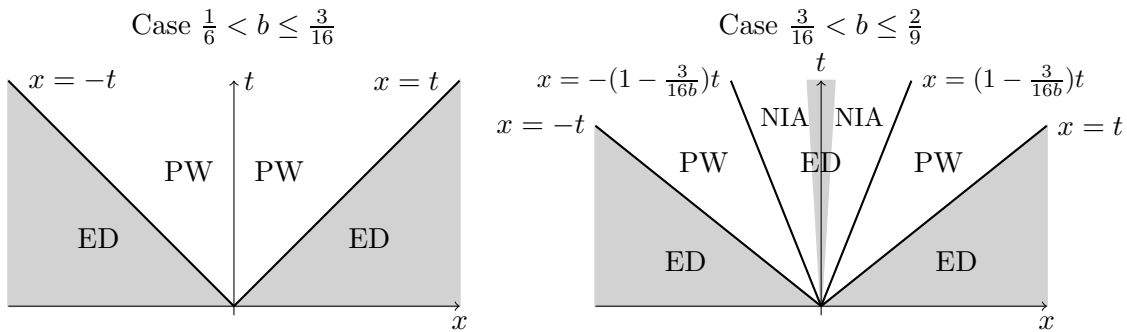
\begin{figure}[h!]
\begin{tikzpicture}[scale=0.6]
\node at (0,6.2){Case $\frac16 < b \le \frac{3}{16}$};
\filldraw[thick, color=lightgray!70] (0,0) -- (5,5) --(5,0) -- (0,0);
\filldraw[thick, color=lightgray!70] (0,0) -- (-5,5) --(-5,0) -- (0,0);
\draw[->] (-5,0) -- (5,0) node[below] {$x$};
\draw[->] (0,-0.2) -- (0,5) node[right] {$t$};
\draw[thick] (0,0) -- (5,5) node[left] {$x=t~$};
\draw[thick] (0,0) -- (-5,5) node[right] {~$x=-t$};
\node at (3,1.5){ED};
\node at (-3,1.5){ED};
\node at (-1,3){PW};
\node at (1,3){PW};
\end{tikzpicture}
\begin{tikzpicture}[scale=0.6]
\node at (0,6.2){Case $\frac{3}{16} < b \le \frac{2}{9}$};
\filldraw[thick, color=lightgray!70] (0,0) -- (0.3,5) --(-0.3,5) -- (0,0);
\filldraw[thick, color=lightgray!70] (0,0) -- (5,4) --(5,0) -- (0,0);
\filldraw[thick, color=lightgray!70] (0,0) -- (-5,4) --(-5,0) -- (0,0);
\draw[->] (-5,0) -- (5,0) node[below] {$x$};
\draw[->] (0,-0.2) -- (0,5) node[above] {$t$};
\draw[thick] (0,0) -- (5,4) node[right] {$x=t$};
\draw[thick] (0,0) -- (-5,4) node[left] {$x=-t$};
\draw[thick] (0,0) -- (2,5) node[right] {\small $x=(1-\frac{3}{16b})t$};
\draw[thick] (0,0) -- (-2,5) node[left] {\small $x=-(1-\frac{3}{16b})t$};
\node at (3,1){ED};
\node at (-3,1){ED};
\node at (2.6,3.2){PW};
\node at (-2.6,3.2){PW};
\node at (0.85,4.2){\small NIA};
\node at (-0.8,4.2){\small NIA};
\node at (0,3.2){ED};
\end{tikzpicture}
\caption{{\it Left.} When $\frac16 < b \le \frac{3}{16}$, decay holds true in the exterior region outside the light cone, while all linear plane waves (PW) stay in the interior of the light cone, see Theorem \ref{TH0}, item (1). {\it Right.} When $\frac{3}{16} < b \le \frac{2}{9}$, decay holds true not only in the exterior region outside the light cone, but also in a proper subset of the light cone around $x=0$ (Theorem \ref{Thm2} is valid), see Theorem \ref{TH0}, item (2). All linear plane waves have the group velocity whose rays span the region  $\{t\}\times \big((1-\frac{3}{16b})t, t\big)$. Here, NIA means the region where no information available.}\label{Fig:11a}
\end{figure}


\begin{figure}[h!]
\begin{tikzpicture}[scale=0.85]
\node at (0,6.6){Case $\frac{2}{9} < b \le \frac{3+\sqrt{3}}{12}$};
\filldraw[thick, color=lightgray!30] (0,0) -- (2,5) --(-2,5) -- (0,0);
\filldraw[thick, color=lightgray!70] (0,0) -- (5,3) --(5,0) -- (0,0);
\filldraw[thick, color=lightgray!70] (0,0) -- (-5,3) --(-5,0) -- (0,0);
\draw[->] (-5,0) -- (5,0) node[below] {$x$};
\draw[->] (0,-0.2) -- (0,5) node[above] {$t$};
\draw[thick] (0,0) -- (5,3) node[right] {$x=t$};
\draw[thick] (0,0) -- (-5,3) node[left] {$x=-t$};
\draw[thick] (0,0) -- (3.5,5) node[right] {\small$x=(1-\frac{3}{16b})t$};
\draw[thick] (0,0) -- (-3.5,5) node[left] {\small$x=-(1-\frac{3}{16b})t$};
\draw[thick] (0,0) -- (2,5) node[above] {\small$x=(1-\frac{2}{9b})t$};
\draw[thick] (0,0) -- (-2,5) node[above] {\small$x=-(1-\frac{2}{9b})t$};
\node at (3,1){ED};
\node at (-3,1){ED};
\node at (3.5,3.2){PW};
\node at (-3.5,3.2){PW};
\node at (2.1,3.9){NIA};
\node at (-2.1,3.9){NIA};
\node at (-0.5,3.2){ED};
\node at (0.5,3.2){ED};
\end{tikzpicture}
\begin{tikzpicture}[scale=0.85]
\node at (0,6.6){Case $b>\frac{3+\sqrt{3}}{12} $};
\filldraw[thick, color=lightgray!30] (0,0) -- (2,5) --(-2,5) -- (0,0);
\filldraw[thick, color=lightgray!70] (0,0) -- (5,1.5) --(5,0) -- (0,0);
\filldraw[thick, color=lightgray!70] (0,0) -- (-5,1.5) --(-5,0) -- (0,0);
\draw[->] (-5,0) -- (5,0) node[below] {$x$};
\draw[->] (0,-0.2) -- (0,5) node[above] {$t$};
\draw[thick] (0,0) -- (5,3) node[right] {$x=t$};
\draw[thick] (0,0) -- (-5,3) node[left] {$x=-t$};
\draw[thick] (0,0) -- (3.5,5) node[right] {\small$x=(1-\frac{3}{16b})t$};
\draw[thick] (0,0) -- (-3.5,5) node[left] {\small$x=-(1-\frac{3}{16b})t$};
\draw[thick] (0,0) -- (2,5) node[above] {\small$x=(1-\frac{2}{9b})t$};
\draw[thick] (0,0) -- (-2,5) node[above] {\small$x=-(1-\frac{2}{9b})t$};
\draw[thick] (0,0) -- (5,1.5) node[right] {$x=\sigma t$};
\draw[thick] (0,0) -- (-5,1.5) node[left] {$x=- \sigma t$};
\node at (3.5,0.5){ED};
\node at (-3.5,0.5){ED};
\node at (3.5,3.2){PW};
\node at (-3.5,3.2){PW};
\node at (2.1,3.9){NIA};
\node at (-2.1,3.9){NIA};
\node at (-0.5,3.2){ED};
\node at (0.5,3.2){ED};
\node at (3.5,1.5){NIA};
\node at (-3.5,1.5){NIA};
\end{tikzpicture}
\caption{{\it Above.} When $\frac{2}{9} < b \le \frac{3+\sqrt{3}}{12}$, decay holds true in a proper subset of the light cone not only around $x=0$, but also far away from $x=0$, i.e. around $x=vt$ for $-1+\frac{2}{9b} < v < 1-\frac{2}{9b}$ (Theorem \ref{TH1} remains valid in this regime). Also, energy decay in the whole exterior of the light cone $|x|\le |t|$ holds true. {\it Below.} When $  b>\frac{3+\sqrt{3}}{12}$, this last result seems not to hold and only decay in the ``smaller'' exterior region $(-\infty, -\sigma t)\cup(\sigma t, \infty)$ is available (Theorem \ref{TH3}, Item (3) is valid here). Here $\sigma = \frac{2(b-\frac16)(b-\frac18)}{b(b-\frac1{12})} >0 $, see \eqref{Sigma000}.  No information was obtained in the region inside $(-\sigma t, -t )\cup (\sigma t, t)$. As usual, NIA means no information available.}\label{Fig:11b}
\end{figure}
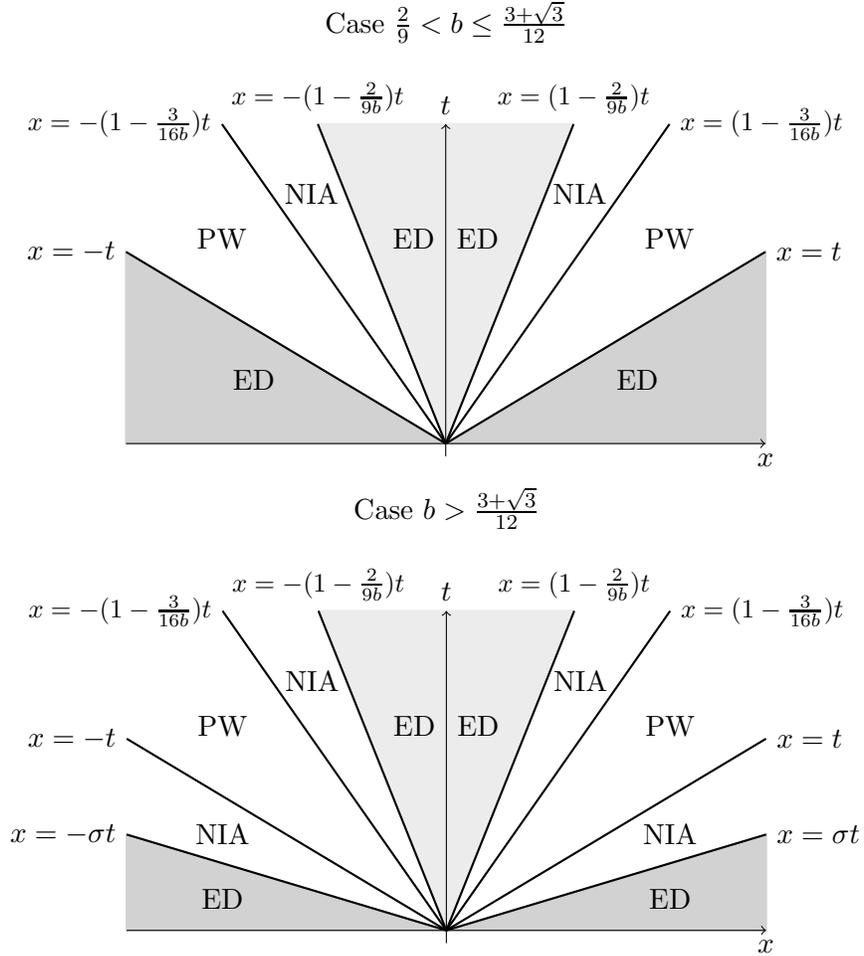


A first corollary from the above Theorem \ref{TH0} is the following \emph{two times the speed of light decay}:
\begin{corollary}\label{Corolario_0}
All globally defined small solutions to the Hamiltonian $abcd$ system \eqref{boussinesq}-\eqref{Conds} in the case $a=c=\frac16-b$, decay to zero in $(H^1\times H^1)(I(t))$ as $t\to\pm\infty$, and 
\[
I(t) := \big(-\infty, -(2+\epsilon)t \big) \cup \big((2+\delta)t, \infty\big),
\]
for arbitrary $\epsilon, \delta > 0$.  In particular, no matter the values of $(a,b,c)$, $a=c$, there are no small solitary waves of speed greater than 2.
\end{corollary}

This corollary follows  from Theorem \ref{TH0} and the increasing character of $\sigma_0(b)$ in terms of $b$, and $\lim_{b\to+\infty}\sigma_0(b)=2$.

%
%

\medskip

Theorem \ref{TH0} will be a consequence of the next results, which consider the more demanding Hamiltonian $abcd$ system in the general regime for $a\neq c$. In fact, the regime $a=c$ stated in Theorem \ref{TH0} will simply follow as a particular case of the general dynamics.

\medskip 

In order to state our next result, for simplicity of notation we need the following definition. {\color{black} First, for $b>0$, define
\be\label{tilde_ac}
\tilde a:= \frac{a}{b},\qquad \tilde c:= \frac{c}{b}.
\ee
Note that from \eqref{Conds} we have $-1\leq \tilde c<0$ and $-1 - \frac{1}{6b} \leq \tilde a <0$. Recall that, in the Hamiltonian setting, $b=d>0.$ Also, condition \eqref{dispersion_like} reads now
\be\label{dispersion_like_b}
3(\tilde a+ \tilde c) + 2  < 8\tilde a \tilde c. 
\ee

\begin{definition}[Refined dispersion-like parameters]\label{New Dis_Par}
We say that $(a,b,c)$ satisfying  \eqref{Conds} are \emph{refined dispersion-like parameters} if, in terms of $(\tilde a,\tilde c)$ introduced in \eqref{tilde_ac}, they are dispersion-like as in \eqref{dispersion_like_b}, or they satisfy the following conditions:
\ben
\item either
\begin{equation}\label{ref_dispersion_like2-1}
\begin{cases}\begin{array}{lcl} 
 45\tilde a \tilde c >1- \tilde a , &\;&\mbox{for} \quad  -1 \leq \tilde c < -\frac{1}{90}(19+\sqrt{181})\sim -0.36,\\
18\tilde a \tilde c +\tilde a+\tilde c>0, &\;&\mbox{for} \quad -\frac{1}{90} (19+\sqrt{181})\le \tilde c < -\frac{1}3,\\
 27\tilde a\tilde c>6\tilde a + 1 , &\;&\mbox{for} \quad -\frac{1}3 \le \tilde c< -\frac{1}9 \sim -0.11,
\end{array}\end{cases}
\end{equation}
whenever $\tilde c \le \tilde a < 0$,  
\item or, whenever $\tilde a \le \tilde c < 0$,
\begin{equation}\label{ref_dispersion_like2-2}
\begin{cases}\begin{array}{lcl} 
45\tilde a \tilde c>1- \tilde c , &\;&\mbox{for} \quad -1-\frac{1}{6b} \leq \tilde a < -\frac{1}{90} (19+\sqrt{181}),\\
18\tilde a\tilde c +\tilde a+ \tilde c>0, &\;&\mbox{for} \quad -\frac{1}{90} (19+\sqrt{181})\le \tilde a < -\frac{1}3,\\
 27\tilde a\tilde c>6\tilde c + 1 , &\;&\mbox{for} \quad -\frac{1}3 \le \tilde a< -\frac{1}9.
\end{array}\end{cases}
\end{equation}
\een
\end{definition}

\medskip

The previous definition may appear too cumbersome, but it is easily understood in terms of the framework of previous Figs. \ref{Fig:0} and \ref{Fig:5}. Indeed, see Figure \ref{Fig:6} for a simple description of Definition \ref{New Dis_Par} constructed on the previous Figs. \ref{Fig:0} and \ref{Fig:5}.

\medskip

Our first result shows decay to zero for all Hamiltonian $abcd$ models with \emph{refined dispersion-like parameters}.

\begin{theorem}[Decay for $abcd$ in the weakly dispersive regime]\label{TH1}
Let $(u,\eta)\in C(\R, H^1\times H^1)$ be a global, small solution of \eqref{boussinesq}-\eqref{Conds}, such that for some $\ve>0$ small
\[
\|(u_0,\eta_0)\|_{H^1\times H^1}< \ve.
\]
Assume additionally that $(a,b,c)$ are the \emph{refined dispersion-like parameters} as in Definition \ref{New Dis_Par}. Then, for $J_0(t)$ as \eqref{Jv}, there is strong decay:
\[
\lim_{t \to \pm\infty}   \|(u,\eta)(t)\|_{(H^1\times H^1)(J_0(t))} =0.
\]
\end{theorem}

\begin{figure}[h!]
\begin{center}
\begin{tikzpicture}[scale=0.9]
\filldraw[thick, color=lightgray!70] (3.5,3.5) arc (80:90:27)--(-1.2,-1.2)--(3.5,3.5);
\filldraw[thick, color=lightgray!70] (3.5,3.5) arc (10:0:27)--(-1.2,-1.2)--(3.5,3.5);
\filldraw[thick, color=lightgray!30] (4,1) arc (23:68:5.6)--(-1.2,4)--(-1.2,-1.2)--(4,-1.2);
\draw[thick,dashed] (2.8,5) -- (2.8,-1);
\draw[thick,dashed] (-1,2.8) -- (5,2.8);
\draw[thick,dashed] (3.5,5) -- (3.5,-1);
\draw[thick,dashed] (-1,3.5) -- (5,3.5);
\draw[->] (-1,4) -- (5.2,4) node[below] {$a$};
\draw[->] (4,-1) -- (4,5.3) node[right] {$c$};
\node at (3.5,4){$\bullet$};
\node at (3.1,4.4){$-\frac{b}9$};
\node at (2.8,4){$\bullet$};
\node at (2.4,4.4){$-\frac{b}4$};
\node at (0.9,4){$\bullet$};
\node at (0.9,4.4){$-\frac{2b}3$};
\node at (0,4){$\bullet$};
\node at (-0.5,4.4){$-b -\frac16$};
\node at (4,3.5){$\bullet$};
\node at (3.5,3.5){$\bullet$};
\node at (4.5,3.2){$-\frac{b}9$};
\node at (4,2.8){$\bullet$};
\node at (4.5,2.5){$-\frac{b}4$};
\node at (4,1){$\bullet$};
\node at (4.5,1){$-\frac{2b}3$};
\node at (4,0.5){$\bullet$};
\node at (4.5,0.5){$-b$};
\draw[thick,dashed] (0,4) -- (0,-1);
\draw[thick,dashed] (4,0.5) -- (-1,0.5);
\draw[thick,dashed] (4.4,-1) arc (0:90:5.3);
\draw[thick] (4.4,-1) arc (0:28:5.3);
\draw[thick] (-0.9,4.3) arc (90:64:5.3);
\draw[thick] (3.5,3.5) arc (80:84.5:27);
\draw[thick] (3.5,3.5) arc (10:6:27);
\draw[thick,dashed] (3.5,3.5) arc (80:90:27);
\draw[thick,dashed] (3.5,3.5) arc (10:0:27);
\node at (4.8,0){$\wt{\gamma}(b)$};
\end{tikzpicture}
\qquad
\qquad
\begin{tikzpicture}[scale=1]
\filldraw[thick, color=lightgray!10] (0,4.7)--(0,2) -- (4/3,1) --(4,3) -- (4,4.7) -- (0,4.7);
\filldraw[thick, color=lightgray!95] (0.7,1.6) arc (220:320:0.8)--(1.7,1.6) arc (270:230:0.6);
\filldraw[thick, color=lightgray!30] (0,4.7)--(0,2) -- (0.6,1.7) --(1,1.6) --(4/3,1.6) --(1.97,1.6) --(4,3) -- (4,4.7) -- (0,4.7);
\draw[thick,dashed] (4,-1) -- (4,3);
\draw[thick] (4,3) -- (4,5);
\draw[thick,dashed] (0,0) -- (4,3);
\draw[thick,dashed] (0,2)--(8/3,0);
\draw[thick] (0,2) -- (0,5);
\draw[->] (-1,0) -- (5,0) node[below] {$\nu$};
\draw[->] (0,-1) -- (0,5) node[right] {$b$};
\node at (0,0){$\bullet$};
\node at (4,0){$\bullet$};
\node at (4.3,-0.4){$1$};
\node at (4/3,0){$\bullet$};
\node at (8/3,-0.4){$\frac23$};
\node at (4/3,-0.4){$\frac13$};
\node at (8/3,0){$\bullet$};
\node at (0,2){$\bullet$};
\node at (-0.3,2.1){$\frac 13$};
\node at (0,1.3){$\bullet$};
\node at (-0.3,1.5){$\frac{3}{16}$};
\node at (0,1){$\bullet$};
\node at (-0.3,0.9){$\frac16$};
\node at (0,3){$\bullet$};
\node at (4,1.6){$\bullet$};
\node at (4.3,1.7){$\frac29$};
\node at (4,1.2){$\bullet$};
\node at (4.3,1.1){$\frac{2}{11}$};
\node at (-0.3,3){$\frac12$};
\node at (4.8,3){$b=\frac12\nu$};
\draw[thick,dashed] (0,1.3) -- (4/3,1.3);
\draw[thick,dashed] (4/3,0) -- (4/3,1.3);
\draw[thick,dashed] (0.7,1.6) arc (220:320:0.8);
\draw[thick,dashed] (0.5,1.7) arc (250:290:2.5);
\node at (2,3){$ \mathcal R$};
\end{tikzpicture}
\end{center}
\caption{(Left and right.) {\color{black} The regions dark shadowed represent the refined dispersion-like parameters, added in Definition \ref{New Dis_Par}, compared to \eqref{dispersion_like} (represented by light shadowed regions). (Left) The {\bf continuous} curve $\wt{\gamma}(b)$ represents the boundary of the regions described by \eqref{dispersion_like}, \eqref{ref_dispersion_like2-1} and \eqref{ref_dispersion_like2-2} for the validity of Theorem \ref{TH1}: below the intersection of this curve and $a, c < 0$, Definition \ref{New Dis_Par} holds true. Consequently, this extended parameter region for which there is decay is the main improvement in Theorem \ref{TH1}. (Right) Note that at $(\nu,b)=(\frac13,\frac{3}{16})$, one has $(a,c)=(-\frac{b}9,-\frac{b}9)=(-\frac{1}{48},-\frac{1}{48})$, which is the supremum value for (negative) $(a,c)$ for which Theorem \ref{TH1} remains valid. These values represent an endpoint case in Theorem \ref{TH1}, in a sense that for $(a,c)$ above, a strange regime appears; see Section \ref{sec:velocity}.} 
}\label{Fig:6}
\end{figure}
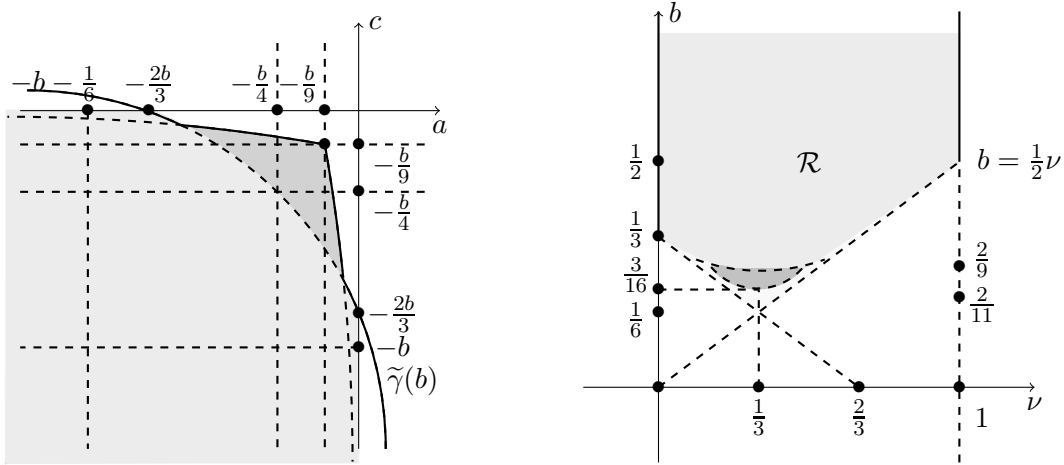

\begin{remark}
Theorem \ref{TH1} is sharp in the following sense: below $b=\frac3{16}$, i.e. above $a=c=-\frac{1}{48}$, the \emph{group velocity} $v=\omega'(k)$ of plane waves associated to the linear $abcd$ system, introduced in Definition \ref{def:PW}, is allowed to vanish at a nonzero wavenumber $k$, implying that scattering may not occur following a standard procedure. Additionally, the virial estimate used to prove Theorem \ref{TH1} fails to be true (Remark \ref{rem:optimal virial}). See Lemma \ref{lem:conclusion} for more details.
\end{remark}

The proof of Theorem \ref{TH1} is based in introducing a new refined virial estimate. The basic ingredient is the set of virial terms obtained in \cite{KMPP2018}, which have been essential to show decay for a first set of parameters  $(a,b,c)$ with sufficiently enough dispersion. In the weakly dispersive case $a,c\sim 0$, such estimates fail to hold and new ones are necessary.

\medskip

In order to overcome this problem, we need refined virial estimates. After some preliminary work, necessary to clearly see the main problem, we are led to analyze the positivity of  an $H^2$ functional with several negative coefficients, that appears as the key element to improve the virial estimate. Such an improvement is then showed to be sharp, see Lemma \ref{lem:conclusion}. Being relatively simple and elementary, our proof shows decay for the supercritical $2\times 2$ Hamiltonian $abcd$ system in all the possible ranges of decay for which the linear group velocity never vanishes.

\medskip

Our next result concerns the question of decay in \emph{exterior regions} of the light cone. 

\begin{theorem}[Decay in the exterior region outside the light cone]\label{TH3}
Let $(a,b,c)$ be parameters satisfying \eqref{Conds}. 
Let also $\epsilon, \delta > 0$ be arbitrary, small parameters. There exists $ \varepsilon_0 = \varepsilon_0(a,b,c,\epsilon, \delta)>0$ such that if $0<\ve<\ve_0$, the following are satisfied.

\smallskip

\begin{enumerate}
\item \emph{Complete decay regime.} Assume that $(a,b,c)$ satisfy
\begin{equation}\label{ellipse0}
153b^2-54b+4 \le 9ac.
\end{equation}
Define $I_{ext}(t) := \big(-\infty, -(1+\epsilon)t \big) \cup \big((1+\delta)t, \infty\big)$, and suppose $\norm{(u,\eta)(0)}_{H^1 \times H^1} < \varepsilon$. Then the $H^1 \times H^1$ global solution $(u,\eta)(t)$ to \eqref{boussinesq} with initial data $(u,\eta)(0)$  satisfies
\begin{equation}\label{Conclusion_Ext}
\begin{aligned}
\lim_{t \to \infty} \norm{(u,\eta)(t)}_{H^1 \times H^1 (I_{ext}(t))} = 0.
\end{aligned} 
\end{equation}
\item \emph{Decay in the $b$ large regime.}  If now $(a,b,c)$ satisfy
\begin{equation}\label{hyperbola0}
54ac \ge b \left(\sqrt{48(6b-1)^2 + (21b-2)^2} -1 \right),
\end{equation}
then the same conclusion as before holds in the interval (recall $\tilde a,\tilde c$ defined in \eqref{tilde_ac})
\begin{equation}\label{I_ext2}
I_{ext}(t) := \left(-\infty, - (3\sqrt{\tilde a \tilde c}+\epsilon )t \right) \cup \left( (3\sqrt{\tilde a \tilde c} +\delta )t, \infty \right).
\end{equation}
\item \emph{Improved decay in the $a=c$ regime.} Suppose $a=c =\frac16 -b<0$. Let
\begin{equation}\label{Sigma000}
\sigma =\sigma(b) :=
\frac{2(b-\frac16)(b-\frac18)}{b(b-\frac1{12})} >0.
\end{equation}
Then there is decay to zero as before inside the interval $\big(-\infty, -(1+\epsilon)t \big)  \cup \big((1+\delta)t, \infty \big)$ whenever $\frac16 < b \le \frac{1}{12}(3+\sqrt{3})$, and in the interval 
\begin{equation}\label{I_ext3}
I_{ext}(t) = \big(-\infty, -\left(\sigma+\epsilon \right)t \big) \cup \big(\left(\sigma+\delta \right)t, \infty \big),
\end{equation}
whenever $b>\frac{1}{12}(3+\sqrt{3}) $. 
\end{enumerate}
\end{theorem}

\medskip

\begin{remark}
The region in $a,b,c$ represented by the condition \eqref{ellipse0} can be understood, after a suitable change of variables, as the interior of an ellipse, in the sense that $b$ cannot be too large. See Remark \ref{ELLIPSE} for further details. Also, condition \eqref{hyperbola0} describes, after a precise change of variables, a hyperbola. See also Remark \ref{HYPERBOLA} and Proposition \ref{prop:exterior2} for the exact meaning of this point.
\end{remark}

A first corollary from the above Theorem \ref{TH3} is the following \emph{three times the speed of light decay}, valid for all possible $a\neq c$ (compare with Corollary \ref{Corolario_0}):

\begin{corollary}\label{Corolario_1}
All globally defined small solutions to the Hamiltonian $abcd$ system \eqref{boussinesq}-\eqref{Conds} decay to zero in $(H^1\times H^1)(I(t))$ as $t\to\pm\infty$, and 
\[
I(t) := \big(-\infty, -(3+\epsilon)t \big) \cup \big((3+\delta)t, \infty \big),
\]
for arbitrary $\epsilon, \delta > 0$.  Moreover, in the case $a=c=\frac16-b$, and from the asymptotic limit of $\sigma(b)$ in \eqref{Sigma000} as $b \to \infty$  Corollary \ref{Corolario_0} is recovered. In particular, no matter the values of $(a,b,c)$, there are no solitary waves of speed greater than 3, and if $a=c$, greater than 2.
\end{corollary}

\medskip

Our last result is an improvement in the interval of decay $J_0(t)$ in \eqref{Jv} by allowing it to move by a speed $v\neq 0$, that is to say, we try to encompass the whole light cone $(-|t|,|t|)$. Naturally the speed of light condition $|v|<1$ will appear, but first, we need another definition. Recall that $b>\frac16$ is a necessary condition for having a well-defined $abcd$ system, see \eqref{primera}.


\medskip

\begin{definition}[Uniform dispersion-like parameters]\label{Uni_Dis_Par}
Let $b_0 > \frac16$ be given. Consider the set $\mathcal R_0$ from \eqref{R0}. Define sets $\mathcal R_1(b_0)$, $\mathcal R_2(b_0)$ and $\mathcal R_3(b_0)$ of pairs $(\nu, b)$ by
\be\label{R1}
\mathcal R_1(b_0):=\Big\{(\nu, b) \in \mathcal R_0 ~:~  3 b_0\nu(3b_0\nu - 2) \le \Big(12(b-b_0)-1\Big) b \Big\},
\ee
\be\label{R2}
\begin{aligned}
\mathcal R_2(b_0):= &~\Big\{(\nu, b) \in \mathcal R_0 :  \\
&~ \qquad    4\Big(10b - 2(3b_0+1) - (9b_0-2)\nu \Big) b \ge 15b_0 \nu (3\nu-2)  \Big\},
\end{aligned}
\ee
and
\be\label{R3}
\begin{aligned}
\mathcal R_3(b_0):=&~\Big\{(\nu, b) \in \mathcal R_0 : \\ 
&~ \qquad 4\Big(30b -2(18b_0+1) + 3(9b_0-2)\nu\Big) b \ge 45b_0\nu (3\nu - 2)  \Big\}.
\end{aligned}
\ee
Then, we say that $(a,b,c)$ satisfying \eqref{Conds} are \emph{uniform dispersion-like parameters} if additionally, they belong to the following set:
\begin{equation}\label{eq:set}
\begin{aligned}
\mathcal R_\#(b_0):= &~ \Bigg\{ (a,b,c) =\left(-\frac{\nu}{2}- \frac13 - b, \; b, \; \frac{\nu}{2} - b \right) : \\
&~ \qquad  \qquad \qquad b \ge b_0, ~ \nu \in \mathcal R_1(b_0) \cap \mathcal R_2(b_0) \cap \mathcal R_3(b_0) \Bigg\}.
\end{aligned}
\end{equation}
\end{definition}
This set has been represented in Fig. \ref{fig:7}, please compare with Fig. \ref{Fig:0}. Precisely, we shall prove below that given $0 < v_0 < 1$, there is $b_0(v_0)>\frac16$ such that, if $(a,b,c)\in \mathcal R_\#(b_0)$, then small global solutions to \eqref{boussinesq} decay in $J_v(t)$ (see \eqref{Jv})  as $t \to \infty$, uniformly in $0< |v|<v_0$. By finding new virial estimates for nonzero speed solutions of $abcd$, we will prove the following last result:

\begin{figure}[h!]
\begin{center}
\begin{tikzpicture}[scale=1]
\filldraw[thick, color=lightgray!60] (1,2.35) arc (220:320:0.4) -- (1.6,2.35) -- (4,5) -- (0,5)--(0,4) --(1,2.35);
\draw[thick,dashed] (4,-1) -- (4,3);
\draw[thick] (4,3) -- (4,5);
\draw[thick,dashed] (0,0) -- (4,3);
\draw[thick,dashed] (0,2)--(8/3,0);
\draw[thick] (0,2) -- (0,5);
\draw[thick,dashed] (0,2.2)--(4/3,2.2);
\draw[thick] (1,2.35) arc (220:320:0.4);
\draw[thick] (1,2.35) -- (0,4);
\draw[thick] (1.6,2.35) -- (4,5);
\draw[->] (-1,0) -- (5,0) node[below] {$\nu$};
\draw[->] (0,-1) -- (0,5) node[right] {$b$};
\node at (0,0){$\bullet$};
\node at (4,0){$\bullet$};
\node at (4.3,-0.4){$1$};
\node at (4/3,0){$\bullet$};
\node at (8/3,-0.4){$\frac23$};
\node at (4/3,-0.4){$\frac13$};
\node at (8/3,0){$\bullet$};
\node at (0,2){$\bullet$};
\node at (-0.3,2){$\frac 13$};
\node at (0,1){$\bullet$};
\node at (-0.3,1){$\frac16$};
\node at (0,3){$\bullet$};
\node at (-0.3,3){$\frac12$};
\node at (0,2.2){$\bullet$};
\node at (-0.3,2.5){$b_0$};
\node at (4.8,3){$b=\frac12\nu$};
\node at (1.6,4.2){$ \mathcal{R}_{\#}(b_0)$};
\end{tikzpicture}
\end{center}
\caption{The slightly dark shadowed region describes the set $\mathcal R_\#(b_0)$ defined in Definition \ref{Uni_Dis_Par}. The set moves up if $b_0$ grows; in other words, it moves up as $v_0$ approaches $1$. Theorem \ref{TH2} says that given $0<v_0<1$, the strong decay in $J_v(t)$ is valid uniformly in $|v| \le v_0$, for all parameters $(a,b,c)$ represented by $\nu, b$ in $\mathcal R_\#(b_0)$. In particular, no small solitary waves of speed $v$ exist in this regime. 
} 
\label{fig:7}
\end{figure}

\begin{theorem}[Decay near the boundary of the light cone]\label{TH2}
Let $0 < v_0 < 1$ be given. Then, there exist $b_0=b_0(v_0)>0$ and $\varepsilon_0 = \varepsilon_0(v_0) > 0$ small such that the following holds true. Let $(u,\eta)\in C(\R, H^1\times H^1)$ be a global, small solution of \eqref{boussinesq}-\eqref{Conds}, such that for some $0 < \ve \le \varepsilon_0$,
\be\label{Smallness000}
\|(u_0,\eta_0)\|_{H^1\times H^1}< \ve.
\ee
Assume additionally that the triple $(a,b,c)$ belongs to the set $\mathcal R_\#(b_0)$ in \eqref{eq:set}.
Then, 
there is, uniform in $v$, strong decay along the ray $x\sim v|t|$: 
\begin{equation}\label{sup_lim}
\sup_{v\in [-v_0,v_0]} \lim_{t \to \pm\infty}   \|(u,\eta)(t)\|_{(H^1\times H^1)(J_v(t))} =0.
\end{equation}
\end{theorem}

\begin{remark}
The explicit dependence of $b_0$ in terms of $v_0$ is given in \eqref{b_0} and \eqref{k_0}; it roughly says that the closer $v_0$ to the speed of light 1, the larger $b_0$.
\end{remark}

\begin{remark}
Recall that Theorem \ref{TH2} states that {\bf no} solitary wave nor breather solution of $abcd$ with finite energy and with having asymptotic speed $v$, exist in this regime. Moreover, Theorem \ref{TH2} also discards the existence of any ``compact'' object satisfying the weaker condition $\limsup_{t \to \pm\infty}   \|(u,\eta)(t)\|_{(H^1\times H^1)(J_v(t))} >0$. Recall that solitary waves for $abcd$ exist in different regimes, see \cite{BCL0,CNS1,CNS2,Olivera} and references therein.
\end{remark}

\begin{remark}\label{rem:finite number}
Theorem \ref{TH2}  extends to decay in a union of finite number of intervals $J_{v_n}(t)$, $v_n \in [-v_0,v_0]$, $v_n\neq v_{n'}$ if $n\neq n'$. Indeed, for any $N \ge 1$, and $v_n \in [-v_0,v_0]$, $n=1,2,\cdots, N$, we have  
\[
\lim_{t \to \pm\infty}   \|(u,\eta)(t)\|_{(H^1\times H^1)(\cup_{n=1}^N J_{v_n}(t))} =0.
\]
The extension of this result to an arbitrary, infinite number of disjoint intervals is an interesting open question.
\end{remark}

\begin{remark}
We do not know whether or not the order between $\sup_v$ and $\lim_t$ in \eqref{sup_lim} can be switched. A positive answer to this question may provide strong information about solutions, in the sense that small, global solutions should decay to zero in any proper light cone inside $(-t,t)$.
\end{remark}

\begin{remark}[The case with surface tension] We believe that the previous results can be extended (with considerably more complicated proofs) to the case where \emph{there is surface tension $\tau$} in $abcd$, namely the case where the parameters satisfy $a+b+c+d=1/3-\tau$ and $\tau>0$. See \cite{DD} for more details on this particular regime.
\end{remark}

\medskip
%

\subsection*{Organization of this paper} This paper is organized as follows: in Section \ref{sec:pre} we introduce the preliminary objects needed for the proofs. Section \ref{sec:TH1} is devoted to the proof of Theorem \ref{TH1}. In Section \ref{sec:TH2} we prove Theorem \ref{TH2}, and finally, Section \ref{sec:TH3a} deals with the proof of Theorem \ref{TH3}.


\section{Preliminaries}\label{sec:pre}

The purpose of this Section is to gather together several well-known results needed for the proofs of Theorems \ref{TH1}, \ref{TH3} and \ref{TH2}.

\subsection{Stretching variables and dilation equivalence}
This section devotes to simplifying \eqref{boussinesq} for the convenience of computations. One finds that \eqref{boussinesq} allows a space-time dilation equivalence, that is to say, if $(u, \eta)$ are solutions to \eqref{boussinesq}, then $(u_b,\eta_b)$, defined by
\[u_b(t,x) = u(\sqrt{b}t, \sqrt{b}x) ,\quad \eta_b(t,x) = \eta(\sqrt{b} t, \sqrt{b} x),\]
are solutions to (see \eqref{tilde_ac})
\begin{equation}\label{boussinesq0000}
\begin{cases}
(1- \partial_x^2)\partial_t \eta_b  + \partial_x \! \left( \tilde a\, \partial_x^2 u_b +u_b + u_b \eta_b \right) =0, \quad (t,x)\in \R\times\R, \\ 
(1- \partial_x^2)\partial_t u_b  + \partial_x \! \left( \tilde c \, \partial_x^2 \eta_b + \eta_b  + \frac12 u_b^2 \right) =0. 
\end{cases}
\end{equation}

Let $b > \frac16$ be fixed. A computation (change of variables) gives
\[\sqrt{b} \int (1+|\xi|^2) |\wh{u}(\xi)|^2 \; d\xi \le \int (1+|\xi|^2)|\wh{u}_b(\xi)|^2 \; d\xi \le \frac{1}{\sqrt{b}} \int (1+|\xi|^2) |\wh{u}(\xi)|^2 \; d\xi,\]
for $\frac16 < b <1$, and
\[\frac{1}{\sqrt{b}} \int (1+|\xi|^2) |\wh{u}(\xi)|^2 \; d\xi \le \int (1+|\xi|^2)|\wh{u}_b(\xi)|^2 \; d\xi \le \sqrt{b} \int (1+|\xi|^2) |\wh{u}(\xi)|^2 \; d\xi,\]
for $1 \le b$, where $\wh{u}$ denotes the Fourier transform of $u$ with respect to the spatial variable. Thus, we conclude (since the same argument holds for $\eta$)
\begin{equation}\label{equiv1}
\norm{(u_b, \eta_b)}_{H^1\times H^1} \sim_b \norm{(u,\eta)}_{H^1\times H^1},
\end{equation}
which ensures the smallness assumption on $(u_b,\eta_b)$. The implicit constant in \eqref{equiv1} depends only on $b$.

\medskip

On the other hand, let $\psi$ be a positive even function, and let $\lambda(t)$ be a function corresponding to the interval $J_v(t)$ in \eqref{Jv}, for instance,
\[\lambda(t) = \frac{t}{\log^2 t}, \quad t > 2.\]
Then, an analogous argument also yields
\[\int \psi\left(\frac{x-vt}{\lambda(t)}\right) \left(u_b^2 + u_{b,x}^2\right) (t,x) \; dx \sim_b \int \psi\left(\frac{x - v\sqrt{b}t}{\sqrt{b}\lambda(t)}\right) \left(u^2 + u_x^2\right)(\sqrt{b}t,x) \; dx.\]
Note that we have
\[\log (\sqrt{b}t) \le \log t \le 2 \log (\sqrt{b}t), \quad t \ge \frac1b,\]
whenever $\frac16 < b < 1$, and
\[\frac12\log (\sqrt{b}t) \le \log t \le \log (\sqrt{b}t), \quad t \ge \sqrt{b},\]
whenever $1 < b$. This implies $\sqrt{b}\lambda(t) \sim_b \lambda(\sqrt{b}t)$, for $t \ge \max(b^{-1}, \sqrt{b})$, $\frac16 < b$. Thus, we conclude (since the same argument holds for $\eta$)
\begin{equation}\label{equiv2}
\norm{(u_b, \eta_b) (t)}_{(H^1 \times H^1)(J_v(t))} \sim_b \norm{(u, \eta) (t)}_{(H^1 \times H^1)(J_v(t))},
\end{equation}
which ensures the equivalence of $H^1$-decay between $(u,\eta)$ and $(u_b,\eta_b)$. The implicit constant in \eqref{equiv2} depends only on $b$.

\medskip

{\color{black}
Our analysis below will be given for $(u_b,\eta_b)$, not original solutions $(u, \eta)$. However, the above observation ensures that Theorems \ref{TH1}, \ref{TH3} and \ref{TH2} (and hence so Theorem \ref{TH0} and its corresponding corollaries) hold for $(u,\eta)$ without changing any condition, in particular, decay regions. 
}
%


\subsection{Virial identities} The following results are essentially contained in \cite{KMPP2018}, with some minor modifications that we will mention below.

\medskip

First, we recall some virial identities. Let $(u,\eta)=(u,\eta)(t,x)$ be an $H^1\times H^1$ global in time solution of the $abcd$ system \eqref{boussinesq0000}, and $\vp$ be a smooth, bounded function to be chosen later. Consider the two functionals 
\begin{equation}\label{I}
\mathcal I(t) := \int \vp(x)(u\eta + \px u \px\eta)(t,x)dx,
\end{equation}
and
\begin{equation}\label{J}
\mathcal J(t) := \int \vp'(x)(\eta \px u)(t,x)dx.
\end{equation}
There is a third functional $\mathcal K(t)$ defined in \cite{KMPP2018}, which will not be useful in this paper, because of some simplifications in the computations that we have found.

\medskip

We also need the general functional
\be\label{H}
\mathcal{H}(t): = \mathcal{H}_{\alpha}(t) := \mathcal{I}(t) + \alpha\mathcal{J}(t),
\ee
defined for some $\alpha\in\R$ to be found later on. Note that the Cauchy-Schwarz inequality yields the basic estimate
\begin{equation}\label{H_est}
|\mathcal{H}(t)| \lesssim \norm{u}_{H^1}^2 + \norm{\eta}_{H^1}^2.
\end{equation}
In what follows we state a description of the dynamics of $\mathcal I(t)$ and $\mathcal J(t)$ under the $abcd$ flow:

\begin{lemma}[Variation of $\mathcal{I}(t)$, \cite{KMPP2018}]\label{Virial_bous}
For any $t\in \R$,
\be\label{Virial0}
\begin{aligned}
\frac{d}{dt} \mathcal I(t) = &~  {}  -\frac{a}{2b} \int  \varphi' u_x^2-\frac{c}{2b} \int  \varphi' \eta_x^2 \\
& ~ {} - \left(\frac{a}{b}+\frac12\right)   \int  \varphi' u^2 -\left( \frac{c}{b}+ \frac12 \right)  \int  \varphi' \eta^2 \\
& ~ {} + \left(1+ \frac{a}{b}\right)\int \varphi' u (1-\partial_x^2)^{-1} u  + \left(1+ \frac{c}{b} \right)\int \varphi' \eta (1-\partial_x^2)^{-1} \eta  \\
&~ {}     -\frac12 \int  \varphi' u^2 \eta   + \int \varphi' u (1-\partial_x^2)^{-1}\left(u \eta \right)  + \frac12\int \varphi' \eta (1-\partial_x^2)^{-1}\left(u^2\right) .
\end{aligned}
\ee
\end{lemma}
We also have
\begin{lemma}[Variation of $\mathcal{J}(t)$, \cite{KMPP2018}]\label{lem:J}
For any $t \in \R$,
\begin{equation}\label{eq:J-1}
\begin{aligned}
\frac{d}{dt} \mathcal J(t)=&~   \left(1+ \frac{c}{b} \right)\int\vp'\eta^2 - \frac{c}{b}\int\vp' \eta_x^2  -\left(1+ \frac{a}{b}\right)\int\vp' u^2 + \frac{a}{b} \int\vp' u_x^2\\
&- \left(1+ \frac{c}{b} \right)\int\vp'\eta\nlop\eta + \left(1+ \frac{a}{b} \right)\int\vp' u \nlop u\\
&+ \left(1+ \frac{a}{b} \right) \int \vp''u \nlop u_x + \frac{c}{2b}\int\vp''' \eta^2\\
&-\frac12\int\vp' u^2\eta -\frac12\int\vp' \eta\nlop \left( u^2 \right)  \\
&+ \int \vp' u\nlop \left( u\eta \right) +\int \vp'' u\nlop ( u\eta )_x.
\end{aligned}
\end{equation}
\end{lemma}
Finally, the following result is obtained from \cite{KMPP2018} by setting $\beta = 0$ in that paper (or gathering \eqref{Virial0} and \eqref{eq:J-1}).
\begin{proposition}[Decomposition of $\frac{d}{dt}\mathcal H(t)$, \cite{KMPP2018}]\label{prop:general virial}
Let $u$ and $\eta$ satisfy \eqref{boussinesq0000}. For any $\alpha\in \R$ and any $t \in \R$, we have the decomposition
\begin{equation}\label{eq:gvirial}
\frac{d}{dt}\mathcal H(t) = \mathcal Q(t) + \mathcal{SQ}(t) + \mathcal{NQ}(t),
\end{equation}
where $\mathcal Q(t) =\mathcal Q[u,\eta](t) $ is the quadratic form
\begin{equation}\label{eq:leading}
\begin{aligned}
\mathcal Q(t) :=& ~\left(\left(1+ \frac{c}{b} \right)(\alpha-1) + \frac{1}2\right)\int \vp' \eta^2 +\frac{c}{b}\Big( - \alpha -\frac12\Big)\int \vp' \eta_x^2 \\
&+\left(\left(1+ \frac{a}{b} \right)( -\alpha-1) + \frac{1}2\right)\int \vp' u^2 +\frac{a}{b}\Big(\alpha-\frac12\Big)\int \vp' u_x^2 \\
&+\left(1+ \frac{c}{b} \right)( - \alpha + 1)\int \vp' \eta \nlop \eta\\
&+\left(1+ \frac{a}{b} \right)(\alpha  +1)\int \vp' u \nlop u ,
\end{aligned}
\end{equation}
$ \mathcal{SQ}(t)$ represents lower order quadratic terms not included in $\mathcal Q(t)$:
\begin{equation}\label{eq:small linear}
\begin{aligned}
 \mathcal{SQ}(t) :=& ~ \left(1+ \frac{c}{b} \right)\int\vp'' \eta \nlop \eta_x + \alpha\left(1+ \frac{a}{b} \right)\int\vp'' u \nlop u_x\\
&+\frac{\alpha c}{2b} \int \vp''' \eta^2 + \frac{ a}{2b} \int \vp'''  u^2,
\end{aligned}
\end{equation}
and  $\mathcal{NQ}(t)$ are truly cubic order terms or higher:
\begin{equation}\label{eq:nonlinear}
\begin{aligned}
 \mathcal{NQ}(t) :=& ~\frac12( - \alpha -1)\int\vp' u^2\eta + \frac12( -\alpha + 1)\int \vp' \eta \nlop (u^2)\\
&+ (\alpha +1)\int \vp'  u \nlop (u\eta) + \alpha\int \vp'' u \nlop (u\eta)_x.
\end{aligned}
\end{equation}
\end{proposition}

\begin{remark}\label{rem:phi'}
In view of the statements of Theorems \ref{TH1} and \ref{TH2}, a weight function $\vp$ in $\mathcal H(t)$ will be chosen as a time dependent function, for instance $\vp( \frac{x}{\lambda(t)} )$ for a function $\lambda(t)$. In the case, an additional term
\begin{equation}\label{eq:phi'}
\begin{aligned}
&-\frac{\lambda'(t)}{\lambda(t)} \int \frac{x}{\lambda(t)}\vp'\left( \frac{x}{\lambda(t)}\right) \left(u\eta + u_x \eta_x \right) \\
&-\alpha\frac{\lambda'(t)}{\lambda(t)^2} \int\left(\vp'\left( \frac{x}{\lambda(t)} \right) + \frac{x}{\lambda(t)}\vp''\left( \frac{x}{\lambda(t)} \right) \right) \eta \px u
\end{aligned}
\end{equation}
will be added in Proposition \ref{prop:general virial}. It follows from the time derivative of the weight function. See Proposition 5.2 in \cite{KMPP2018} for more details.
\end{remark}

\subsection{Passage to canonical variables}  We say that $f$ and $g$ are canonical variables for $u$ and $\eta$ respectively, if
\begin{equation}\label{eq:fg}
u = f-f_{xx}, \quad \hbox{ and } \quad \eta = g-g_{xx}.
\end{equation}
Note that these variables are always well-defined in $H^3$ if $u,\eta\in H^1$. See also \cite{ElDika,ElDika2,ElDika_Martel} for a first introduction and use of these variables. We recall the basic properties associated with these canonical variables.
\begin{lemma}[Equivalence of local $H^1$ norms, \cite{KMPP2018}]\label{lem:L2 comparable}
Let $f$ be as in \eqref{eq:fg}. Let $\phi$ be a smooth, bounded positive weight satisfying $|\phi''| \le \lambda \phi$ for some small but fixed $0 < \lambda \ll1$. Then, for any $a_1,a_2,a_3,a_4 > 0$, there  exist $c_1, C_1 >0$, depending on $a_j$ and $\lambda >0$, such that
\begin{equation}\label{eq:L2_est}
c_1  \int \phi \, (u^2 + u_x^2) \le \int \phi\left(a_1f^2+a_2f_x^2+a_3f_{xx}^2 +a_4 f_{xxx}^2 \right) \le C_1 \int \phi \, (u^2 + u_x^2).
\end{equation}
Moreover, we particularly have
\begin{equation}\label{eq:H1_est}
\int \phi \, (u^2 + u_x^2) = \int \phi\left(f^2+3f_x^2+3f_{xx}^2 + f_{xxx}^2 \right) - \int \phi'' \left(f^2 + f_x^2 \right).
\end{equation} 
The same consequences hold for $g$ corresponding to $\eta$.
\end{lemma} 

\begin{lemma}[Lemma 4.1 in \cite{KMPP2018}]\label{lem:Can_H1}
Let $f$ be as in \eqref{eq:fg} and $\phi$ be a smooth, bounded positive weight function. Then, one has 
\begin{equation}\label{eq:L2}
\int \phi u^2 = \int\phi\left(f^2 + 2f_x^2 + f_{xx}^2\right) - \int \phi''f^2,
\end{equation}
\begin{equation}\label{eq:H1}
\int \phi u_x^2 = \int\phi\left(f_x^2 + 2f_{xx}^2 + f_{xxx}^2\right) - \int \phi''f_x^2,
\end{equation}
and
\begin{equation}\label{eq:nonlocal}
\int \phi u \nlop u = \int\phi\left(f^2 + f_x^2\right) - \frac12\int \phi''f^2.
\end{equation}
\end{lemma}

Using these two representations \eqref{eq:fg}, one has
\begin{lemma}[Representation of $\mathcal{Q}(t)$ in canonical variables, \cite{KMPP2018}]\label{lem:leading}
The quadratic form $\mathcal Q(t)$ satisfies
\begin{equation}\label{eq:leading-1}
\begin{aligned}
\mathcal Q(t) =& \int \vp' \Big( A_1 f^2 + A_2 f_x^2 + A_3 f_{xx}^2 + A_4 f_{xxx}^2\Big)\\
&+\int \vp' \Big( B_1 g^2 + B_2 g_x^2 + B_3 g_{xx}^2 + B_4 g_{xxx}^2\Big)\\
&+\int \vp''' \Big(D_{11}f^2 + D_{12}f_x^2 + D_{21}g^2 + D_{22}g_x^2\Big),
\end{aligned}
\end{equation}
where
\begin{equation}\label{eq:A1B1}
A_1 = B_1 = \frac12>0,
\end{equation}

\begin{equation}\label{eq:A2B2}
A_2 = -\alpha -\frac{3a}{2b}, \qquad B_2 = \alpha - \frac{3c}{2b},
\end{equation}

\begin{equation}\label{eq:A3B3}
A_3 = -\left(1-\frac{a}{b}\right)\alpha -\frac{2a}{b} - \frac12, \qquad B_3 = \left(1-\frac{c}{b}\right)\alpha -\frac{2c}{b} - \frac12,
\end{equation}

\begin{equation}\label{eq:A4B4}
A_4= -\frac{a}{b}\left( \frac12 - \alpha\right), \qquad B_4= -\frac{c}{b}\left(\frac12 +  \alpha\right),
\end{equation}
and
\[
\begin{aligned}
&D_{11} =-\frac12\left(1+\frac{a}{b}\right)(-\alpha - 1) - \frac12 , \qquad D_{12}=-\frac{a}{b}\left(\alpha - \frac12\right),\\
&D_{21}= -\frac12\left(1+\frac{c}{b}\right)(\alpha - 1) - \frac12,\qquad D_{22}=-\frac{c}{b}\left(- \alpha - \frac12\right).
\end{aligned}
\]
\end{lemma} 

\begin{remark}
It is known \cite{KMPP2018} that there exists $\alpha \in \R$ such that $A_i, B_i > 0$, $i=1,2,3,4$, precisely when $a$, $b$ and $c$ satisfy \eqref{dispersion_like}. This is the key ingredient for the proof of Theorem \ref{Thm2}. In this paper, we will assume that some $A_i, B_i $ may take negative values, and prove decay even if there are negative terms in the quadratic form $\mathcal Q(t)$. 
\end{remark}

\bigskip
\section{Improved decay in the $abcd$ system: Proof of Theorem \ref{TH1}}\label{sec:TH1}
 
\subsection{Preliminaries} We start out with some simple results. The following definition are in some sense inspired in the sets $\mathcal B_j$, $j=1,2,3,4$, introduced in \cite{KMPP2018}.

\begin{lemma}\label{B5} Let $\mathcal B_5=\mathcal B_5(b)$ be the set of points $(a,b,c)$ such that \eqref{ref_dispersion_like2-1} and \eqref{ref_dispersion_like2-2} hold. 
Then, the following are satisfied:
\ben
\item In the case $a=c$, one has $b>\frac{3}{16}$, $a=c<-\frac1{48}$, and $\tilde a=\tilde c<-\frac1{9}$;
\smallskip
\item $\mathcal B_5(b)$ is symmetric with respect to the axis $a=c$, and invariant under the change $a\leftrightarrow c$;
\smallskip
\item $\mathcal B_5(b)$ is nonempty if $b> \frac{3}{16}$.
\een
\end{lemma}

\begin{proof} 
Putting $a=c= \frac16 -b$ into the last inequality in \eqref{ref_dispersion_like2-1}, valid for the largest possible $c$ (or \eqref{ref_dispersion_like2-2}), and solving $128b^2 - 40b + 3 > 0$ with $b > \frac16$, one has $b > \frac{3}{16}$, which  implies $a<-\frac1{48}$. It completes the proof of Item (1). Item (2) follows from the definition of $\mathcal B_5$. As a consequence of Items (1) and (2), we have Item (3). 
\end{proof}

In what follows, we will show that under $(a,b,c)\in \mathcal B_5(b)$, the virial functional $\mathcal H(t)$ introduced in \eqref{H} has coercive variation in time. In order to show this result, first we fix  
\[
\vp (t,x):= \tanh\left(\frac{x}{\lambda(t)}\right),
\]
and define 
\begin{equation}\label{eq:f'g'}
\wt{f}:= \vp'^{\frac12}f =  \sech\left(\frac{x}{\lambda(t)}\right) f, \qquad \wt{g}:= \vp'^{\frac12}g = \sech\left(\frac{x}{\lambda(t)}\right)g.
\end{equation}
\begin{lemma}
We have the representation
\begin{equation}\label{eq:leading-2}
\begin{aligned}
\mathcal Q(t) =& \int  A_1 \wt{f}^2 + A_2 \wt{f}_x^2 + A_3 \wt{f}_{xx}^2 + A_4 \wt{f}_{xxx}^2\\
&+\int  B_1 \wt{g}^2 + B_2 \wt{g}_x^2 + B_3 \wt{g}_{xx}^2 + B_4 \wt{g}_{xxx}^2\\
&+O\left(\frac{1}{\lambda(t)} \int \left(\wt{f}^2 + \wt{f}_x^2+ \wt{f}_{xx}^2 + \wt{f}_{xxx}^2 +  \wt{g}^2 + \wt{g}_x^2 + \wt{g}_{xx}^2 +\wt{g}_{xxx}^2\right)\right),
\end{aligned}
\end{equation}
where $A_i$ and $B_i$, $i=1,2,3,4$, are given in \eqref{eq:A1B1}-\eqref{eq:A4B4}.
\end{lemma}
\begin{proof}
The proof follows from a computation after putting \eqref{eq:f'g'} into \eqref{eq:leading-1}.
\end{proof}

Following the idea introduced in \cite{KM2018}, we will see that $\mathcal{Q}(t)$ can be rewritten as a nonnegative functional.  The following lemma is useful to verify the positivity of $\mathcal Q(t)$ in \eqref{eq:leading-2}.
\begin{lemma}\label{lem:aux}
Let 
\be\label{P0}
r_0 := \left(\frac{2\sqrt{571}}{3\sqrt{3}} - \frac{170}{27}\right)^{\frac13} - \frac{32}{9\left(\frac{2\sqrt{571}}{3\sqrt{3}} - \frac{170}{27}\right)^{\frac13}} + \frac43 \sim 0.27  >0.
\ee
Then, $r=0$ and $r=r_0$ are the unique real-valued roots of the quartic polynomial
\be\label{poly}
p(r):=\left(\frac{(1-r)^2+5}{2}\right)^2 +5r - 9.
\ee
Moreover, for any $0 < \epsilon < r_0$, we can find $\delta = \delta(\epsilon) >0$ such that the following equation holds
\begin{equation}\label{ep-de}
\left(\frac{(1-\epsilon)^2+5}{2}\right)^2 - 2\sqrt{9-\delta}(1-\epsilon) = 3+\epsilon.
\end{equation}
\end{lemma}

\begin{proof}
A computation gives
\[
p(r) = \frac{r}{4} \left(r^3-4r^2+16r-4 \right).
\]
Let denote the cubic polynomial by $q(r) =r^3-4r^2+16r-4$. It is not difficult to check that $q'(r) = 3(r-1)^2 + 13 >0$, and hence there is the unique non-trivial real root of $q(r)$, and $r_0$ as in \eqref{P0} is, indeed, the only real root of $q(r)$.\footnote{It is well-known to solve a cubic equation, so one can obtain \eqref{P0} via the well-known cubic formula or using a calculator.}

\medskip

On the other hand, a computation gives
\[p'(r) = r^3 - 3r^2 +8r -1, \quad p''(r) = 3(r-1)^2 + 5 >0,\]
which reveal that $p(r) < 0$ for $0 < r < r_0$. Note that $p(1) > 0$, and hence $r_0 < 1$. 

\medskip

For the second claim of Lemma \ref{lem:aux}, we fix $0 < \epsilon < r_0$. First, note that 
\[
p(r)=\left(\frac{(1-r)^2+5}{2}\right)^2 - 2\sqrt{9}(1-r) - (3+r).
\]
Then, a straightforward calculation exactly shows that
\begin{equation}\label{delta}
\delta = \delta(\epsilon) := 9 - \left(3 + \frac{p(\epsilon)}{2(1-\epsilon)} \right)^2
\end{equation}
solves the equation \eqref{ep-de}. $\delta >0$ is well-defined thanks to $p(\epsilon) < 0$ and $0 < \epsilon < r_0 < 1$.
\end{proof}

\begin{lemma}\label{lem:StrPos}
Let $r_0$ be as in \eqref{P0}. 
Then, for $0 < \epsilon < r_0$, there exists $\delta > 0$, only depending on $\epsilon$, and such that
\begin{equation}\label{StrPos}
\int  (9-\delta)\wt{f}^2 + (3+\epsilon)\wt{f}_x^2 -5\wt{f}_{xx}^2 + \wt{f}_{xxx}^2 \ge 0.
\end{equation}
\end{lemma}

\begin{proof}
Assume $w\in H^3(\R)$, and $\hat a,\hat b,\hat c,\hat d$ real valued. A straightforward calculation gives
\begin{equation}\label{eq:coefficients}
\begin{aligned}
&0\leq \int (\hat aw+\hat bw_x+\hat cw_{xx}+\hat dw_{xxx})^2 \\
& \qquad = \int \hat a^2 w^2 + (\hat b^2-2\hat a \hat c)w_x^2 + (\hat c^2-2\hat b \hat d)w_{xx}^2 + \hat d^2 w_{xxx}^2.
\end{aligned}
\end{equation}
For given $0 < \epsilon < r_0$, we can choose $\delta >0$ as in \eqref{delta} thanks to Lemma \ref{lem:aux}. Putting 
\[
\hat a = \sqrt{9-\delta}, \quad \hat b = \left(\frac{(1-\epsilon)^2+5}{2}\right), \quad \hat c= 1-\epsilon, \quad \hat d=1
\]
into \eqref{eq:coefficients}, one has
\[
\mbox{LHS of } \eqref{StrPos} = \int  \left( (\sqrt{9-\delta})\wt{f} + \left(\frac{(1-\epsilon)^2+5}{2}\right)\wt{f}_x + (1-\epsilon)\wt{f}_{xx} + \wt{f}_{xxx} \right)^2 \ge 0.
\]
The proof is complete.
\end{proof}

\subsection{First positivity conditions} Consider again the set $\mathcal B_5$ introduced in Lemma \ref{B5}. In a first step, we consider the last two conditions introduced for $\mathcal B_5$ in the case $c \le a < 0$, see \eqref{ref_dispersion_like2-1}. The condition $\tilde c \geq -\frac{1}{90}(19+\sqrt{181}) $ will be introduced later.


\begin{lemma}[Positivity from \eqref{ref_dispersion_like2-1}, first case]\label{lem:Pos1}
Let $c \le a <0$. When $a$ and $c$ satisfy
\begin{equation}\label{ref_dispersion_like2-1.1}
\begin{cases}\begin{array}{lcl} 
6ab + b^2 < 27ac, &\;&\mbox{for} \quad -\frac{b}{3} \le c< -\frac{b}9,\\
ab+bc+18ac>0, &\;&\mbox{for} \quad c < -\frac{b}3,
\end{array}\end{cases}
\end{equation}
for sufficiently large $t \gg 1$, $\mathcal Q(t)$ is positive definite, that is, 
\begin{equation}\label{Pos1.1}
\mathcal Q(t) \gtrsim \frac{1}{\lambda(t)}\int \sech^2\left( \frac{x}{\lambda(t)} \right) \left(u^2 + \eta^2 + u_x^2 + \eta_x^2 \right).
\end{equation}
\end{lemma}

\begin{proof}
In view of \eqref{eq:leading-2} with \eqref{eq:f'g'}, it suffices for \eqref{Pos1.1} to show 
\begin{equation}\label{Pos1.2}
\begin{aligned}
 \int  A_1 \wt{f}^2 + A_2 &\wt{f}_x^2 + A_3 \wt{f}_{xx}^2 + A_4 \wt{f}_{xxx}^2 +\int  B_1 \wt{g}^2 + B_2 \wt{g}_x^2 + B_3 \wt{g}_{xx}^2 + B_4 \wt{g}_{xxx}^2\\
 &\ge c_0 \int  \left(\wt{f}^2 + \wt{f}_x^2 + \wt{f}_{xx}^2 + \wt{f}_{xxx}^2 + \wt{g}^2 + \wt{g}_x^2 + \wt{g}_{xx}^2 + \wt{g}_{xxx}^2 \right),
\end{aligned}
\end{equation}
for some $c_0 > 0$, where $A_i$ and $B_i$, $i=1,2,3,4$, are given in \eqref{eq:A1B1}-\eqref{eq:A4B4}. The left-hand side of \eqref{Pos1.2} is rewritten as 
\begin{align}
 &\frac{1}{18} \left( \int  9\wt{f}^2 + 3\wt{f}_x^2 -5\wt{f}_{xx}^2 +\wt{f}_{xxx}^2 +\int  9 \wt{g}^2 + 3\wt{g}_x^2 - 5 \wt{g}_{xx}^2 + \wt{g}_{xxx}^2 \right) \label{Pos1.3}\\
 &+\int  A_2' \wt{f}_x^2 + A_3' \wt{f}_{xx}^2 + A_4' \wt{f}_{xxx}^2 +\int  B_2' \wt{g}_x^2 + B_3' \wt{g}_{xx}^2 + B_4' \wt{g}_{xxx}^2, \nonumber
\end{align}
where
\begin{equation}\label{eq:A2B2-1}
A_2' = -\alpha -\frac{3a}{2b} -\frac16, \qquad B_2' = \alpha - \frac{3c}{2b} -\frac16,
\end{equation}

\begin{equation}\label{eq:A3B3-1}
A_3' = -\left(1-\frac{a}{b}\right)\alpha -\frac{2a}{b} - \frac29, \qquad B_3' = \left(1-\frac{c}{b}\right)\alpha -\frac{2c}{b} - \frac29,
\end{equation}
and
\begin{equation}\label{eq:A4B4-1}
A_4' = - \frac{a}{b}\left(\frac12 - \alpha\right) - \frac{1}{18}, \qquad B_4' = -\frac{c}{b}\left(\frac12 + \alpha\right) -\frac{1}{18}.
\end{equation}
If $A_i', B_i' > 0$, $i=2,3,4$, applying Lemma \ref{lem:StrPos} to \eqref{Pos1.3}, one proves \eqref{Pos1.2}. 

\medskip

We remark that $A_i', B_i' > 0$, $i=2,3,4$ holds true by taking $\alpha = 0$, when $a < -\frac{b}9$. Moreover, there is no $\alpha \in \R$ such that $A_i', B_i' > 0$, $i=2,3,4$ holds true, when $c \ge -\frac{b}9$. In what follows, we assume $c < - \frac{b}9 \le a$.

\medskip

 A computation gives
\[
\begin{aligned}
\alpha < & ~{} \min \left(-\frac{9a+b}{6b}, -\frac{18a+2b}{9(b-a)}, \frac12 + \frac{b}{18a} \right) \\
= & ~ {} \frac12 + \frac{b}{18a}, \quad -\frac{b}9 \le a < 0 ,
\end{aligned}
\]
for $A_i$, and 
\[
\begin{aligned}
\alpha >&~{} \max \left(\frac{9c+b}{6b}, \frac{18c+2b}{9(b-c)}, -\frac12 - \frac{b}{18c} \right) \\
= & ~ {} \begin{cases}\begin{array}{lcl}\-\frac{9c+b}{6b}, &\;&\mbox{if} \quad -\frac{b}3 \le c < -\frac{b}9 ,\\ -\frac12 - \frac{b}{18c}, &\;&\mbox{if} \quad c < -\frac{b}3,\end{array}\end{cases} 
\end{aligned}
\]
for $B_i$. Hence, for $c \le a < 0$, there exists $\alpha \in \R$ such that $A_i', B_i' > 0$, $i=2,3,4$, when $a$ and $c$ satisfy
\[
\frac{9c+b}{6b} < \frac12 + \frac{b}{18a} \quad \mbox{if} \quad -\frac{b}3 \le c < -\frac{b}9,
\]
and
\[-\frac12 - \frac{b}{18c}< \frac12 + \frac{b}{18a} \quad \mbox{if} \quad c < -\frac{b}3,\]
which are equivalent to \eqref{ref_dispersion_like2-1.1}.
\end{proof}

Due to the symmetry of $a$ and $c$ in \eqref{eq:A2B2}--\eqref{eq:A4B4}, taking $\alpha \mapsto -\alpha$ carries 

\begin{lemma}[Positivity from \eqref{ref_dispersion_like2-2}, first case]\label{lem:Pos2}
Let $a \le c <0$. When $a$ and $c$ satisfy
\begin{equation}\label{ref_dispersion_like2-2.1}
\begin{cases}\begin{array}{lcl} 
6bc + b^2 < 27ac, &\;&\mbox{for} \quad -\frac{b}3 \le a < -\frac{b}9,\\
ab+bc+18ac>0, &\;&\mbox{for} \quad a < -\frac{b}3,
\end{array}\end{cases}
\end{equation}
for sufficiently large $t \gg 1$, $\mathcal Q(t)$ is positive definite, that is, 
\[\mathcal Q(t) \gtrsim \frac{1}{\lambda(t)}\int \sech^2\left( \frac{x}{\lambda(t)} \right) \left(u^2 + \eta^2 + u_x^2 + \eta_x^2 \right).\]
\end{lemma}

\begin{remark}\label{rem:optimal virial}
Note  that the case when $(a,c) = (-\frac{b}9,-\frac{b}9)$ cannot be covered via our method. Precisely, one has
\[\begin{aligned}
\mathcal{Q}(t) \sim&~{} \frac{1}{18} \int \varphi' \left(9(f^2+g^2) + 3(f_x^2+g_x^2) - 5(f_{xx}^2+ g_{xx}^2) + (f_{xxx}^2+g_{xxx}^2) \right)\\
=&~{} \frac{1}{18} \int \varphi' \left(3f + 3f_x + f_{xx} + f_{xxx} \right)^2 + \left(3g + 3g_x + g_{xx} + g_{xxx} \right)^2 \ge 0,
\end{aligned}\]
when $(a,c) = (-\frac19,-\frac19)$. Such critical coefficients form a perfect square formula which do not allow \eqref{Pos1.1}.
\end{remark}

\subsection{Second positivity conditions} Consider now $\mathcal B_5$ introduced in Lemma \ref{B5}, in the case of the first condition in the case $c \le a < 0$, see \eqref{ref_dispersion_like2-1}. Now we have the following result.

\begin{lemma}[Positivity from \eqref{ref_dispersion_like2-1}, final case]\label{lem:Pos3}
Let $c \le a <0$ . When $a$ and $c$ satisfy
\begin{equation}\label{ref_dispersion_like2-1.2}
\begin{cases}\begin{array}{lcl} 
11ab+9bc+4b^2 <9ac, &\;&\mbox{for} \quad -\frac{b}3 \le c < \frac{(10-2\sqrt{34})b}{9},\\
b^2-ab < 45ac, &\;&\mbox{for} \quad c < -\frac{b}3,
\end{array}\end{cases}
\end{equation}
for sufficiently large $t \gg 1$, $\mathcal Q(t)$ is positive, that is, 
\[\mathcal Q(t) \gtrsim \frac{1}{\lambda(t)}\int \sech^2\left( \frac{x}{\lambda(t)} \right) \left(u^2 + \eta^2 + u_x^2 + \eta_x^2 \right).\]
\end{lemma}

\begin{proof}
Similarly as the proof of Lemma \ref{lem:Pos1}, it suffices to show \eqref{Pos1.2}. Note that $A_i, B_i >0$, $i=1,2,3,4$,  in \eqref{eq:A1B1} -- \eqref{eq:A4B4} holds true by taking $\alpha = 0$, when $c \le a < -\frac{b}4$. In what follows, we assume $a \ge -\frac{b}4$ (also $c < -\frac{b}9$ by the same reason as in the proof of Lemma \ref{lem:Pos1}). It is known that $A_i >0$, $i=1,2,3,4$, holds true when
\[
\alpha < \min\left(-\frac{3a}{2b}, -\frac{b+4a}{2(b-a)}, \frac12 \right) = -\frac{b+4a}{2(b-a)}, \quad \mbox{if} \quad a \ge -\frac{b}4, 
\]
and hence, we obtain $\wt{f}$-portions in \eqref{Pos1.2}. On the other hand, the $\wt{g}$-portions in the left-hand side of \eqref{Pos1.2} is rewritten as 
\begin{equation}
 \frac{1}{18} \int \left( 9 \wt{g}^2 + 3\wt{g}_x^2 - 5 \wt{g}_{xx}^2 + \wt{g}_{xxx}^2 \right)  +\int  \left( B_2' \wt{g}_x^2 + B_3' \wt{g}_{xx}^2 + B_4' \wt{g}_{xxx}^2 \right),
\end{equation}
where $B_i'$, $i=2,3,4$, is given in \eqref{eq:A2B2-1} -- \eqref{eq:A4B4-1}. In order to show $B_i' > 0$, $i=2,3,4$, $\alpha$ should satisfy
\[
\begin{aligned}
\alpha > &~\max \left(\frac{9c+b}{6b}, \frac{18c+2b}{9(b-c)}, -\frac12 - \frac{b}{18c} \right)\\
 =& ~ \begin{cases}\begin{array}{lcl}\-\frac{9c+b}{6b}, &\;&\mbox{if} \quad -\frac{b}3 \le c < -\frac{b}9 ,\\ -\frac12 - \frac{b}{18c}, &\;&\mbox{if} \quad c < -\frac{b}3.\end{array}\end{cases} 
\end{aligned} 
\] 
Hence, for $c \le a < 0$, there exists $\alpha \in \R$ such that $A_i, B_i' > 0$, $i=2,3,4$, when $a$ and $c$ satisfy
\[\frac{9c+b}{6b} < -\frac{b+4a}{2(b-a)} \quad \mbox{if} \quad -\frac{b}3 \le c < \frac{(10-2\sqrt{34})b}{9}\]
and
\[-\frac12 - \frac{b}{18c}< -\frac{b+4a}{2(b-a)} \quad \mbox{if} \quad c < -\frac{b}3,\]
which are equivalent to \eqref{ref_dispersion_like2-1.2}.
\end{proof}

\begin{remark}
An opposite way to the proof of Lemma \ref{lem:Pos3} provides a different parameter conditions. Indeed, under the restriction of $\alpha$ satisfying $B_i >0$, applying Lemma \ref{lem:StrPos} to $\wt{f}$-portions, one has
\begin{equation}\label{comparison}
\begin{cases}\begin{array}{lcl} 
9ab + 11bc+4b^2 <9ac, &\;&\mbox{for} \quad -\frac{b}4 \le c < \frac{(10-2\sqrt{34})b}{9},\\
b^2-bc < 45ac, &\;&\mbox{for} \quad c < -\frac{b}4.
\end{array}\end{cases}
\end{equation}
However, the new parameter condition \eqref{comparison} is useless in our result in the sense that, for each $b$, the region of $(a,c)$ satisfying \eqref{comparison} is covered by the region \eqref{ref_dispersion_like2-1.1}, even though it is wider than one described by \eqref{ref_dispersion_like2-1.2} up to $-\frac{(\sqrt{1865}+80)b}{414} \le c$. Thus, Lemma \ref{lem:Pos3} in addition to Lemma \ref{lem:Pos1} guarantees a better result.
%
\end{remark}

Similarly as Lemma \ref{lem:Pos2}, we have
\begin{lemma}[Positivity from \eqref{ref_dispersion_like2-2}, final case]\label{lem:Pos4}
Let $a \le c <0$ . When $a$ and $c$ satisfy
\begin{equation}\label{ref_dispersion_like2-2.2}
\begin{cases}\begin{array}{lcl} 
11bc+9ab+4b^2 <9ac, &\;&\mbox{for} \quad -\frac{b}3 \le a < \frac{(10-2\sqrt{34})b}{9},\\
b^2-bc < 45ac, &\;&\mbox{for} \quad a < -\frac{b}3,
\end{array}\end{cases}
\end{equation}
for sufficiently large $t \gg 1$, $\mathcal Q(t)$ is positive definite, that is, 
\[\mathcal Q(t) \gtrsim \frac{1}{\lambda(t)}\int \sech^2\left( \frac{x}{\lambda(t)} \right) \left(u^2 + \eta^2 + u_x^2 + \eta_x^2 \right).\]
\end{lemma}

\begin{remark}\label{rem:Pos}
A direct comparison between \eqref{ref_dispersion_like2-1.1} (resp. \eqref{ref_dispersion_like2-2.1}) and \eqref{ref_dispersion_like2-1.2} (resp. \eqref{ref_dispersion_like2-2.2}) reveals that the result obtained in Lemma \ref{lem:Pos3} (resp. \ref{lem:Pos4}) covers wider regions of parameters than one in Lemma \ref{lem:Pos1} (resp. \ref{lem:Pos2}), when $c\le a < 0$ (resp. $a\le c < 0$) and $c < -\frac{(19+\sqrt{181})b}{90}$ (resp. $a < -\frac{(19+\sqrt{181})b}{90}$). More precisely, a computation in \eqref{ref_dispersion_like2-1.1} and \eqref{ref_dispersion_like2-1.2} for $c \ge -\frac{b}{3}$ gives
\[
a < -\frac{b^2}{6b-27c}, \quad \mbox{and} \quad a < -\frac{4b^2+9bc}{11b+9c},
\]
respectively. Moreover, one could obtain, under the fundamental assumption on $b$ and $c$, that 
\[
\begin{aligned}
-\frac{4b^2+9bc}{11b+9c} < -\frac{b^2}{6b-27c}\quad  \Longleftrightarrow \quad  &~{}  243c^2 + 45bc -24b^2 <0\\
  \Longleftrightarrow \quad &~{}  -\frac{(5+\sqrt{313})b}{54} < c.
  \end{aligned}
\]
Since $-\frac{(5+\sqrt{313})b}{54} < -\frac{b}{3}$, we conclude that the region described by \eqref{ref_dispersion_like2-1.1} contains the region by \eqref{ref_dispersion_like2-1.2} for $c \ge -\frac{b}{3}$. Otherwise, an analogous argument yields that the regions described by \eqref{ref_dispersion_like2-1.1} and \eqref{ref_dispersion_like2-1.2}, for $c < -\frac{b}{3}$, are expressed as
\[a < -\frac{bc}{b+18c}, \quad \mbox{and} \quad a < \frac{b^2}{b+45c},\]
respectively. A  computation gives
\[
-\frac{bc}{b+18c} = \frac{b^2}{b+45c} \quad  \Longleftrightarrow \quad  c = \frac{(-19\pm\sqrt{181})b}{90},
\]
which, in addition to the fundamental assumption on $b$ and $c$, implies
\[-\frac{bc}{b+18c} \ge \frac{b^2}{b+45c}, \quad \mbox{if} \quad -\frac{(19+\sqrt{181})b}{90} \le c < -\frac{b}{3}\]
and
\[-\frac{bc}{b+18c} < \frac{b^2}{b+45c}, \quad \mbox{if} \quad c <-\frac{(19+\sqrt{181})b}{90}.\]
The condition $c \geq -\frac{(19+\sqrt{181})b}{90} $ appears in this observation and this proves our claim. The same argument is available to the comparison between \eqref{ref_dispersion_like2-2.1} and \eqref{ref_dispersion_like2-2.2}.
\end{remark}

\subsection{Refined dispersion-like parameters} Gathering Lemmas \ref{lem:Pos1} and \ref{lem:Pos4}, in addition to Remark \ref{rem:Pos}, we extend the results proved in Lemma 4.5 in \cite{KMPP2018} as follows:
\begin{lemma}[Refined positivity of the quadratic form $\mathcal{Q}(t)$]\label{lem:Q}
Let $a, c <0$ satisfy \eqref{Conds}. Then, we have 
\be\label{Q_new_new}
\mathcal Q(t) \gtrsim \frac{1}{\lambda(t)}\int \sech^2\left( \frac{x}{\lambda(t)} \right) \left(u^2 + \eta^2 + u_x^2 + \eta_x^2 \right),
\ee
whenever $(a,c)$ are the \emph{refined dispersion-like parameter} defined in Definition \ref{New Dis_Par}.
\end{lemma}

\begin{remark}
Lemma \ref{lem:Q} is valid under Definition \ref{New Dis_Par}, and it is the main new ingredient for the proof of Theorem \ref{TH1}. Indeed, once we have the validity of \eqref{Q_new_new}, the rest of the proof will follow with standard energy methods.
\end{remark}

\subsection{End of Proof of Theorem \ref{TH1}} The end of the proof is similar to the one in \cite{KMPP2018}. For the sake of reader's convenience, we give, here, a sketch of the proof.

\medskip

Let $\la=\lambda(t)$ be the time-dependent function given by
\begin{equation}\label{eq:lambda000}
\lambda(t) := \frac{t}{\log^2t}, \quad t\geq 2.
\end{equation}
Note that
\begin{equation}\label{eq:lambda0001}
\lambda'(t) = \frac{1}{\log^2 t}\left(1-\frac{2}{\log t} \right) \quad \mbox{and} \quad \frac{\lambda'(t)}{\lambda(t)} = \frac{1}{t}\left(1-\frac{2}{\log t} \right).
\end{equation}

\begin{proposition}[Weak decay property]\label{prop:TH1}
Let $(a,b,c)$ be \emph{refined dispersion-like} parameters defined as in Definition \ref{New Dis_Par}. Let $(u,\eta)(t)$ be $H^1 \times H^1$ global solutions to \eqref{boussinesq0000} such that for some $\ve>0$ small
\[
\|(u_0,\eta_0)\|_{H^1\times H^1}< \ve.
\]
Then, we have
\begin{equation}\label{eq:TH1-1}
\int_2^{\infty}\frac{1}{\lambda(t)}\int \sech^2 \left(\frac{x}{\lambda(t)}\right) \left(u^2 + (\px u)^2 + \eta^2 + (\px \eta)^2\right)(t,x) \, dx \,dt\lesssim \ve^2.
\end{equation}
As an immediate consequence, there exists an increasing sequence of time $\{t_n\}$ $(t_n \to \infty$ as $n \to \infty)$ such that
\begin{equation}\label{eq:TH1-2}
\int \sech^2 \left(\frac{x}{\lambda(t_n)}\right) \left(u^2 + (\px u)^2 + \eta^2 + (\px \eta)^2\right)(t_n,x) \; dx \longrightarrow 0 \mbox{ as } n \to \infty.
\end{equation}
\end{proposition}

\begin{proof}
We choose, in Proposition \ref{prop:general virial}, the weight $\vp = \tanh$ with $\la(t)$ given by \eqref{eq:lambda000}. Analogous arguments in the proof of Propositions 5.1 and 5.2 (see also Remark \ref{rem:phi'}) in \cite{KMPP2018} yield
\[
\eqref{eq:phi'} \lesssim \frac{1}{\lambda(t)} \int \sech^2\Big( \frac{x}{\lambda(t)}\Big) \left(u^2 + (\px u)^2 + \eta^2 + (\px \eta)^2\right)+\frac{\ve^2}{t \log ^2 t},
\]
\[
|\mathcal{SQ}(t)| \lesssim \frac{1}{\lambda(t)^2}\int \vp' \left( \frac{x-vt}{\lambda(t)} \right)\left(u^2 + \eta^2 + u_x^2 + \eta_x^2 \right),
\]
and
\[
|\mathcal{NQ}(t)| \lesssim \frac{\norm{u}_{H^1} + \norm{\eta}_{H^1}}{\lambda(t)}\int \vp' \left( \frac{x-vt}{\lambda(t)} \right)\left(u^2 + \eta^2 + u_x^2 + \eta_x^2 \right).
\]

Together with all above, Lemma \ref{lem:Q} with the smallness assumption implies \eqref{eq:TH1-1} thanks to the fact that $\frac{1}{t\log^pt}$ is integrable on $[2,\infty)$ when $p>1$. The limit argument ensures that \eqref{eq:TH1-1} implies \eqref{eq:TH1-2}.
\end{proof}

Proposition \ref{prop:TH1} in addition to local energy estimate (see Section 6 in \cite{KMPP2018}) guarantees the strong decay property. The proof follows from the proof of Propositions 7.1 and 7.2 in \cite{KMPP2018}, thus we omit the detail (see also the proof of Proposition \ref{prop:energy2} below). We end this section with the statement of the strong decay property.
\begin{proposition}[Strong decay property]\label{prop:1}
Let $(a,b,c)$ be \emph{refined dispersion-like} parameters defined as in Definition \ref{New Dis_Par}. Let $(u,\eta)(t)$ be $H^1 \times H^1$ global solutions to \eqref{boussinesq0000} such that for some $\ve>0$ small
\[
\|(u_0,\eta_0)\|_{H^1\times H^1}< \ve.
\]
Then, we have
\begin{equation}\label{eq:TH1-3}
\lim_{t \to \infty} \int \sech^4 \left(\frac{x}{\lambda(t)}\right) \left(u^2 + (\px u)^2 + \eta^2 + (\px \eta)^2\right)(t,x) \; dx = 0.
\end{equation}
\end{proposition}

\begin{remark}
This last result ends the proof of Theorem \ref{TH1}.
\end{remark}

\subsection{Sharpness in Theorem \ref{TH1}}\label{sec:velocity}
The aim of this section devotes to justifying the sharpness in Theorem \ref{TH1} in the sense that the group velocity of linear waves associated to \eqref{boussinesq} never vanish at any (nonzero) wavenumber, in particular for $a=c$ case.

\medskip

Recall \eqref{dispersion relation_0} and \eqref{group velocity_0} (see also \cite{KMPP2018}), the formula of the dispersion relation and the group velocity of linear waves to \eqref{boussinesq} as 
\begin{equation}\label{dispersion relation}
w(k) = \frac{\pm |k| }{(1+ bk^2)}(1-ak^2)^{1/2}(1-ck^2)^{1/2}
\end{equation}
and
\begin{equation}\label{group velocity}
|w'(k)|= \frac{|abck^6 +3ack^4 -(b+2a+2c)k^2 +1|}{(1+bk^2)^2 (1-ak^2)^{1/2}(1-ck^2)^{1/2}},
\end{equation}
respectively. In what follows, we deal with the quantity
\begin{equation}\label{group velocity1}
\frac{abck^6 +3ack^4 -(b+2a+2c)k^2 +1}{(1+bk^2)^2 (1-ak^2)^{1/2}(1-ck^2)^{1/2}}.
\end{equation}

\medskip

Let us denote \eqref{group velocity1} as $P(bk^2)$, and put $a = c = \frac16 -b$. Let $\widetilde b = \frac{1}{6b} - 1$. Note that $\widetilde b$ is well-defined and $-1 < \widetilde b < 0$, since $b > \frac16$. Then, a simple calculation gives
\begin{equation}\label{eq:P}
P(\mu) = \frac{(1-\widetilde b \mu)(1+(-1-3\widetilde b)\mu - \widetilde b \mu^2)}{(1+\mu)^2 (1-\widetilde b\mu)} = \frac{1+(-1-3\widetilde b)\mu - \widetilde b \mu^2}{(1+\mu)^2},
\end{equation}
where $\mu = bk^2$.

\medskip

\begin{lemma}\label{lem:p>0}
Fix $b > \frac16$. Let $P(\mu)$ be the function given in the RHS of \eqref{eq:P}. Then, $P(\mu)$ never vanishes  for $\mu \ge 0$, provided $b > \frac{3}{16}$.
\end{lemma}

\begin{proof}
We have
\begin{equation}\label{p'(mu)}
P'(\mu) = \frac{(1+\widetilde b)(\mu-3)}{(1+\mu)^3},
\end{equation}
which  says $\mu = 3$ is the critical point of $P$. Moreover, we know
\[P''(3) = \frac{1+ \widetilde b}{64} > 0,\]
which clarifies that $P$ has the minimum value at $\mu = 3$.  

\medskip

A straightforward calculation gives all maximum and minimum values of $P$ as follows:
\[P(0) = 1, \quad P(3) = -\frac{9 \widetilde b +1}{8}, \quad P(\infty) := \lim_{\mu \to \infty} P(\mu) = - \widetilde b.\]
We conclude that the maximum and minimum values of $P$ are $1$ and $-\frac{9 \widetilde b +1}{8}$, respectively. Moreover, $P$ never vanishes at any $\mu$ if and only if
\[-\frac{9 \widetilde b +1}{8} > 0 \quad \Longleftrightarrow \quad b > \frac{3}{16},\]
which completes the proof. 
\end{proof}

Together with Lemma \ref{lem:p>0} in \eqref{group velocity} and Definition \ref{New Dis_Par}, we conclude
\begin{lemma}\label{lem:conclusion}
Let $a,b,c$ satisfy \eqref{Conds}, and assume $a=c$. Then, parameter conditions defined in Definition \ref{New Dis_Par} are optimal, in the sense that the group velocity $|w'(k)|$ never vanishes for any wavenumber $k$. Moreover, for $b=\frac3{16}$, $|w'(k)|$ does vanish at nonzero wave number $k=\pm 4$. 
\end{lemma}

\begin{remark}
From Lemma \ref{lem:conclusion}, and in terms of the language of scattering theory, $(u,\eta)=(1,A(\pm 4)) e^{i(\pm 4x-w(\pm 4)t)}$ are internal modes for the linear dynamics near zero. The nonlinear dynamics in those cases should be treated using other techniques, such as the ones in \cite{KMM1}.
\end{remark}

\begin{remark}
As far as we understand, there is no classical wave-like linear PDE with zero group velocity at nonzero wave number $k$. Indeed, for linear Klein-Gordon one has
\[
\partial_t^2 u -\partial_x^2 u + u=0 \implies w'(k)= \pm \frac{|k|}{\sqrt{1+k^2}},
\]
which vanishes only at $k=0$ wave number. Also, for the weakly ill-posed improved Boussinesq model one has
\[
(1-\partial_x^2)\partial_t^2 u -\partial_x^2 u =0 \implies w'(k)= \frac{\pm  1}{(1+k^2)^{3/2}},
\]
which implies that in the limit $k\to\pm\infty$, the group velocity vanishes. Lemma \ref{lem:conclusion} certainly reveals a new phenomenon for $abcd$ in the regime $b\leq \frac3{16}$.
\end{remark}

\bigskip

\section{Decay for nonzero speeds: Proof of Theorem \ref{TH2}}\label{sec:TH2}

In this Section we start the proof of Theorem \ref{TH2}.  

\subsection{A refined variation of functional $\mathcal{H}$ and motivation in Theorem \ref{TH2}}
With a slight abuse of notation, let us choose a weight function of the form
\be\label{new_vp}
\vp = \vp\left(\frac{x-vt}{\lambda(t)}\right)
\ee
in \eqref{I} and \eqref{J}, and where $v \in \R$ is a fixed speed, and $\lambda(t)$ is a time-dependent function to be chosen later. Then it is not difficult to check that we have a refined variation of the functional $\mathcal{H}$:

\begin{proposition}\label{prop:H}
Let $\eta$ and $u$ satisfy \eqref{boussinesq0000}. Let $\mathcal H$ be the functional defined in \eqref{H} with weight function \eqref{new_vp}. For any $\alpha\in \R$ and any $t \in \R$, we have the decomposition
\begin{equation}\label{eq:refined gvirial}
\frac{d}{dt}\mathcal H(t) = \mathcal{VH}_{\lambda(t)}(t) + \mathcal Q_{\lambda(t)}(t) + \mathcal{SQ}_{\lambda(t)}(t) + \mathcal{NQ}_{\lambda(t)}(t),
\end{equation}
where $\mathcal Q_{\lambda(t)}(t)$, $\mathcal{SQ}_{\lambda(t)}(t)$ and $\mathcal{NQ}_{\lambda(t)}(t)$ are defined in \eqref{eq:leading}, \eqref{eq:small linear} and \eqref{eq:nonlinear}, respectively, by simply replacing $\vp^{(n)}$ by $\frac{1}{\lambda(t)^{n}}\vp^{(n)}$, and $\mathcal{VH}_{\lambda(t)}(t) = \mathcal{VH}_{\lambda(t)}[u,\eta](t)$ represents the new additional terms wrt \eqref{eq:gvirial}, and whose sizes are either similar or smaller compared to $\mathcal{Q}_{\lambda(t)}$. More precisely,
\begin{equation}\label{VH}
\begin{aligned}
\mathcal{VH}_{\lambda(t)}(t) :=~{}& -\frac{v}{\lambda(t)}\int \vp' \left(\frac{x-vt}{\lambda(t)} \right) (u\eta + u_x\eta_x) \\
&-\frac{\lambda'(t)}{\lambda(t)}\int\left( \frac{x-vt}{\lambda(t)} \right)\vp' \left(\frac{x-vt}{\lambda(t)} \right) (u\eta + u_x\eta_x) \\
&-\frac{\lambda'(t)}{\lambda(t)^2}\int\vp'\left( \frac{x-vt}{\lambda(t)}\right)(\eta u_x) -\frac{v}{\lambda(t)^2}\int \vp'' \left(\frac{x-vt}{\lambda(t)} \right)(\eta u_x)  \\
&-\frac{\lambda'(t)}{\lambda(t)^2}\int \left( \frac{x-vt}{\lambda(t)} \right) \vp'' \left(\frac{x-vt}{\lambda(t)} \right) (\eta u_x).
\end{aligned}
\end{equation}
\end{proposition} 

\begin{proof}
The proof follows the steps of the proof of Proposition \ref{prop:general virial}, together with the calculations
\begin{equation}\label{Derivative Weight}
\frac{d}{dt} \vp\left( \frac{x-vt}{\lambda(t)}\right) = - \left( \frac{v\lambda(t) + (x-vt)\lambda'(t)}{\lambda(t)^2} \right)\vp'\left( \frac{x-vt}{\lambda(t)}\right),
\end{equation}
and
\[
\begin{aligned}
\frac{d}{dt} \left(\frac{1}{\lambda(t)}\vp'\left( \frac{x-vt}{\lambda(t)}\right) \right) = &~{} -\frac{\lambda'(t)}{\lambda(t)^2}\vp'\left( \frac{x-vt}{\lambda(t)}\right) \\
&~{} - \left( \frac{v\lambda(t) + (x-vt)\lambda'(t)}{\lambda(t)^3} \right)\vp''\left( \frac{x-vt}{\lambda(t)}\right),
\end{aligned}
\]
the rest of the proof is very similar and we skip the details.
\end{proof}

In view of the ideas in \cite{KMPP2018}, the purpose after Proposition \ref{prop:H}, and more specifically Lemma \ref{lem:leading}, is being able to show an estimate of the type 
\begin{equation}\label{1/2 positivity}
\begin{aligned}
\frac{d}{dt} \mathcal{H}(t) \ge&~{}  \frac{1-|v|}{2\la(t)} \int \vp' (u^2+\eta^2 + u_x^2 + \eta_x^2) \\
& ~ {} + \mbox{(positive quantities)} + \mbox{(small terms)},
\end{aligned}
\end{equation}
for $|v| < 1$. Such an estimate is valid if e.g., $A_2,A_3,B_2,B_3>\frac32$ and $A_4,B_4>\frac12$ in Lemma \ref{lem:leading}. Indeed, in that case, the identity (see \eqref{eq:H1_est})
\[
\frac12\int \vp' (f^2 + 3 f_x^2 +3f_{xx}^2 +f_{xxx}^2)  = \frac12 \int \vp' (u^2 + u_x^2)   + \mbox{(small terms)},
\]
(and similar for $g$), together with \eqref{VH}, allows us to conclude \eqref{1/2 positivity}. Consequently, $(a,b,c)$ must satisfy
{\color{black}
\begin{equation}\label{empty set}
a < 0, \;\;\; c < 0, \;\;\; 2ab+3ac \ge b^2, \;\;\; 2bc + 3ac \ge b^2.
\end{equation}
}
This ensures that small (dependent only on $v$), global solutions to \eqref{boussinesq0000} decay to zero in $J_v(t)$ as in \eqref{Jv}, uniformly in $v \in (-1,1)$.

\medskip


However, in contrast with other results (Lemma \ref{lem:Q}, and Lemma 4.5 in \cite{KMPP2018}), from the setting in Lemma \ref{lem:leading}, and under the alternative expression $(\nu,v)$ of $(a,b,c)$ as in \eqref{R0}, \eqref{1/2 positivity} seems not possible to hold, even in some simple parametric settings for $(a,b,c)$. The following lemma claims that the virial identitiy \eqref{eq:refined gvirial} together with \eqref{Conds} do not allow \eqref{1/2 positivity}.

\begin{lemma}\label{lem:lack of parameters}
There is no common set of parameters $(a,b,c)$ satisfying both \eqref{Conds} and \eqref{empty set}. Moreover, \eqref{1/2 positivity} does not hold for \eqref{eq:refined gvirial} with physically well-defined parameters \eqref{Conds}. 
\end{lemma}
\begin{proof}
Putting $a = -\frac{\nu}{2} + \frac13 - b$ and $c = \frac{\nu}{2} - b$ into \eqref{empty set}, one has
\[
b \le \frac{(1-\frac{3\nu}{2})\nu}{2(\nu+\frac13)} < \frac16, \quad b \le \frac{(1-\frac{3\nu}{2})\nu}{2(1-\nu)} < \frac16,
\]
for all $0 \le \nu <1$, which is nonsense. A simple computation also yields that \eqref{empty set} does not make sense for $a = -\frac{\nu}{2} + \frac13 - b$ and $c = \frac{\nu}{2} - b$ in the case $\nu =1$ and $b >0$. Thus, we prove that there is no common parameter $(a,b,c)$ satisfying \eqref{Conds} and \eqref{empty set}.

\medskip

Now we show that the virial identity \eqref{eq:refined gvirial} with \eqref{Conds} never allow \eqref{1/2 positivity}. Putting $a = -\frac{\nu}{2} + \frac13 - b$ and $c = \frac{\nu}{2} - b$ into \eqref{eq:A2B2}--\eqref{eq:A4B4}, one has \eqref{eq:leading-1} (only for the leading terms) as follows:
\[\begin{aligned}
& \frac12 \int \vp' \left( (f^2+g^2) + 3( f_x^2 +g_x^2) + 3 (f_{xx}^2 + g_{xx}^2) + (f_{xxx}^2 + g_{xxx}^2)\right)\\
&\qquad + \int \vp' \left( \widetilde A_2 f_x^2 + \widetilde A_3 f_{xx}^2 + \widetilde A_4 f_{xxx}^2 + \widetilde B_2 g_x^2 +\widetilde B_3 g_{xx}^2 + \widetilde B_4 g_{xxx}^2\right),
\end{aligned}\]
where
\[\widetilde A_2 = \frac{3\nu-2}{4b} - \alpha, \quad \widetilde B_2 = -\frac{3\nu}{4b} + \alpha,\]
\[\widetilde A_3 = \frac{3\nu-2}{3b} - \Bigg(2+\frac{3\nu-2}{6b}\Bigg) \alpha, \quad \widetilde B_3 =-\frac{\nu}{b} + \left(2-\frac{\nu}{2b}\right)\alpha\]
and
\[
\widetilde A_4 = \frac{3\nu-2}{12b} -\Bigg(1+\frac{3\nu-2}{6b}\Bigg) \alpha, \quad \widetilde B_4 = -\frac{\nu}{4b} + \left(1-\frac{\nu}{2b}\right)\alpha.
\]
The first  terms, in addition to the first term in the right-hand side of \eqref{VH} and \eqref{eq:H1_est}, ensures \eqref{1/2 positivity}. However, a  comparison of corresponding coefficients says that there is no $\alpha \in \R$ such that $\widetilde A_i, \widetilde B_i > 0$, $i=2,3,4$. Precisely, we have
\[
\widetilde A_2, \widetilde B_2 \ge 0 \Longleftrightarrow \frac{3\nu -2}{4b} \ge \frac{3\nu}{4b},
\]
while the left-hand side is not possible to hold even for $b > 0$ and $\nu \in \R$. Moreover, an analogous argument yields the same conclusion for all $(\nu,b) \in \mathcal R_0$. As a consequence, we prove there is no $\alpha$ such that $\widetilde A_i, \widetilde B_i > 0$, $i=2,3,4$ for $(\nu,b) \in \mathcal R_0$, which completes the proof of Lemma \ref{lem:lack of parameters}.
\end{proof}

\begin{remark}
 
Thanks to the symmetric structure of $a$ and $c$ in coefficients \eqref{eq:A2B2}--\eqref{eq:A4B4}, putting $a=c=\frac16 - b$ and $\alpha = 0$  into \eqref{eq:A2B2}--\eqref{eq:A4B4}, one has \eqref{eq:leading-1} as follows:
\begin{equation}\label{note1}
\begin{aligned}
&\int \vp' \Bigg( \frac12 (f^2 + g^2) + \Big(\frac32 -\frac{1}{4b}\Big) (f_x^2 + g_x^2) \\
& \qquad \qquad + \Big(\frac32 - \frac{1}{3b}\Big) (f_{xx}^2 + g_{xx}^2 )+ \Big(1-\frac{1}{12b}\Big) (f_{xxx}^2 + g_{xxx}^2)\Bigg),
\end{aligned}
\end{equation}
which, compared to \eqref{eq:H1_est}, cannot be reduced to the first term in right-hand side of \eqref{1/2 positivity}. However, \eqref{note1} reveals that the coefficients in terms of $f_x^2$ and $f_{xx}^2$ approach to $\frac32$, when $b \to \infty$, in that case, one may obtain the less demanding identity
\[
\frac{d}{dt} \mathcal{H}(t) \ge \frac12(1-\delta) \int \vp' (u^2+\eta^2 + u_x^2 + \eta_x^2),
\]
for $0<\delta < 1$ and sufficiently large $b$ depending on $\delta$. 

\end{remark}

\medskip

Theorem \ref{TH2} is inspired by this observation, and thus, we need a refined version of the virial estimates presented in \cite{KMPP2018} for the rigorous description of this observation.


\subsection{A refined representation of $\mathcal Q(t)$} Recall the quadratic term $\mathcal Q(t)$ described in Lemma \ref{lem:leading}:
\begin{equation}\label{Q}
\begin{aligned}
\mathcal Q(t) =& \int \vp' \Big( A_1 f^2 + A_2 f_x^2 + A_3 f_{xx}^2 + A_4 f_{xxx}^2\Big)\\
&+\int \vp' \Big( B_1 g^2 + B_2 g_x^2 + B_3 g_{xx}^2 + B_4 g_{xxx}^2\Big) + \hbox{l.o.t.},
\end{aligned}
\end{equation}
where \hbox{l.o.t.} means lower order terms, and 
\begin{equation}\label{A1B1}
A_1 = B_1 = \frac12>0,
\end{equation}

\begin{equation}\label{A2B2}
A_2 = -\alpha -\frac{3a}{2b}, \qquad B_2 = \alpha - \frac{3c}{2b},
\end{equation}

\begin{equation}\label{A3B3}
A_3 = -\left(1-\frac{a}{b}\right)\alpha -\frac{2a}{b} - \frac12, \qquad B_3 = \left(1-\frac{c}{b}\right)\alpha -\frac{2c}{b} - \frac12
\end{equation}
and
\begin{equation}\label{A4B4}
A_4= -\frac{a}{b}\left( \frac12 - \alpha\right), \qquad B_4= -\frac{c}{b}\left(\frac12 +  \alpha\right).
\end{equation}

Let $0 < v_0 < 1$ be given. The rewriting of the coefficients in \eqref{A2B2}--\eqref{A4B4} in terms of an alternative expression of the parameters given in \eqref{R0} enables us to represent $\mathcal Q(t)$ in \eqref{Q} as follows:
\begin{equation}\label{Q-1}
\begin{aligned}
& \int \vp' \Big( \frac{v_0^+}{2} f^2 + \frac{3v_0^+}{2} f_x^2 + \frac{3v_0^+}{2} f_{xx}^2 + \frac{v_0^+}{2} f_{xxx}^2\Big)\\
&+\int \vp' \Big( \frac{v_0^+}{2} g^2 + \frac{3v_0^+}{2} g_x^2 + \frac{3v_0^+}{2} g_{xx}^2 + \frac{v_0^+}{2} g_{xxx}^2\Big)\\
&+\int \vp' \Big( A_1^* f^2 + A_2^* f_x^2 + A_3^* f_{xx}^2 + A_4^* f_{xxx}^2\Big)\\
&+\int \vp' \Big( B_1^* g^2 + B_2^* g_x^2 + B_3^* g_{xx}^2 + B_4^* g_{xxx}^2\Big) + l.o.t.,
\end{aligned}
\end{equation}
where
\begin{equation}\label{v_0^+}
v_0^+ := 1- \frac{1-v_0}{2} = \frac{1+v_0}{2},
\end{equation} 
and
\begin{equation}\label{AB1}
A_1^* = B_1^* = \frac{1-v_0^+}{2}>0,
\end{equation}

\begin{equation}\label{AB2}
A_2^* = \frac{3(1-v_0^+)}{2} + \frac{3\nu-2}{4b} - \alpha, \quad B_2^* = \frac{3(1-v_0^+)}{2} -\frac{3\nu}{4b} + \alpha,
\end{equation}

\begin{equation}\label{AB3}
\begin{aligned}
A_3^* =&~{} \frac{3(1-v_0^+)}{2} + \frac{3\nu-2}{3b} - \left(2+\frac{3\nu-2}{6b}\right) \alpha ,\\
B_3^* = &~{}\frac{3(1-v_0^+)}{2} -\frac{\nu}{b} + \left(2-\frac{\nu}{2b}\right)\alpha,
\end{aligned}
\end{equation}
and
\begin{equation}\label{AB4}
\begin{aligned}
A_4^*= &~{} \frac{1-v_0^+}{2} + \frac{3\nu-2}{12b} - \left(1+\frac{3\nu-2}{6b}\right) \alpha,\\
B_4^*=&~{} \frac{1-v_0^+}{2} -\frac{\nu}{4b} + \left(1-\frac{\nu}{2b}\right)\alpha.
\end{aligned}
\end{equation}

\medskip

In order to obtain a \emph{maximal positivity condition} (see Proposition \ref{prop:CONDITION} below), we shall see that the last two terms in \eqref{Q-1} necessarily must obey a positivity condition. Let
\begin{equation}\label{k_0}
\kappa_0 = \frac{1-v_0^+}{2} =\frac{1-v_0}{4}.
\end{equation}
One has $0 < \kappa_0 <\frac14$ for all $0< v_0 < 1$. 

\begin{lemma}[Basic nonnegativity conditions]
Let $0 < v_0 < 1$ be given. Define $\kappa_0$ as in \eqref{k_0}. Then, one has
\ben
\item Nonnegativity on $f$:
\begin{equation}\label{vanish1}
\int \vp' \Big( A_1^* f^2 + A_2^* f_x^2 + A_3^* f_{xx}^2 + A_4^* f_{xxx}^2\Big) \ge 0,
\end{equation}
provided $\alpha \le \min (\bar{A}_2, \bar{A}_3, \bar{A}_4)$, where
\begin{equation}\label{A2A3A4}
\begin{aligned}
& \bar{A}_2 := 3\kappa_0 + \frac{3\nu-2}{4b}, \quad \bar{A}_3 := \frac{18\kappa_0b+2(3\nu-2)}{12b+3\nu-2}, \\
& \qquad \qquad  \hbox{and} \qquad\bar{A}_4: = \frac{12\kappa_0b+3\nu-2}{12b+2(3\nu-2)}.
\end{aligned}
\end{equation}
\item Nonnegativity on $g$: similarly,
\begin{equation}\label{vanish2}
\int \vp' \Big( B_1^* g^2 + B_2^* g_x^2 + B_3^* g_{xx}^2 + B_4^* g_{xxx}^2\Big) \ge 0,
\end{equation}
if $\alpha \ge \max (\bar{B}_2, \bar{B}_3, \bar{B}_4)$, where
\begin{equation}\label{B2B3B4}
\bar{B}_2 := \frac{3\nu}{4b} - 3\kappa_0, \quad \bar{B}_3 := \frac{2\nu-6\kappa_0b}{4b-\nu}, \quad \bar{B}_4 := \frac{\nu-4\kappa_0b}{4b-2\nu}.
\end{equation}
\een
\end{lemma}

\begin{proof}
The nonnegativity in \eqref{vanish1} holds provided $A_j^*\geq 0$, for $j=1,2,3,4$ in \eqref{AB1}-\eqref{AB4}. We readily have $\alpha \le \min (\bar{A}_2, \bar{A}_3, \bar{A}_4)$. The case for $g$ in \eqref{vanish2} is obtained in a similar fashion.
\end{proof}

In what follows, our objective will be a deeper understanding of the quantities $\bar A_j$ and $\bar B_k$ defined in \eqref{A2A3A4} and \eqref{B2B3B4}.

\subsection{Preliminary setting} We start out with some definitions.

\begin{definition}[Threshold parameters]\label{rs}
Let $0 < \sigma < \frac12$ and $\rho >0$ be given. 

\ben
\item We define the threshold parameters $r_{\pm}(\sigma,\rho)$ and $\wt{r}_{\pm}(\sigma,\rho)$ by
\begin{equation}\label{r+-min}
r_{\pm}(\sigma,\rho) := \frac23 + \frac23\rho\left(-(2\sigma+1) \pm \sqrt{(2\sigma+1)^2 - 6\sigma}\right),
\end{equation}
and
\begin{equation}\label{r+-max}
\wt{r}_{\pm}(\sigma,\rho) := \frac23\rho\left((2\sigma+1) \pm \sqrt{(2\sigma+1)^2 - 6\sigma}\right),
\end{equation}
respectively. 
\smallskip
\item Additionally, if $0 < \sigma < \frac{2-\sqrt{3}}{3}$, we further define the modified threshold parameters $s_{\pm}(\sigma,\rho)$ and $\wt{s}_{\pm}(\sigma,\rho)$ by the expressions
\begin{equation}\label{s+-min}
s_{\pm}(\sigma,\rho): = \frac23 + \frac23\rho\left(-(3\sigma+1) \pm \sqrt{(3\sigma+1)^2 - 18\sigma}\right),
\end{equation}
and
\begin{equation}\label{s+-max}
\wt{s}_{\pm}(\sigma,\rho) := \frac23\rho\left((3\sigma+1) \pm \sqrt{(3\sigma+1)^2 - 18\sigma}\right).
\end{equation}
\een
\end{definition}

The above defined parameters will be important to differentiate the possible behaviors of the parameters $\bar A_j$ and $\bar B_k$. We first prove

\begin{lemma}[Properties of $r_{\pm}(\sigma,\rho)$, $\wt{r}_{\pm}(\sigma,\rho)$, $s_{\pm}(\sigma,\rho)$ and $\wt{s}_{\pm}(\sigma,\rho)$]\label{lem:properties of points}
Let $0 < \sigma < \frac12$ and $\rho > 0$ be given. Then,
\begin{enumerate}
\item $\frac23 - 2\rho < r_-(\sigma,\rho)$.
\smallskip
\item $\wt{r}_+(\sigma,\rho) < 2\rho$.
\smallskip
\item The interval $(r_-(\sigma,\rho),r_+(\sigma,\rho))$ is well-defined and nonempty. 
\smallskip
\item The interval $(\wt{r}_-(\sigma,\rho),\wt{r}_+(\sigma,\rho))$ is well-defined and nonempty. 
\end{enumerate}
Moreover, if $0 < \sigma < \frac{2-\sqrt{3}}{3}\sim 0.09$, one has
\begin{enumerate}
\item[(5)] The interval $(s_-(\sigma,\rho),s_+(\sigma,\rho))$ is well-defined and nonempty. 
\smallskip
\item[(6)] The interval $(\wt{s}_-(\sigma,\rho),\wt{s}_+(\sigma,\rho))$ is well-defined and nonempty. 
\smallskip
\item[(7)] $s_+(\sigma,\rho) < r_+(\sigma,\rho)$ and $\wt{s}_+(\sigma,\rho) < \wt{r}_+(\sigma,\rho)$.
\smallskip
\item[(8)] $r_-(\sigma,\rho) < s_-(\sigma,\rho)$ and $\wt{r}_-(\sigma,\rho) < \wt{s}_-(\sigma,\rho) $.
\end{enumerate}
\end{lemma}

\begin{proof}
A straightforward computation using \eqref{r+-min} and \eqref{r+-max} shows items (1) and (2). Indeed, (1) follows from the bound
\[
2\sigma +1 + \sqrt{ (2\sigma+1)^2 -6\sigma} <3, \qquad 0 < \sigma < \frac12.
\]
Additionally, (2) follows from the same identity. Since 
\begin{equation}\label{positive}
(2\sigma+1)^2-6\sigma = (2\sigma -1)^2 + 2\sigma>0,
\end{equation} 
we prove Items (3) and (4).

\medskip

Now we deal with the second part. For $0 < \sigma < \frac{2-\sqrt{3}}{3}$, $(3\sigma+1)^2 - 18\sigma > 0$, hence items (5) and (6) immediately hold. The fact that
\[
\sqrt{(2\sigma+1)^2-6\sigma} > \sqrt{(3\sigma+1)^2-18\sigma},
\]
for $0 < \sigma < 2$, ensures item (7). Item (8) follows similarly, provided $\sigma \in (0,2)$.
\end{proof}

\begin{definition}
In what follows we denote by $r_{\pm}$, $\wt{r}_{\pm}$, $s_{\pm}$ and $\wt{s}_{\pm}$ the values $r_{\pm}(\kappa_0,b)$, $\wt{r}_{\pm}(\kappa_0,b)$, $s_{\pm}(\kappa_0,b)$ and $\wt{s}_{\pm}(\kappa_0,b)$, respectively, introduced in Definition \ref{rs}.
\end{definition}

More precisely, and in order to avoid confusion below, we have
\begin{equation}\label{r+-min_new}
r_{\pm} = \frac23 + \frac23 b\left(-(2\kappa_0+1) \pm \sqrt{(2\kappa_0+1)^2 - 6\kappa_0}\right),
\end{equation}
and
\begin{equation}\label{r+-max_new}
\wt{r}_{\pm}= \frac23b \left((2\kappa_0+1) \pm \sqrt{(2\kappa_0+1)^2 - 6\kappa_0}\right).
\end{equation}
Similarly, for $0 < \kappa_0 < \frac{2-\sqrt{3}}{3}$,
\begin{equation}\label{s+-min_new}
s_{\pm}: = \frac23 + \frac23b\left(-(3\kappa_0+1) \pm \sqrt{(3\kappa_0+1)^2 - 18\kappa_0}\right),
\end{equation}
and
\begin{equation}\label{s+-max_new}
\wt{s}_{\pm} := \frac23b\left((3\kappa_0+1) \pm \sqrt{(3\kappa_0+1)^2 - 18\kappa_0}\right).
\end{equation}

\begin{remark}
Lemma \ref{lem:properties of points} is valid for the case when $(\sigma,\rho) = (\kappa_0, b)$, where $\kappa_0$ is in \eqref{k_0}, in particular for \eqref{r+-min_new}--\eqref{s+-max_new}, since $0 < \kappa_0 < \frac14$ and $b > \frac16$.
\end{remark}

\subsection{Quantitative comparison, first part} Let $\kappa_0$ be as in \eqref{k_0}. In what follows, our objective will be to compare the quantities $\bar A_j$ in the case $v_0$ small. 

\begin{lemma}[Case $\kappa_0$ large, i.e. $v_0$ small]\label{lem:upper minimum}
Let $\kappa_0$ satisfy $\frac{2-\sqrt{3}}{3} \le \kappa_0 < \frac14$. Let $\bar{A}_2, \bar{A}_3,\bar{A}_4$ be given in \eqref{A2A3A4} and $(\nu, b) \in \mathcal R_0$ \eqref{R0}. Let  $r_{\pm}$ be defined in \eqref{r+-min_new}. Then the following hold:
\begin{enumerate}
\item $\bar{A}_2 \ge \bar{A}_3$.
\item $\min (\bar{A}_3,\bar{A}_4) = \bar{A}_3$, if $\nu \in [r_-,r_+]$, otherwise $\min (\bar{A}_3,\bar{A}_4)=\bar{A}_4$.
\end{enumerate}
\end{lemma}

\begin{proof}
A computation gives
\begin{equation}\label{A2-A3}
\bar{A}_2 - \bar{A}_3 = \frac{(3\nu-2)^2 + 4b(3\kappa_0+1)(3\nu-2) + 72\kappa_0b^2}{4b(12b+3\nu-2)}. 
\end{equation}
Note that the denominator $4b(12b+3\nu-2) >0$ for all $(\nu,b) \in \mathcal R_0$. If {\color{black} $(3\kappa_0 +1)^2 - 18\kappa_0 \le 0$}, we have the numerator $(3\nu-2)^2 + 4b(3\kappa_0+1)(3\nu-2) + 72\kappa_0b^2 \ge 0$ for all possible value of $3\nu-2$, and hence Item (1) holds. Note that the condition $\frac{2-\sqrt{3}}{3} \le \kappa_0 \; (< \frac14)$ ensures $(3\kappa_0 +1)^2 - 18\kappa_0 \le 0$.

\medskip

Similarly, a straightforward calculation in \eqref{A2A3A4} gives
\[
\bar{A}_3 - \bar{A}_4 = \frac{3\left((3\nu-2)^2 + 4b(2\kappa_0+1)(3\nu-2) + 24\kappa_0b^2\right)}{(12b+3\nu-2)(12b+2(3\nu-2))}.
\]
Note that the denominator $(12b+3\nu-2)(12b+2(3\nu-2)) > 0$ for all $(\nu,b) \in \mathcal R_0$, see \eqref{R0}. From \eqref{positive}, we know that $(3\nu-2)^2 + 4b(2\kappa_0+1)(3\nu-2) + 24\kappa_0b^2$ is \emph{nonnegative} if 
\[
2(-(2\kappa_0+1)-\sqrt{(2\kappa_0+1)^2-6\kappa_0})b \le 3\nu -2 \le 2(-(2\kappa_0+1)+ \sqrt{(2\kappa_0+1)^2-6\kappa_0})b,
\]
that is, from \eqref{r+-min_new},
\[
3r_- - 2 \le 3\nu -2 \le 3 r_+ -2.
\]
Otherwise, it is positive, which exactly proves Item (2).
\end{proof}

\begin{lemma}[Case $\kappa_0$ small, i.e. $v_0$ large]\label{lem:lower minimum}
Let $\kappa_0$ satisfy now $0 < \kappa_0 < \frac{2-\sqrt{3}}{3}$. Let $\bar{A}_2, \bar{A}_3,\bar{A}_4$ be given in \eqref{A2A3A4} and $(\nu, b) \in \mathcal R_0$. Let $r_{\pm}$ and $s_{\pm}$ be defined in \eqref{r+-min_new} and \eqref{s+-min_new}, respectively. Then the following hold:
\smallskip
\begin{enumerate}
\item $\min (\bar{A}_2,\bar{A}_3,\bar{A}_4) = \bar{A}_2$, if $\nu \in [s_-,s_+]$.
\smallskip
\item $\min (\bar{A}_2,\bar{A}_3,\bar{A}_4) = \bar{A}_3$, if $\nu \in [r_-,s_-] \cup [s_+,r_+]$.
\smallskip
\item $\min (\bar{A}_2,\bar{A}_3,\bar{A}_4) = \bar{A}_4$, if $\nu \le r_-$ or $\nu \ge r_+$.
\end{enumerate}
\end{lemma}

\begin{proof}
In view of the proof of Lemma \ref{lem:upper minimum} (in particular, \eqref{A2-A3}), under the new condition $0 < \kappa_0 < \frac{2-\sqrt{3}}{3}$, the polynomial $(3\nu-2)^2 + 4b(3\kappa_0+1)(3\nu-2) + 72\kappa_0b^2$ is no longer nonnegative for all $3\nu-2$. Therefore, it has roots and changes sign.

\medskip

Similarly as the second part of the proof of Lemma \ref{lem:upper minimum}, we have from \eqref{A2-A3} that  $\bar{A}_2 \le \bar{A}_3$ if $\nu \in [s_-,s_+]$, otherwise, $\bar{A}_2 > \bar{A}_3$. Indeed, the roots of the convex polynomial on $(3\nu-2)$ given by $(3\nu-2)^2 + 4b(3\kappa_0+1)(3\nu-2) + 72\kappa_0b^2$ are
\[
3\nu-2 = 2\left(-(3\kappa_0+1) \pm \sqrt{ (3\kappa_0+1)^2 - 18 \kappa_0}\right)b = 3s_{\pm} -2.
\] 
On the other hand, Lemma \ref{lem:properties of points} ensures the ordering $r_- < s_- < s_+ < r_+$. Thanks to Item (2) in Lemma \ref{lem:upper minimum}, we have $\min (\bar{A}_3,\bar{A}_4) = \bar{A}_3$, if $\nu \in [r_-,r_+]$, otherwise $\min (\bar{A}_3,\bar{A}_4)=\bar{A}_4$. Gathering all this information, we finally prove all items in Lemma \ref{lem:lower minimum}.
\end{proof}

Collecting Lemmas \ref{lem:upper minimum} and \ref{lem:lower minimum}, one has the following complete description.
 
\begin{proposition}[Comparison of $\bar A_j$]\label{prop:min}
Let $0 < \kappa_0 < \frac14$ be given. Let $\bar{A}_2, \bar{A}_3,\bar{A}_4$ be given in \eqref{A2A3A4} and $(\nu, b) \in \mathcal R_0$. Let $r_{\pm}$ and $s_{\pm}$ be defined as in \eqref{r+-min} and \eqref{s+-min}, respectively.
\begin{enumerate}
\item When $\nu \le r_-$ or $\nu \ge r_+$, one has 
\begin{equation}
\min(\bar{A}_2, \bar{A}_3,\bar{A}_4) = \bar{A}_4.
\end{equation}

\item When $\nu \in [r_-, r_+]$, one has 
\begin{equation}
\begin{aligned}
&\min(\bar{A}_2, \bar{A}_3,\bar{A}_4) =  \\
& \begin{cases}\begin{array}{lcl} 
\begin{cases}\begin{array}{lcl} 
\bar{A}_2, &&\mbox{if} \quad \nu \in [s_-,s_+],\\
\bar{A}_3, &&\mbox{if} \quad \nu \in \left([r_-, s_-] \cup [s_+,r_+]\right) \cap [0,1],
\end{array}\end{cases} &&\hspace{-1.5em}\mbox{if} \quad 0 < \kappa_0 < \frac{2-\sqrt{3}}{3},\\
\bar{A}_3, &&\hspace{-1.5em}\mbox{if} \quad \frac{2-\sqrt{3}}{3} \le \kappa_0 < \frac14.
\end{array}\end{cases}
\end{aligned}
\end{equation}
\end{enumerate}
\end{proposition}

Figure \ref{fig:Minimum} describes the results in Proposition \ref{prop:min}.
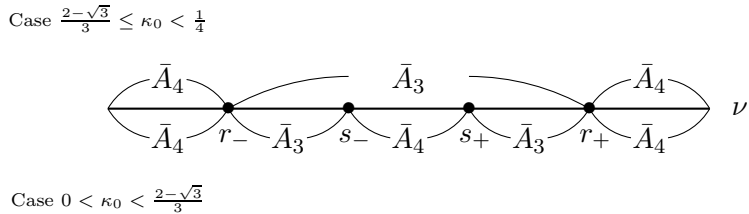
\begin{figure}[h!]
\begin{center}
\begin{tikzpicture}[scale=0.8]
\draw[thick] (-5,3) -- (5,3);
\node at (-3,3){$\bullet$};
\node at (-1,3){$\bullet$};
\node at (1,3){$\bullet$};
\node at (3,3){$\bullet$};
\node at (-2.9,2.5){$r_-$};
\node at (-0.9,2.5){$s_-$};
\node at (1.1,2.5){$s_+$};
\node at (3.1,2.5){$r_+$};
\node at (-5,4.5){\tiny{Case $\frac{2-\sqrt{3}}{3} \le \kappa_0 < \frac14$}};
\draw (-5,3) arc (150:100:1);
\node at (-4,3.5){$\bar{A}_4$};
\draw (-3,3) arc (30:80:1);
\draw (-3,3) arc (120:90:4);
\node at (0,3.5){$\bar{A}_3$};
\draw (3,3) arc (60:90:4);
\draw (3,3) arc (150:100:1);
\node at (4,3.5){$\bar{A}_4$};
\draw (5,3) arc (30:80:1);
\node at (-5,1.5){\tiny{Case $0 < \kappa_0 < \frac{2-\sqrt{3}}{3}$}};
\draw (-5,3) arc (-150:-100:1);
\node at (-4,2.5){$\bar{A}_4$};
\draw (-3,3) arc (-30:-80:1);
\draw (-3,3) arc (-150:-100:1);
\node at (-2,2.5){$\bar{A}_3$};
\draw (-1,3) arc (-30:-80:1);
\draw (-1,3) arc (-150:-100:1);
\node at (0,2.5){$\bar{A}_4$};
\draw (1,3) arc (-30:-80:1);
\draw (1,3) arc (-150:-100:1);
\node at (2,2.5){$\bar{A}_3$};
\draw (3,3) arc (-30:-80:1);
\draw (3,3) arc (-150:-100:1);
\node at (4,2.5){$\bar{A}_4$};
\draw (5,3) arc (-30:-80:1);
\node at (5.5,3){$\nu$};
\end{tikzpicture}
\end{center}
\caption{Minimum of $\bar{A}_2, \bar{A}_3,\bar{A}_4$ with respect to the position of $\nu$} \label{fig:Minimum}
\end{figure}

\subsection{Quantitative comparison, second part} Recall that $\kappa_0$ was introduced in \eqref{k_0}.

\begin{lemma}\label{lem:upper maximum}
Let $\kappa_0$ satisfy $\frac{2-\sqrt{3}}{3} \le \kappa_0 < \frac14$. Let $\bar{B}_2, \bar{B}_3,\bar{B}_4$ be given in \eqref{B2B3B4} and $(\nu, b) \in \mathcal R_0$. Let $\wt{r}_{\pm}$ be defined in \eqref{r+-max}. Then the following hold:
\begin{enumerate}
\item $\bar{B}_2 \le \bar{B}_3$.
\smallskip
\item $\max (\bar{B}_3,\bar{B}_4) = \bar{B}_3$, if $\nu \in [\wt{r}_-,\wt{r}_+]$, otherwise $\bar{B}_4$.
\end{enumerate}
\end{lemma}

\begin{proof}
The proof is analogous to the proof Lemma \ref{lem:upper minimum}, thus we omit the details. 
\end{proof}

\begin{lemma}\label{lem:lower maximum}
Let $\kappa_0$ satisfy $0 < \kappa_0 < \frac{2-\sqrt{3}}{3}$. Let $\bar{B}_2, \bar{B}_3,\bar{B}_4$ be given in \eqref{B2B3B4} and $(\nu, b) \in \mathcal R_0$. Let $\wt{r}_{\pm}$ and $\wt{s}_{\pm}$ be defined in \eqref{r+-max} and \eqref{s+-max}, respectively. Then the following hold:
\begin{enumerate}
\item $\max (\bar{B}_2, \bar{B}_3,\bar{B}_4) = \bar{B}_2$, if $\nu \in [\wt{s}_-,\wt{s}_+]$.
\smallskip
\item $\max (\bar{B}_2, \bar{B}_3,\bar{B}_4) = \bar{B}_3$, if $\nu \in [\wt{r}_-,\wt{s}_-] \cup [\wt{s}_+,\wt{r}_+]$.
\smallskip
\item $\max (\bar{B}_2, \bar{B}_3,\bar{B}_4) = \bar{B}_4$, if $\nu \le \wt{r}_-$ or $\nu \ge \wt{r}_+$.
\end{enumerate}
\end{lemma}

\begin{proof}
A same argument used in the proof of Lemma \ref{lem:lower minimum}, in addition to Lemmas \ref{lem:properties of points} and \ref{lem:upper maximum}, proves Lemma \ref{lem:lower maximum}, thus we omit the details.
\end{proof}

Together with Lemmas \ref{lem:upper maximum} and \ref{lem:lower maximum}, we summarize

\begin{proposition}\label{prop:max}
Let $0 < \kappa_0 < \frac14$ be given. Let $\bar{B}_2, \bar{B}_3,\bar{B}_4$ be given in \eqref{B2B3B4} and $(\nu, b) \in \mathcal R_0$. Let $\wt{r}_{\pm}$ and $\wt{s}_{\pm}$ be defined as in \eqref{r+-max} and \eqref{s+-max}, respectively. 
\begin{enumerate}
\item When $\nu \le \wt{r}_-$ or $\nu \ge \wt{r}_+$, one has 
\begin{equation}
\max(\bar{B}_2, \bar{B}_3,\bar{B}_4) = \bar{B}_4.
\end{equation}

\item When $\nu \in [\wt{r}_-, \wt{r}_+]$, one has 
\begin{equation}
\begin{aligned}
&\max(\bar{B}_2, \bar{B}_3,\bar{B}_4) =  \\
& \begin{cases}\begin{array}{lcl} 
\begin{cases}\begin{array}{lcl} 
\bar{B}_2, &&\mbox{if} \quad \nu \in [\wt{s}_-,\wt{s}_+],\\
\bar{B}_3, &&\mbox{if} \quad \nu \in \left([\wt{r}_-, \wt{s}_-] \cup [\wt{s}_+,\wt{r}_+]\right) \cap [0,1],
\end{array}\end{cases} &&\hspace{-1.5em}\mbox{if} \quad 0 < \kappa_0 < \frac{2-\sqrt{3}}{3},\\
\bar{B}_3, &&\hspace{-1.5em}\mbox{if} \quad \frac{2-\sqrt{3}}{3} \le \kappa_0 < \frac14.
\end{array}\end{cases}
\end{aligned}
\end{equation}
\end{enumerate}
\end{proposition}

Similarly, Proposition \ref{prop:max} can be easily understood from Figure \ref{fig:Maximum}.

\begin{figure}[h!]
\begin{center}
\begin{tikzpicture}[scale=0.8]
\draw[thick] (-5,3) -- (5,3);
\node at (-3,3){$\bullet$};
\node at (-1,3){$\bullet$};
\node at (1,3){$\bullet$};
\node at (3,3){$\bullet$};
\node at (-2.9,2.5){$\wt{r}_-$};
\node at (-0.9,2.5){$\wt{s}_-$};
\node at (1.1,2.5){$\wt{s}_+$};
\node at (3.1,2.5){$\wt{r}_+$};
\node at (-5,4.5){\tiny{$\frac{2-\sqrt{3}}{3} \le \kappa_0 < \frac14$}};
\draw (-5,3) arc (150:100:1);
\node at (-4,3.5){$\bar{B}_4$};
\draw (-3,3) arc (30:80:1);
\draw (-3,3) arc (120:90:4);
\node at (0,3.5){$\bar{B}_3$};
\draw (3,3) arc (60:90:4);
\draw (3,3) arc (150:100:1);
\node at (4,3.5){$\bar{B}_4$};
\draw (5,3) arc (30:80:1);
\node at (-5,1.5){\tiny{$0 < \kappa_0 < \frac{2-\sqrt{3}}{3}$}};
\draw (-5,3) arc (-150:-100:1);
\node at (-4,2.5){$\bar{B}_4$};
\draw (-3,3) arc (-30:-80:1);
\draw (-3,3) arc (-150:-100:1);
\node at (-2,2.5){$\bar{B}_3$};
\draw (-1,3) arc (-30:-80:1);
\draw (-1,3) arc (-150:-100:1);
\node at (0,2.5){$\bar{B}_4$};
\draw (1,3) arc (-30:-80:1);
\draw (1,3) arc (-150:-100:1);
\node at (2,2.5){$\bar{B}_3$};
\draw (3,3) arc (-30:-80:1);
\draw (3,3) arc (-150:-100:1);
\node at (4,2.5){$\bar{B}_4$};
\draw (5,3) arc (-30:-80:1);
\end{tikzpicture}
\end{center}
\caption{Maximum of $\bar{B}_2, \bar{B}_3,\bar{B}_4$ with respect to the position of $\nu$} \label{fig:Maximum}
\end{figure}
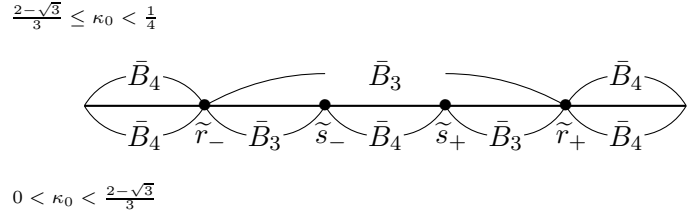

\subsection{Maximal positivity condition} We need some further understanding of the position of the points $r_{\pm}$, $\wt{r}_{\pm}$, $s_{\pm}$ and $\wt{s}_{\pm}$. That information will depend on the value of $\kappa_0$, and two additional parameters $b_1$ and $b_2$ that we introduce now.

\begin{lemma}\label{lem:rs ordering}
Let $0 < \kappa_0 < \frac14$ be given. Let $r_{\pm}$, $\wt{r}_{\pm}$, $s_{\pm}$ and $\wt{s}_{\pm}$ be defined as in \eqref{r+-min_new}-\eqref{r+-max_new}-\eqref{s+-min_new}-\eqref{s+-max_new}. Set 
\begin{equation}\label{b_1}
b_1 := \frac{1}{2(2\kappa_0+1)},
\end{equation}
and, for $0 < \kappa_0 < \frac{2-\sqrt{3}}{3}$,
\begin{equation}\label{b_2}
b_2 := \frac{1}{5\kappa_0+2-\sqrt{(2\kappa_0+1)^2-6\kappa_0}-\sqrt{(3\kappa_0+1)^2-18\kappa_0}}.
\end{equation}
Then, the following hold.
\ben
\item When $\frac{2-\sqrt{3}}{3} \le \kappa_0 < \frac14$,  one has
\begin{equation}\label{r ordering1}
r_- <  r_+ < \wt{r}_- <  \wt{r}_+ 
\end{equation}
or
\begin{equation}\label{r ordering2}
r_- <  \wt{r}_- < r_+  <  \wt{r}_+,
\end{equation}
if $b > b_1$.
\smallskip
\item When $0 < \kappa_0 < \frac{2-\sqrt{3}}{3}$, 
\begin{equation}\label{rs ordering1}
r_- <s_-<s_+< r_+ < \wt{r}_-<\wt{s}_-<\wt{s}_+ < \wt{r}_+
\end{equation}
or
\begin{equation}\label{rs ordering2}
r_- <s_-<s_+< \wt{r}_- < r_+<\wt{s}_-<\wt{s}_+ < \wt{r}_+,
\end{equation}
if $b > b_2$.
\een
\end{lemma}

\begin{proof}
A straightforward calculation gives
\[r_+ < \wt{r}_+ \quad \Longleftrightarrow \quad \frac23 < \frac23 \left(2(\kappa_0+1) \right)b,\]
and hence so \eqref{r ordering1} or \eqref{r ordering2}.

\medskip

When $0 < \kappa_0 < \frac{2-\sqrt{3}}{3}$, similarly as before, one knows 
\[r_+ < \wt{s}_- \quad  \Longleftrightarrow \quad \frac23 < \frac23 \left(5\kappa_0 +2 -\sqrt{(2\kappa_0+1)^2-6\kappa_0} -\sqrt{(3\kappa_0+1)^2-18\kappa_0}\right)b.\]
On the other hand, a direct calculation gives $b_2 > b_1$. Thus, \eqref{r ordering1} and \eqref{r ordering2} in addition to Lemma \ref{lem:properties of points} prove \eqref{rs ordering1} or \eqref{rs ordering2}.
\end{proof}

%


\begin{lemma}\label{lem:lower}
Let $0 < \kappa_0 < \frac14$ be given and $b_0$ be given in \eqref{b_0}. Let $\bar{A}_i$ and $\bar{B}_i$ be given in \eqref{A2A3A4} and \eqref{B2B3B4}, respectively. Set 
\begin{equation}\label{b_0}
b_0 = \frac{1}{9\kappa_0}.
\end{equation}
Then,
\begin{enumerate}
\item $\bar{A}_4 \ge \bar{B}_4$ happens only when $r_- < \nu < \wt{r}_+$, if $b > \frac16$.
\end{enumerate}
Assume now $0 < \kappa_0 <\frac{2-\sqrt{3}}{3}$. If $b\ge b_0$ and $\nu \in [0,1]$, the following hold:
\begin{enumerate}
\item[(2)] $s_+ < \nu < \wt{s}_+ $.
\smallskip
\item[(3)] The region of $(\nu,b)$, for which $\bar{A}_3 \ge \bar{B}_3$ holds, is included in the region defined by $\bar{A}_4 \ge \bar{B}_4$, if $0 \le \nu \le 1$.
\smallskip
\item[(4)] $\bar{A}_4 \ge \bar{B}_2$ happens only when $\frac23 - 2b < \nu < \wt{s}_-$, if $b > \frac16$.
\end{enumerate}
\end{lemma}

\begin{proof}
A computation gives
\begin{equation}\label{hyperbola1}
\bar{A}_4 \ge \bar{B}_4 \quad \Longleftrightarrow \quad \frac{\left(b-\frac{2\kappa_0+1}{24\kappa_0}\right)^2}{3} - \frac{\left(\nu - \frac13\right)^2}{24\kappa_0} \ge \frac{(2\kappa_0-1)^2}{1728\kappa_0^2}.
\end{equation}
Let $b=\Gamma_1(\nu)$ be a hyperbola derived by \eqref{hyperbola1} (only upper part is available). Note that $\Gamma_1(\nu)$ has a minimum value at $\nu = \frac13$. Moreover, two lines $\nu = r_-$ and $\nu = \wt{r}_+$ intersect at 
\[(\nu, b) = \left(\frac13, \frac{1}{2\left((2\kappa_0+1) +\sqrt{(2\kappa_0+1)^2-6\kappa_0} \right)} \right).\]
Hence, in order to show that the hyperbola $b = \Gamma_1(\nu)$ is located inside the region constructed by $\max(0,r_-) <\nu < \min(1,\wt{r}_+)$, it suffices to show\footnote{One can show it without the restriction $\nu \in [0,1]$ by the comparing the slopes, but we do not explore it here.}
\begin{equation}\label{comparison1}
\frac{1}{2\left((2\kappa_0+1) +\sqrt{(2\kappa_0+1)^2-6\kappa_0} \right)} < \Gamma_1\left(\frac13\right),
\end{equation}
\begin{equation}\label{comparison2}
\frac{1}{(2\kappa_0+1) +\sqrt{(2\kappa_0+1)^2-6\kappa_0}} < \Gamma_1(0),
\end{equation}
and
\begin{equation}\label{comparison3}
\frac{3}{2\left((2\kappa_0+1) +\sqrt{(2\kappa_0+1)^2-6\kappa_0}\right)} < \Gamma_1(1).
\end{equation}
One  has $\Gamma_1(1/3) = \frac{1}{12\kappa_0}$, and hence \eqref{comparison1}. On the other hand, one has 
\[\Gamma_1(0) = \frac{2\kappa_0+1}{12\kappa_0}, \quad \Gamma_1(1) = \frac{(2\kappa_0+1) + \sqrt{4\kappa_0 + 28\kappa_0 +1}}{24\kappa_0},\]
and hence \eqref{comparison2} and \eqref{comparison3} by a computation. Thus we complete the proof of Item (1).

\medskip

For Item (2), $\nu = s_+$ is decreasing and $\nu = \wt{s}_+$ is increasing  in $\nu$ , we put $\nu = 0$ (for the left) and $\nu = 1$ (for the right). Then, a computation gives
\[\frac{1}{3\kappa_0+1 - \sqrt{(3\kappa_0+1)^2-18\kappa_0}} < b, \]
for the left, and
\[\frac{3}{6\kappa_0+2 + 2\sqrt{(3\kappa_0+1)^2-18\kappa_0}} < b,\]
for the right. Both lower bounds are definitely smaller than $b_0$ for $0 < \kappa_0 < \frac{2-\sqrt{3}}{3}$. 

\medskip

A not very complicated computation gives
\begin{equation}\label{hyperbola2}
\bar{A}_3 \ge \bar{B}_3 \quad\Longleftrightarrow \quad\frac{\left(b-\frac{3\kappa_0+4}{72\kappa_0}\right)^2}{3} - \frac{\left(\nu - \frac13\right)^2}{36\kappa_0} \ge \frac{(3\kappa_0-4)^2}{15552\kappa_0^2}.
\end{equation}
Let $b=\Gamma_2(\nu)$ be the hyperbola derived by \eqref{hyperbola2} (only the upper parts are available). Note that $\Gamma_2(\nu)$ also has a minimum value at $\nu = \frac13$, in particular, $\Gamma_2(1/3) = \frac{1}{9\kappa_0} = b_0$. Then, since
\[\Gamma_2(0) = \frac{3\kappa_0+4}{36\kappa_0}, \quad \Gamma_2(1) = \frac{3\kappa_0+4 + \sqrt{9\kappa_0^2+168\kappa_0+16}}{72\kappa_0},\]
one  checks 
\[\Gamma_1(0) < \Gamma_2(0), \quad \Gamma_1\left(\frac13\right) < \Gamma_2\left(\frac13\right), \quad \Gamma_1(1) < \Gamma_2(1),\]
for $0 < \kappa_0 < \frac{2-\sqrt{3}}{3}$, which implies Item (3).

\medskip

Finally, we have
\begin{equation}\label{hyperbola3}
\begin{aligned}
\bar{A}_4 \ge &~{} \bar{B}_2 \quad \Longleftrightarrow \\
& \frac{\left(b-\frac{(1+6\kappa_0)+3(1-3a)\nu}{48\kappa_0}\right)^2}{9M} -\frac{\left(\nu-\frac{N}{3M}\right)^2}{(48\kappa_0)^2} \ge \frac{72\kappa_0(1-2\kappa_0)}{9M^2(48\kappa_0)^2},
\end{aligned}
\end{equation}
where $M = 9\kappa_0^2+18\kappa_0+1$ and $N = 18\kappa_0+21\kappa_0-1$. Since $M > 0$, $(1-3a) > 0$ and $\kappa_0(1-2\kappa_0)>0$ for $0 < \kappa_0 < \frac{2-\sqrt{3}}{3}$, the hyperbola (denoted by $b = \Gamma_3(\nu)$) derived by \eqref{hyperbola3} (also only upper part is available) can be rotated to a hyperbola of the form $y^2/a^2 - x^2/b^2 = c^2$. Hence, it suffices to show that $b = \Gamma_3(\nu)$ is always located above $\nu = \wt{s}_-$ on the line $\nu =1$. A computation yields
\[\frac{3}{2\left((3\kappa_0+1) - \sqrt{9\kappa_0^2-12\kappa_0+1}\right)} < \frac{(4-3\kappa_0)+\sqrt{9\kappa_0^2-48\kappa_0+16}}{48\kappa_0},\]
for $<\kappa_0<\frac{2-\sqrt{3}}{3}$.
\end{proof}

\begin{remark}\label{rem:b_0>b_2>b_1}
Remark that $b_0 > b_1$ for $0 < \kappa_0 < \frac14$ and $b_0 > b_2 > b_1$ for $0 < \kappa_0 < \frac{2-\sqrt{3}}{3}$, where $b_0$, $b_1$ and $b_2$ are as in \eqref{b_0}, \eqref{b_1} and \eqref{b_2}, respectively.
\end{remark}

Recall the sets $\mathcal R_1(b_0)$, $\mathcal R_2(b_0)$, $\mathcal R_3(b_0)$ and $\mathcal R_\#(b_0)$ defined in Definition \ref{Uni_Dis_Par}:
\[\mathcal R_1(b_0):=\left\{(\nu, b) \in \mathcal R_0 ~:~  9b_0\nu^2 - 6b_0\nu \le 12b^2 - \left(12b_0+1\right)b  \right\},\]
\smallskip
\[\mathcal R_2(b_0):=\left\{(\nu, b) \in \mathcal R_0 ~:~ 40b^2 - 8(3b_0+1)b + 4(2-9b_0)\nu b \ge 15b_0 (3\nu^2-2\nu)  \right\},
\]
\smallskip
\[\mathcal R_3(b_0):=\left\{(\nu, b) \in \mathcal R_0 ~:~ 120b^2 -8(18b_0+1)b + 12(9b_0-2)\nu b \ge 45b_0(3\nu^2 - 2\nu) \right\},
\]
and
\[
\begin{aligned}
\mathcal R_\#(b_0):= &~ \Bigg\{ (a,b,c) =\left(-\frac{\nu}{2}- \frac13 - b, \; b, \; \frac{\nu}{2} - b \right) : \\
&~ \qquad  \qquad \qquad b \ge b_0, ~ \nu \in \mathcal R_1(b_0) \cap \mathcal R_2(b_0) \cap \mathcal R_3(b_0) \Bigg\}.
\end{aligned}
\]
Note that all the sets above are well-defined (they are not empty) if $b$ is sufficiently large. See also Fig. \ref{Fig:0} for the understanding of $\mathcal R_0$. 

\medskip

The key lemma in this chapter is the following one. Here $\mathcal R_\#(b_0)$ appears.

\begin{lemma}\label{lem:alpha existence}
Let $0 < \kappa_0 < \frac14$ and $b_0$ be given in \eqref{b_0}. Let $\bar{A}_i$ and $\bar{B}_i$ be given in \eqref{A2A3A4} and \eqref{B2B3B4}, respectively. 
Then, 
\begin{equation}
\max(\bar{B}_2, \bar{B}_3,\bar{B}_4) \le \min(\bar{A}_2, \bar{A}_3,\bar{A}_4),
\end{equation}
if a triple $(a,b,c)$ belongs to the set $\mathcal R_\#(b_0)$ defined in \eqref{eq:set}.
\end{lemma}

\begin{proof}
Thanks to Remark \ref{rem:b_0>b_2>b_1}, we have that Lemmas \ref{lem:rs ordering} and \ref{lem:lower} are valid. We split the proof into two parts in terms of the range of $\kappa_0$.

\medskip

\medskip

\textbf{Case I.} $\frac{2-\sqrt{3}}{3} \le \kappa_0 < \frac14$. Lemma \ref{lem:rs ordering} ensures that \eqref{r ordering1} or \eqref{r ordering2} happens (see Figure \ref{fig:min-max2}), and then, Propositions \ref{prop:min} and \ref{prop:max} yield
\begin{enumerate}
\item $\max(\bar{B}_2, \bar{B}_3,\bar{B}_4) = \bar{B}_4$ and $\min(\bar{A}_2, \bar{A}_3,\bar{A}_4) = \bar{A}_4$, if $\nu \le r_-$.
\smallskip
\item $\max(\bar{B}_2, \bar{B}_3,\bar{B}_4) = \bar{B}_4$ and $\min(\bar{A}_2, \bar{A}_3,\bar{A}_4) = \bar{A}_3$, if $r_- \le \nu \le r_+$.
\smallskip
\item $\max(\bar{B}_2, \bar{B}_3,\bar{B}_4) = \bar{B}_4$ and $\min(\bar{A}_2, \bar{A}_3,\bar{A}_4) = \bar{A}_4$, if $r_+ \le \nu \le \wt{r}_-$.
\smallskip
\item $\max(\bar{B}_2, \bar{B}_3,\bar{B}_4) = \bar{B}_3$ and $\min(\bar{A}_2, \bar{A}_3,\bar{A}_4) = \bar{A}_4$, if $\wt{r}_- \le \nu \le \wt{r}_+$.
\smallskip
\item $\max(\bar{B}_2, \bar{B}_3,\bar{B}_4) = \bar{B}_4$ and $\min(\bar{A}_2, \bar{A}_3,\bar{A}_4) = \bar{A}_4$, if $\wt{r}_+ < \nu$.
\end{enumerate}
and
\begin{enumerate}
\item $\max(\bar{B}_2, \bar{B}_3,\bar{B}_4) = \bar{B}_4$ and $\min(\bar{A}_2, \bar{A}_3,\bar{A}_4) = \bar{A}_4$, if $\nu \le r_-$.
\smallskip
\item $\max(\bar{B}_2, \bar{B}_3,\bar{B}_4) = \bar{B}_4$ and $\min(\bar{A}_2, \bar{A}_3,\bar{A}_4) = \bar{A}_3$, if $r_- \le \nu \le \wt{r}_-$.
\smallskip
\item $\max(\bar{B}_2, \bar{B}_3,\bar{B}_4) = \bar{B}_3$ and $\min(\bar{A}_2, \bar{A}_3,\bar{A}_4) = \bar{A}_3$, if $\wt{r}_- \le \nu \le r_+$.
\smallskip
\item $\max(\bar{B}_2, \bar{B}_3,\bar{B}_4) = \bar{B}_3$ and $\min(\bar{A}_2, \bar{A}_3,\bar{A}_4) = \bar{A}_4$, if $r_+ \le \nu \le \wt{r}_+$.
\smallskip
\item $\max(\bar{B}_2, \bar{B}_3,\bar{B}_4) = \bar{B}_4$ and $\min(\bar{A}_2, \bar{A}_3,\bar{A}_4) = \bar{A}_4$, if $\wt{r}_+ < \nu$.
\end{enumerate}
\medskip

From Lemma \ref{lem:lower}, we know that there is no possible $(\nu, b) \in \mathcal R_0$ for which the cases of Items (1) and (5) both happen. If $(\nu,b)$ belongs to the region $r_+ \le \nu \le \wt{r}_-$, for $b > b_0$, it satisfies $ \bar{A}_3 \ge \bar{B}_3$. Indeed, the intersection point of $\nu=r_+$ and $\nu=\wt{r}_-$ is
\[
(\nu,b) = \left(\frac13,\frac{1}{2\left((2\kappa_0+1)-\sqrt{(2\kappa_0+1)^2-6\kappa_0} \right)}\right),
\]
which is located above $(\nu,b) = (\frac13,b_0)$, the minimum point of $b = \Gamma_2(\nu)$, and where $\Gamma_2(\nu)$ was the hyperbola introduced in the proof of Lemma \ref{lem:lower}. Moreover, one can  check that
\[\frac{1}{(2\kappa_0+1)-\sqrt{(2\kappa_0+1)^2-6\kappa_0}} < \Gamma_2(0)\]
and
\[\frac{3}{2\left((2\kappa_0+1)-\sqrt{(2\kappa_0+1)^2-6\kappa_0}\right)} < \Gamma_2(1).\]
Then, since each region described by Items (2) and (4) corresponds to $\mathcal R_2 (b_0)$ and $\mathcal R_3 (b_0)$, respectively, we complete the proof.
 
\begin{figure}[h!]
\begin{center}
\begin{tikzpicture}[scale=0.8]
\draw[thick] (-5,3) -- (5,3);
\node at (-3,3){$\bullet$};
\node at (-1,3){$\bullet$};
\node at (1,3){$\bullet$};
\node at (3,3){$\bullet$};
\node at (-2.9,2.5){$r_-$};
\node at (-0.9,2.5){$r_+$};
\node at (1.1,2.5){$\wt{r}_-$};
\node at (3.1,2.5){$\wt{r}_+$};
\node at (5,4.5){\tiny{$\frac{2-\sqrt{3}}{3} \le \kappa_0 < \frac14$}};
\node at (-5,4.5){\tiny{$\min(\bar{A}_2,\bar{A}_3,\bar{A}_4)$}};
\draw (-5,3) arc (150:100:1);
\node at (-4,3.5){$\bar{A}_4$};
\draw (-3,3) arc (30:80:1);
\draw (-3,3) arc (150:100:1);
\node at (-2,3.5){$\bar{A}_3$};
\draw (-1,3) arc (30:80:1);
\draw (-1,3) arc (150:100:1);
\node at (0,3.5){$\bar{A}_4$};
\draw (1,3) arc (30:80:1);
\draw (1,3) arc (150:100:1);
\node at (2,3.5){$\bar{A}_4$};
\draw (3,3) arc (30:80:1);
\draw (3,3) arc (150:100:1);
\node at (4,3.5){$\bar{A}_4$};
\draw (5,3) arc (30:80:1);
\node at (-5,1.5){\tiny{$\max(\bar{B}_2,\bar{B}_3,\bar{B}_4)$}};
\draw (-5,3) arc (-150:-100:1);
\node at (-4,2.5){$\bar{B}_4$};
\draw (-3,3) arc (-30:-80:1);
\draw (-3,3) arc (-150:-100:1);
\node at (-2,2.5){$\bar{B}_4$};
\draw (-1,3) arc (-30:-80:1);
\draw (-1,3) arc (-150:-100:1);
\node at (0,2.5){$\bar{B}_4$};
\draw (1,3) arc (-30:-80:1);
\draw (1,3) arc (-150:-100:1);
\node at (2,2.5){$\bar{B}_3$};
\draw (3,3) arc (-30:-80:1);
\draw (3,3) arc (-150:-100:1);
\node at (4,2.5){$\bar{B}_4$};
\draw (5,3) arc (-30:-80:1);
\end{tikzpicture}

\begin{tikzpicture}[scale=0.8]
\draw[thick] (-5,3) -- (5,3);
\node at (-3,3){$\bullet$};
\node at (-1,3){$\bullet$};
\node at (1,3){$\bullet$};
\node at (3,3){$\bullet$};
\node at (-2.9,2.5){$r_-$};
\node at (-0.9,2.5){$\wt{r}_-$};
\node at (1.1,2.5){$r_+$};
\node at (3.1,2.5){$\wt{r}_+$};
\node at (5,4.5){\tiny{$\frac{2-\sqrt{3}}{3} \le \kappa_0 < \frac14$}};
\node at (-5,4.5){\tiny{$\min(\bar{A}_2,\bar{A}_3,\bar{A}_4)$}};
\draw (-5,3) arc (150:100:1);
\node at (-4,3.5){$\bar{A}_4$};
\draw (-3,3) arc (30:80:1);
\draw (-3,3) arc (150:100:1);
\node at (-2,3.5){$\bar{A}_3$};
\draw (-1,3) arc (30:80:1);
\draw (-1,3) arc (150:100:1);
\node at (0,3.5){$\bar{A}_3$};
\draw (1,3) arc (30:80:1);
\draw (1,3) arc (150:100:1);
\node at (2,3.5){$\bar{A}_4$};
\draw (3,3) arc (30:80:1);
\draw (3,3) arc (150:100:1);
\node at (4,3.5){$\bar{A}_4$};
\draw (5,3) arc (30:80:1);
\node at (-5,1.5){\tiny{$\max(\bar{B}_2,\bar{B}_3,\bar{B}_4)$}};
\draw (-5,3) arc (-150:-100:1);
\node at (-4,2.5){$\bar{B}_4$};
\draw (-3,3) arc (-30:-80:1);
\draw (-3,3) arc (-150:-100:1);
\node at (-2,2.5){$\bar{B}_4$};
\draw (-1,3) arc (-30:-80:1);
\draw (-1,3) arc (-150:-100:1);
\node at (0,2.5){$\bar{B}_3$};
\draw (1,3) arc (-30:-80:1);
\draw (1,3) arc (-150:-100:1);
\node at (2,2.5){$\bar{B}_3$};
\draw (3,3) arc (-30:-80:1);
\draw (3,3) arc (-150:-100:1);
\node at (4,2.5){$\bar{B}_4$};
\draw (5,3) arc (-30:-80:1);
\end{tikzpicture}
\end{center}
\caption{Minimum and maximum of $\bar{A}_i$ and $\bar{B}_i$, $i=2,3,4$, with respect to the position of $\nu$, for $\frac{2-\sqrt{3}}{3} \le \kappa_0 < \frac14$. (Above) when $r_+ < \wt{r}_-$ and (below) when $\wt{r}_- < r_+$. 
} \label{fig:min-max2}
\end{figure}

\medskip

\textbf{Case II.} $0 < \kappa_0 < \frac{2-\sqrt{3}}{3}$. For $b \ge b_0 > b_1$, we know from Lemmas \ref{lem:rs ordering} and \ref{lem:lower} (2) that the cases
\[s_+< r_+ < \wt{r}_-<\wt{s}_-<\wt{s}_+\]
and
\[s_+< \wt{r}_- < r_+<\wt{s}_-<\wt{s}_+\]
must happen, and hence from Propositions \ref{prop:min} and \ref{prop:max} we have
\begin{enumerate}
\item $\max(\bar{B}_2, \bar{B}_3,\bar{B}_4) = \bar{B}_4$ and $\min(\bar{A}_2, \bar{A}_3,\bar{A}_4) = \bar{A}_3$, if $s_+ \le \nu \le r_+$.
\smallskip
\item $\max(\bar{B}_2, \bar{B}_3,\bar{B}_4) = \bar{B}_4$ and $\min(\bar{A}_2, \bar{A}_3,\bar{A}_4) = \bar{A}_4$, if $r_+ \le \nu \le \wt{r}_-$.
\smallskip
\item $\max(\bar{B}_2, \bar{B}_3,\bar{B}_4) = \bar{B}_3$ and $\min(\bar{A}_2, \bar{A}_3,\bar{A}_4) = \bar{A}_4$, if $\wt{r}_- \le \nu \le \wt{s}_-$.
\smallskip
\item $\max(\bar{B}_2, \bar{B}_3,\bar{B}_4) = \bar{B}_2$ and $\min(\bar{A}_2, \bar{A}_3,\bar{A}_4) = \bar{A}_4$, if $\wt{s}_- \le \nu \le \wt{s}_+$,
\end{enumerate}
or
\begin{enumerate}
\item $\max(\bar{B}_2, \bar{B}_3,\bar{B}_4) = \bar{B}_4$ and $\min(\bar{A}_2, \bar{A}_3,\bar{A}_4) = \bar{A}_3$, if $s_+ \le \nu \le \wt{r}_-$.
\smallskip
\item $\max(\bar{B}_2, \bar{B}_3,\bar{B}_4) = \bar{B}_3$ and $\min(\bar{A}_2, \bar{A}_3,\bar{A}_4) = \bar{A}_3$, if $\wt{r}_- \le \nu \le r_+$.
\smallskip
\item $\max(\bar{B}_2, \bar{B}_3,\bar{B}_4) = \bar{B}_3$ and $\min(\bar{A}_2, \bar{A}_3,\bar{A}_4) = \bar{A}_4$, if $r_+ \le \nu \le \wt{s}_-$.
\smallskip
\item $\max(\bar{B}_2, \bar{B}_3,\bar{B}_4) = \bar{B}_2$ and $\min(\bar{A}_2, \bar{A}_3,\bar{A}_4) = \bar{A}_4$, if $\wt{s}_- \le \nu \le \wt{s}_+$.
\end{enumerate}

\medskip

Moreover, Lemma \ref{lem:lower} (4) enables us to drop the last case in both situations (see Figure \ref{fig:min-max3}). 
Similarly as \textbf{Case I}, the remaining cases exactly describe $\mathcal R_{\#}(b_0)$, thus we complete this case.

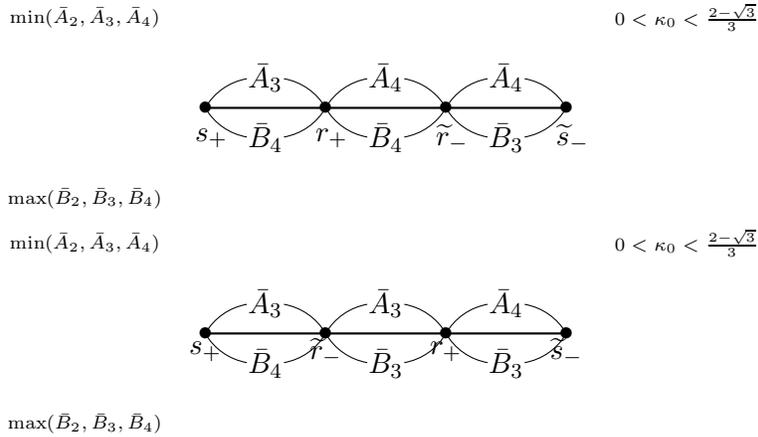
\begin{figure}[h!]
\begin{center}
\begin{tikzpicture}[scale=0.8]
\draw[thick] (-3,3) -- (3,3);
\node at (-3,3){$\bullet$};
\node at (-1,3){$\bullet$};
\node at (1,3){$\bullet$};
\node at (3,3){$\bullet$};
\node at (-2.9,2.5){$s_+$};
\node at (-0.9,2.5){$r_+$};
\node at (1.1,2.5){$\wt{r}_-$};
\node at (3.1,2.5){$\wt{s}_-$};
\node at (5,4.5){\tiny{$0 < \kappa_0 < \frac{2-\sqrt{3}}{3}$}};
\node at (-5,4.5){\tiny{$\min(\bar{A}_2,\bar{A}_3,\bar{A}_4)$}};
\draw (-3,3) arc (150:100:1);
\node at (-2,3.5){$\bar{A}_3$};
\draw (-1,3) arc (30:80:1);
\draw (-1,3) arc (150:100:1);
\node at (0,3.5){$\bar{A}_4$};
\draw (1,3) arc (30:80:1);
\draw (1,3) arc (150:100:1);
\node at (2,3.5){$\bar{A}_4$};
\draw (3,3) arc (30:80:1);
\node at (-5,1.5){\tiny{$\max(\bar{B}_2,\bar{B}_3,\bar{B}_4)$}};
\draw (-3,3) arc (-150:-100:1);
\node at (-2,2.5){$\bar{B}_4$};
\draw (-1,3) arc (-30:-80:1);
\draw (-1,3) arc (-150:-100:1);
\node at (0,2.5){$\bar{B}_4$};
\draw (1,3) arc (-30:-80:1);
\draw (1,3) arc (-150:-100:1);
\node at (2,2.5){$\bar{B}_3$};
\draw (3,3) arc (-30:-80:1);
\end{tikzpicture}

\begin{tikzpicture}[scale=0.8]
\draw[thick] (-3,3) -- (3,3);
\node at (-3,3){$\bullet$};
\node at (-1,3){$\bullet$};
\node at (1,3){$\bullet$};
\node at (3,3){$\bullet$};
\node at (-3,2.7){$s_+$};
\node at (-1,2.7){$\wt{r}_-$};
\node at (1,2.7){$r_+$};
\node at (3,2.7){$\wt{s}_-$};
\node at (5,4.5){\tiny{$0 < \kappa_0 < \frac{2-\sqrt{3}}{3}$}};
\node at (-5,4.5){\tiny{$\min(\bar{A}_2,\bar{A}_3,\bar{A}_4)$}};
\draw (-3,3) arc (150:100:1);
\node at (-2,3.5){$\bar{A}_3$};
\draw (-1,3) arc (30:80:1);
\draw (-1,3) arc (150:100:1);
\node at (0,3.5){$\bar{A}_3$};
\draw (1,3) arc (30:80:1);
\draw (1,3) arc (150:100:1);
\node at (2,3.5){$\bar{A}_4$};
\draw (3,3) arc (30:80:1);
\node at (-5,1.5){\tiny{$\max(\bar{B}_2,\bar{B}_3,\bar{B}_4)$}};
\draw (-3,3) arc (-150:-100:1);
\node at (-2,2.5){$\bar{B}_4$};
\draw (-1,3) arc (-30:-80:1);
\draw (-1,3) arc (-150:-100:1);
\node at (0,2.5){$\bar{B}_3$};
\draw (1,3) arc (-30:-80:1);
\draw (1,3) arc (-150:-100:1);
\node at (2,2.5){$\bar{B}_3$};
\draw (3,3) arc (-30:-80:1);
\end{tikzpicture}
\end{center}
\caption{Minimum and maximum of $\bar{A}_i$ and $\bar{B}_i$, $i=2,3,4$, with respect to the position of $\nu$, for $0 < \kappa_0 < \frac{2-\sqrt{3}}{3}$, in particular, (above) when $r_+ < \wt{r}_-$ and (below) when $\wt{r}_- < r_+$.} \label{fig:min-max3}
\end{figure}

\medskip

We end the proof with the following remarks\footnote{One can check them via simple computations.}:
\begin{enumerate}
\item The hyperbolas determined by $\bar{A}_3=\bar{B}_4$ and $\bar{A}_3=\bar{B}_3$, and the line $\nu = \wt{r}_-$, intersect at only one point.
\smallskip
\item The hyperbolas determined by $\bar{A}_4=\bar{B}_3$ and $\bar{A}_3=\bar{B}_3$, and the line $\nu = r_+$, intersect at only one point.
\end{enumerate}
This, indeed, shows that $\mathcal R_\#(b_0)$ is equivalent to the union of all $(\nu,b)$ obtained from each case above described. 
\end{proof}

\subsection{Positivity of the quadratic form $\mathcal{Q}_{\lambda(t)}(t)$}\label{sec:TH3}

As a consequence of all the analysis above performed, one has 
\begin{proposition}\label{prop:CONDITION}
Let $v_0 \in (0,1)$ be given. Let $v_0^+$ and $\kappa_0$ be defined as in \eqref{v_0^+} and \eqref{k_0}, respectively. Let $b_0$ be defined as in \eqref{b_0}. Let $\vp$ satisfy $|\vp^{(n)}| \lesssim |\vp'|$, $n \ge 1$. Then, there exists $\alpha   \in \R$, such that  
\begin{equation}\label{Q_est}
\begin{aligned}
\mathcal{Q}_{\lambda(t)}(t) =&~{} \frac{v_0^+}{2\lambda(t)} \int \vp' \left(\frac{x-vt}{\lambda(t)} \right) (u^2+\eta^2 + u_x^2 + \eta_x^2) \\
&+ O\left(\frac{1}{\lambda(t)^2}\int \vp' \left(\frac{x-vt}{\lambda(t)} \right) (u^2+\eta^2 + u_x^2 + \eta_x^2)\right),
\end{aligned}
\end{equation}
provided the triple $(a,b,c)$ belongs to the set $\mathcal R_\#(b_0)$ defined in \eqref{eq:set}.
\end{proposition}

\begin{remark}\label{rem:choice}
The leading coefficient $v_0^+$ in Proposition \ref{prop:CONDITION} can be freely chosen in $(v_0, 1)$ in order to find a probably smaller $b_0$ if possible. However, the free parameter $v_0^+$ also forces Theorem \ref{TH2} under a $v_0^+$-dependent smallness condition $\ve = \ve(v_0,b_0)$. Hence, in order to avoid this problem, we fix $v_0^+$ as in \eqref{v_0^+}.
\end{remark}

\begin{remark}
Proposition \ref{prop:CONDITION} with the choice of $v_0 = 0$ (as an asymptotic limit) corresponds to the result in \cite{KMPP2018} (or Theorem \ref{Thm2}) in the sense that we have $b > \frac29$ (see \eqref{k_0} and \eqref{b_0}). This is because all positive coefficients ($A_i, B_i > 0$) are needed for \eqref{Q_est}, while the virial estimate for Theorem \ref{TH1} allows some of negative coefficients. 
\end{remark}

\begin{proof}[Proof of Proposition \ref{prop:CONDITION}]
Thanks to \eqref{eq:L2_est}, the third terms in the right-hand side of \eqref{Q} are in
\[
O\left(\frac{1}{\lambda(t)^2}\int \vp' \left(\frac{x-vt}{\lambda(t)} \right) (u^2+\eta^2 + u_x^2 + \eta_x^2)\right).
\]
Then, we can rewrite $\mathcal{Q}_{\lambda(t)}(t)$ as
\be\label{Q_est1}
\begin{aligned}
\mathcal{Q}_{\lambda(t)}(t) =&~{} \frac{v_0^+}{2\lambda(t)} \int \vp' \left(\frac{x-vt}{\lambda(t)} \right) (f^2 + 3f_x^2 + 3f_{xx}^2 + f_{xxx}^2) \\
&+\frac{v_0^+}{2\lambda(t)} \int \vp' \left(\frac{x-vt}{\lambda(t)} \right) (g^2 + 3g_x^2 + 3g_{xx}^2 + g_{xxx}^2)\\
&+\frac{1}{\lambda(t)}\int \vp'\left(\frac{x-vt}{\lambda(t)} \right) \Big( A_1^* f^2 + A_2^* f_x^2 + A_3^* f_{xx}^2 + A_4^* f_{xxx}^2\Big)\\
&+\frac{1}{\lambda(t)}\int \vp' \left(\frac{x-vt}{\lambda(t)} \right)\Big( B_1^* g^2 + B_2^* g_x^2 + B_3^* g_{xx}^2 + B_4^* g_{xxx}^2\Big)\\
&+ O\left(\frac{1}{\lambda(t)^2}\int \vp' \left(\frac{x-vt}{\lambda(t)} \right) (u^2+\eta^2 + u_x^2 + \eta_x^2)\right),
\end{aligned}
\ee
where $A_i^*$ and $B_i^*$, $i=1,2,3,4$, are defined as in \eqref{AB1}--\eqref{AB4}. Note that $A_1^*, B_1^* > 0$. If 
\begin{equation}\label{Q_est2}
(a,b,c) \in \mathcal R_{\#}(b_0) \quad \Longleftrightarrow \quad A_i^*, \; B_i^* \ge 0, \; i=2,3,4,
\end{equation}
holds true, we know that \eqref{Q_est1} implies \eqref{Q_est} thanks to \eqref{eq:H1}. Moreover, Lemma \ref{lem:alpha existence}  shows \eqref{Q_est2}, and thus we complete the proof.
\end{proof}


\subsection{Sequentially-in-time decay}
Let $\la=\lambda(t)$ be the time-dependent function given in  \eqref{eq:lambda000} satisfying \eqref{eq:lambda0001}.

\begin{proposition}[Decay in time-dependent intervals]\label{prop:FULL Region}
Let $0 < v_0 < 1$ be given. Let $v_0^+$ and $\kappa_0$ be defined as in \eqref{v_0^+} and \eqref{k_0}, respectively. Let $(a,b,c)$ be a triple of parameters in the set \eqref{eq:set} and $v \in \R$ be a fixed constant with $|v| < v_0$.  Then, there exists $\varepsilon_0 = \varepsilon_0(v_0)>0$ such that for $H^1 \times H^1$ global solutions $(u,\eta)(t)$ to \eqref{boussinesq0000} satisfying 
\[
\norm{(u,\eta)(t=0)}_{H^1 \times H^1} < \varepsilon,
\]
for $0 < \varepsilon \le \varepsilon_0$, we have
\begin{equation}\label{Decay1}
\int_2^{\infty}\frac{1}{\lambda(t)}\int \sech^2 \left( \frac{x-vt}{\lambda(t)} \right)\left(u^2 + \eta^2 + u_x^2 + \eta_x^2 \right)(t,x) \; dx dt \lesssim \varepsilon^2.
\end{equation}
It  implies that there exists an increasing sequence of time $\{t_n\}$ $(t_n \to \infty$ as $n \to \infty)$ such that
\begin{equation}\label{Decay2}
\int \sech^2 \left(\frac{x-vt_n}{\lambda(t_n)}\right) \left(u^2 + (\px u)^2 + \eta^2 + (\px \eta)^2\right)(t_n,x) \; dx \longrightarrow 0,
\end{equation}
as $n \to \infty$.
\end{proposition}

\begin{proof}
We choose, in \eqref{eq:refined gvirial} (after replacing the parameters $a, c$ by $a/b, c/b$), the weight $\vp = \tanh$ with $\la(t)$ given by \eqref{eq:lambda000}. It  follows from the proof of Proposition 5.1 in \cite{KMPP2018} that
\begin{equation}\label{SQ_est}
|\mathcal{SQ}_{\lambda(t)}(t)| \lesssim \frac{1}{\lambda(t)^2}\int \vp' \left( \frac{x-vt}{\lambda(t)} \right)\left(u^2 + \eta^2 + u_x^2 + \eta_x^2 \right)
\end{equation}
and
\begin{equation}\label{NQ_est}
|\mathcal{NQ}_{\lambda(t)}(t)| \lesssim \frac{\norm{u}_{H^1} + \norm{\eta}_{H^1}}{\lambda(t)}\int \vp' \left( \frac{x-vt}{\lambda(t)} \right)\left(u^2 + \eta^2 + u_x^2 + \eta_x^2 \right).
\end{equation}
On the other hand, the Young's inequality and the Cauchy-Schwarz inequality in addition to \eqref{eq:lambda0001} yield
\begin{equation}\label{VH_est}
\begin{aligned}
\Big|\frac{\lambda'(t)}{\lambda(t)}&\int \frac{x-vt}{\lambda(t)}\sech^2 \left(\frac{x-vt}{\lambda(t)} \right)fg \Big| \\
\le&~{}\frac{\norm{f}_{L^2}\norm{g}_{L^2}}{t\log^{\frac32} t}+ \frac{1}{4\lambda(t)\log^{\frac12} t} \int \sech^2\left(\frac{x-vt}{\lambda(t)} \right) (f^2 + g^2),
\end{aligned}
\end{equation}
for $t > e^2$. Applying \eqref{VH_est} to the right-hand side of \eqref{VH}, one has
\begin{equation}\label{VH_est1}
\begin{aligned}
|\mathcal{VH}_{\lambda(t)}| \le&~{} \frac{|v|}{2\lambda(t)} \int \sech^2 \left(\frac{x-vt}{\lambda(t)} \right) (u^2 + \eta^2 + u_x^2 + \eta_x^2) \\
&+\frac{1}{\log^{\frac12} t}\frac{1}{\lambda(t)}\int \sech^2 \left(\frac{x-vt}{\lambda(t)} \right) (u^2 + \eta^2 + u_x^2 + \eta_x^2)\\
&+\frac{3\varepsilon^2}{t\log^{\frac32} t}.
\end{aligned}
\end{equation}
Collecting \eqref{Q_est}, \eqref{SQ_est}, \eqref{NQ_est} and \eqref{VH_est1}, we have
\[\begin{aligned}
\frac{3\varepsilon^2}{t\log^{\frac32} t} + \frac{d}{dt} \mathcal{H} \ge&~{}\frac{v_0^+ - |v|}{2\lambda(t)} \int \sech^2 \left(\frac{x-vt}{\lambda(t)} \right) (u^2 + \eta^2 + u_x^2 + \eta_x^2)\\
&-\frac{C_1}{\log^{\frac12} t}\frac{1}{\lambda(t)}\int \sech^2 \left(\frac{x-vt}{\lambda(t)} \right) (u^2 + \eta^2 + u_x^2 + \eta_x^2)\\
&-\frac{C_2\varepsilon_0}{\lambda(t)}\int \sech^2 \left(\frac{x-vt}{\lambda(t)} \right) (u^2 + \eta^2 + u_x^2 + \eta_x^2),
\end{aligned}\]
where $C_1, C_2 >0$ are implicit constants independent of $t$ and $\varepsilon$. 

\medskip

Note that, for any $|v| \le v_0$, the coefficient $\frac{v_0^+ - |v|}{2}$ of the local $H^1$-norm is always bounded below by  $\frac{v_0^+ -v_0}{2}$. This enables us to obtain the strong decay property uniformly in $v$.

\medskip

Take $\varepsilon_0 >0$ and $t_0 >0$ such that
\[\frac{C_1}{\log^{\frac12} t} < \frac{v_0^+-v_0}{8}, \quad t \ge t_0 \quad \mbox{and} \quad C_2\varepsilon_0 < \frac{v_0^+-v_0}{8}.\]
Therefore, we conclude for $\varepsilon \le \varepsilon_0$ and $t \ge t_0$ that
\[\frac{1}{\lambda(t)}\int \sech^2 \left( \frac{x-vt}{\lambda(t)} \right)\left(u^2 + \eta^2 + u_x^2 + \eta_x^2 \right) \lesssim \frac{\varepsilon^2}{t\log^{\frac32} t} + \frac{d}{dt} \mathcal{H},\]
which implies \eqref{Decay1} thanks to the fact that $\frac{1}{t\log^pt}$ is integrable on $[2,\infty)$ when $p>1$, and \eqref{H_est} in addition to \eqref{Energy}. The standard limiting process ensures that \eqref{Decay1} implies \eqref{Decay2}.
\end{proof}

\subsection{Local energy estimates}
We recall from \cite{KMPP2018} the localized energy functional (corresponding to \eqref{boussinesq0000}) defined by
\begin{equation}\label{eq:local energy}
E_{loc}(t) = \frac12 \int \psi(x) \left( - \frac ab(\px u)^2 - \frac cb(\px \eta)^2 + u^2 + \eta^2 + u^2\eta\right)(t,x)dx,
\end{equation}
and its variation
\begin{lemma}[Variation of local energy $E_{loc}$]\label{lem:energy1}
Let $u$ and $\eta$ satisfy \eqref{boussinesq0000}. Let $f$ and $g$ be canonical variables of $u$ and $\eta$ as in \eqref{eq:fg}. Then, the following hold.
\begin{enumerate}
\item Time derivative. We have
\begin{equation}\label{eq:energy1}
\begin{aligned}
\frac{d}{dt} E_{loc}(t) =&~{} \int \psi' fg + \left(1-\frac{2(a+c)}{b}\right) \int \psi' f_xg_x \\
&+ \frac{(3ac-2b(a+c))}{b^2}\int \psi'  f_{xx}g_{xx} + \frac{3ac}{b^2}\int \psi'  f_{xxx}g_{xxx}\\
& -\frac ab\int \psi''  f_xg  -\frac cb \int \psi'' fg_x\\
& + \frac{a(c-2b^2)}{b^2}\int \psi''  f_{xx}g_x+ \frac{c(a-2b^2)}{b^2}\int \psi'' f_xg_{xx}\\
&+SNL_1(t) + SNL_2(t) + SNL_3(t) + SNL_4(t).
\end{aligned}
\end{equation}
\item The small nonlinear parts $SNL_j(t)$ are given by
\begin{equation}\label{eq:SNL1}
SNL_1(t) := \frac{a}{2b} \int (\psi' u_x)_x\nlop(u^2) + \frac cb \int (\psi' \eta_x)_x\nlop(u\eta),
\end{equation}
\begin{equation}\label{eq:SNL2}
\begin{aligned}
SNL_2(t) :=&~{} \frac12 \int \psi' f \nlop (u^2) + \frac{a}{2b} \int \psi' f_{xx} \nlop (u^2)\\
&+ \int \psi' g \nlop (u \eta) + \frac cb \int \psi' g_{xx} \nlop (u \eta)\\
&+ \frac12\int \psi' f_x \nlop (u^2)_x +  \int \psi' g_x \nlop (u \eta)_x ,
\end{aligned}
\end{equation}
\begin{equation}\label{eq:SNL3}
SNL_3(t) := \frac{a}{2b} \int \psi' f_{xxx} \nlop (u^2)_x+ \frac cb \int \psi' g_{xxx} \nlop (u \eta)_x,
\end{equation}
and
\begin{equation}\label{eq:SNL4}
\begin{aligned}
SNL_4(t) :=&~{}\frac12 \int \psi' \nlop(u\eta) \nlop(u^2) \\
&+ \frac12 \int \psi' \nlop (u\eta)_x \nlop (u^2)_x.
\end{aligned}
\end{equation}
\end{enumerate}
\end{lemma}

Replacing the weight function $\psi(x)$ by the time-dependent function $\psi(\frac{x-vt}{\lambda(t)})$, where $\lambda(t)$ is defined in \eqref{eq:lambda000}, Lemma \ref{lem:energy1} extends to the case of a time-dependent, increasing (in the ``$x=vt$" axial) interval in space. 

\begin{lemma}[A refined variation of local energy $E_{loc}$]\label{lem:energy2}
Let $u$ and $\eta$ satisfy \eqref{boussinesq0000}. Let $f$ and $g$ be canonical variables of $u$ and $\eta$ as in \eqref{eq:fg}. Then, we have
\begin{equation}\label{eq:energy2}
\begin{aligned}
\frac{d}{dt} E_{loc}(t) =&~{} \frac12 \int \psi_t \left( - \frac ab(\px u)^2 - \frac cb(\px \eta)^2 + u^2 + \eta^2 + u^2\eta\right)\\
&+\mbox{RHS of } \eqref{eq:energy1},
\end{aligned}
\end{equation}
by replacing $\psi^{(n)}$ by $\frac{1}{\lambda(t)^{n}}\psi^{(n)}$.
\end{lemma}

\begin{remark}\label{rem:Derivative Weight}
We  have from \eqref{Derivative Weight} that
\begin{equation}\label{Derivative Weight1}
\frac{d}{dt} \psi\left( \frac{x-vt}{\lambda(t)}\right) = -\frac{v\lambda(t) + (x-vt)\lambda'(t)}{\lambda(t)^2}\psi'\left( \frac{x-vt}{\lambda(t)}\right).
\end{equation}
\end{remark}

\subsection{Proof of Theorem \ref{TH2}}
Now we take in \eqref{eq:energy2} $\psi = \sech^4$. The following Proposition finally proves Theorem \ref{TH2}.
\begin{proposition}\label{prop:energy2}
Let $0 < v_0 < 1$ be given. Let $v_0^+$ and $\kappa_0$ be defined as in \eqref{v_0^+} and \eqref{k_0}, respectively. Let $(a,b,c)$ be a triple of parameters in the set \eqref{eq:set} and $v \in \R$ be a fixed constant with $|v| < v_0$.  Then, there exists $\varepsilon_0 = \varepsilon_0(v_0)>0$ such that for $H^1 \times H^1$ global solutions $(u,\eta)(t)$ to \eqref{boussinesq0000} satisfying 
\[
\norm{(u,\eta)(t=0)}_{H^1 \times H^1} < \varepsilon,
\]
for $0 < \varepsilon \le \varepsilon_0$, we have
\begin{equation}\label{eq:energy2.1}
\lim_{t \to \infty} \int \sech^4 \left(\frac{x-vt}{\lambda(t)}\right) \left(u^2 + (\px u)^2 + \eta^2 + (\px \eta)^2\right)(t,x) \; dx = 0.
\end{equation}
\end{proposition}

\begin{proof}
The proof is almost identical to the proof of Proposition 7.2 in \cite{KMPP2018}, while the estimate of $\psi_t$ is slightly different. The proof of Proposition 7.2 in \cite{KMPP2018} gives
\[|\mbox{RHS of } \eqref{eq:energy1}| \lesssim \frac{1}{\lambda(t)}\int \sech^2\left( \frac{x-vt}{\lambda(t)}\right) \left(u^2 + \eta^2 + u_x^2 + \eta_x^2 \right).\]
Moreover, a computation in addition to \eqref{eq:lambda0001} gives
\[\begin{aligned}
|\psi_t| =&~{}\left| \frac{v\lambda(t) + (x-vt)\lambda'(t)}{\lambda(t)^2}\sech^4\left( \frac{x-vt}{\lambda(t)}\right)\tanh\left( \frac{x-vt}{\lambda(t)}\right) \right|\\
\lesssim&~{}\frac{1}{\lambda(t)}\sech^2\left( \frac{x-vt}{\lambda(t)}\right),
\end{aligned}\]
which in addition to the Sobolev inequality implies
\[\begin{aligned}
\int \psi_t \Big( - \frac ab(\px u)^2 - & \frac cb(\px \eta)^2 + u^2 + \eta^2 + u^2\eta\Big) \\
&\lesssim \frac{1}{\lambda(t)}\int \sech^2\left( \frac{x-vt}{\lambda(t)}\right) \left(u^2 + \eta^2 + u_x^2 + \eta_x^2 \right).
\end{aligned}\]
Collecting all, we obtain
\[\left|\frac{d}{dt} E_{loc}(t) \right| \lesssim  \frac{1}{\lambda(t)}\int \sech^2 \left(\frac{x-vt}{\lambda(t)}\right) \left(u^2 + (\px u)^2 + \eta^2 + (\px \eta)^2\right) (t,x) \; dx.\]
Integrating on $[t, t_n]$, for $t < t_n$ as in \eqref{Decay2}, and \eqref{Decay1} yield
\[\begin{aligned}
\Big|E_{loc}&(t_n) - E_{loc}(t) \Big|\\
 \lesssim &~ \int_t^{\infty}\!\!\frac{1}{\lambda(s)} \int \sech^2 \left(\frac{x-vs}{\lambda(s)}\right) (u^2 + (\px u)^2 + \eta^2 + (\px \eta)^2)(s,x) \; dxds  ~\lesssim~ \varepsilon^2.
\end{aligned}\]
Since
\[
|E_{loc}(t_n)| \lesssim \int \sech^2 \left(\frac{x-vt}{\lambda(t)}\right) (u^2 + (\px u)^2 + \eta^2 + (\px \eta)^2) (t_n) \longrightarrow 0, \quad \mbox{ as } n \to \infty,
\]
thanks to the Sobolev embedding (with $\norm{\eta}_{H^1} \lesssim \ve$) and \eqref{Decay2}. Thus, a straightforward conclusion using the smallness condition ensures \eqref{eq:energy2.1}.
\end{proof}

\subsection{Proof of Theorem \ref{TH0} (3) and (4)}
Proposition \ref{prop:energy2} (or Theorem \ref{TH2}) does not immediately imply the decay in $J_v(t)$ for $|v| < 1 - \frac{2}{9b}$, when $a=c= \frac16 -b$, due to our choice of $\kappa_0$ as 
\begin{equation}\label{eq:choice of b}
\kappa_0 = \frac{1}{9b_0} = \frac{1-v_0}{4}.
\end{equation}
A direct calculation in $b > b_0$ says 
\[|v| < v_0 < 1 - \frac{4}{9b}.\]
However, the choice of $\kappa_0$ (or $b_0$) is performed to prove the decay in $J_v(t)$, uniformly in $|v| < v_0$ (smallness independent only on $v_0$ as well), see Remark \ref{rem:choice}.

\medskip

Thus we end this section with the proof of Theorem \ref{TH0} (3) and (4), in particular, decay in $J_v(t)$ for $|v| < 1 - \frac{2}{9b}$, when $a=c= \frac16 -b$, for the completeness the results presented in Section \ref{sec:new results}. Recall \eqref{note1}
\begin{equation}\label{note1-0}
\begin{aligned}
&\int \vp' \Bigg( \frac12 (f^2 + g^2) + \Big(\frac32 -\frac{1}{4b}\Big) (f_x^2 + g_x^2) \\
& \qquad \qquad + \Big(\frac32 - \frac{1}{3b}\Big) (f_{xx}^2 + g_{xx}^2 )+ \Big(1-\frac{1}{12b}\Big) (f_{xxx}^2 + g_{xxx}^2)\Bigg),
\end{aligned}
\end{equation}
An immediate comparison in coefficients shows that $\frac32 - \frac{1}{3b}$ is the smallest one among other coefficients. Using \eqref{eq:H1_est}, one rewrites \eqref{note1-0} as
\[\frac13 \left(\frac32 - \frac{1}{3b}\right) \int \vp' \left(u^2 + \eta^2 + u_x^2 +  \eta_x^2 \right) + \mbox{(positive quantities)} + \mbox{(small terms)}. \]
An analogous argument in the proof of Proposition \ref{prop:FULL Region} ensures
\[\begin{aligned}
\frac{3\varepsilon^2}{t\log^{\frac32} t} + \frac{d}{dt} \mathcal{H} \ge&~{} \left(\left(1-\frac{2}{9b} - |v| \right)\right)\frac{1}{2\lambda(t)} \int \sech^2 \left(\frac{x-vt}{\lambda(t)} \right) (u^2 + \eta^2 + u_x^2 + \eta_x^2)\\
&-\frac{C_1}{\log^{\frac12} t}\frac{1}{\lambda(t)}\int \sech^2 \left(\frac{x-vt}{\lambda(t)} \right) (u^2 + \eta^2 + u_x^2 + \eta_x^2)\\
&-\frac{C_2\varepsilon_0}{\lambda(t)}\int \sech^2 \left(\frac{x-vt}{\lambda(t)} \right) (u^2 + \eta^2 + u_x^2 + \eta_x^2),
\end{aligned}\]
and hence, for given $|v| < 1-\frac{2}{9b}$, one conclude  Theorem \ref{TH0} (3) and (4), by taking $t \gg 1$ and $\varepsilon_0 \ll 1$ dependent on $v$ (to make the rest small enough). To extend this to a union of the finite number of $J_v(t)$, let $|v_n| < 1-\frac{2}{9b}$, $n = 1, 2, \cdots, N$, be given for fixed $N \ge 1$. Without loss of generality, we assume that $|v_{\ell}| < |v_n|$ if $\ell < n$. Then the smallness of data depends only on $v_N$, thus an $N$-sum of virial estimates allow us to conclude the desired result.

\bigskip

\section{Decay in exterior regions: Proof of Theorem \ref{TH3}}\label{sec:TH3a}

\medskip

\subsection{Preliminaries} Let $\sigma > 0$, to be fixed later. We choose 
\be\label{psi(t,x)}
\psi(t,x) = \psi\left(\frac{x-\sigma t}{L}\right)
\ee
in Lemma \ref{lem:energy2}, for a large $L \gg 1$ to be chosen below. We obtain first
\[
\begin{aligned}
&  \frac12 \int \psi_t \left( - \frac ab(\px u)^2 - \frac cb(\px \eta)^2 + u^2 + \eta^2 + u^2\eta\right) \\
& \qquad = -\frac{\sigma}{2L}\int \psi' \left( - \frac ab(\px u)^2 - \frac cb(\px \eta)^2 + u^2 + \eta^2 + u^2\eta\right).
\end{aligned}
\]
With this identity on one hand, Lemma \ref{lem:energy2} in addition to \eqref{eq:H1_est} yields
\begin{equation}\label{Loc_Energy}
\begin{aligned}
2L  \frac{d}{dt}E_{loc}(t) =&~{} - \sigma \int \psi' \left(f^2 + \left(2-\frac{a}{b} \right) f_x^2 + \left(1-\frac{2a}{b} \right) f_{xx}^2 + \left(-\frac{a}{b} \right) f_{xxx}^2 \right)\\
&- \sigma \int \psi' \left(g^2 + \left(2-\frac{c}{b} \right) g_x^2 + \left(1-\frac{2c}{b} \right) g_{xx}^2 + \left(-\frac{c}{b} \right) g_{xxx}^2\right)\\
&+2\int \psi' fg + \frac{2(15b-2)}{3b}\int \psi' f_xg_x\\
&+\frac{2(12b^2-2b+9ac)}{3b^2}\int \psi' f_{xx}g_{xx} +\frac{6ac}{b^2}\int \psi' f_{xxx}g_{xxx}\\
&+{\color{black} O\left(\frac1L, \norm{(u,\eta)}_{H^1 \times H^1}\right)}.
\end{aligned}
\end{equation}
Here, $O\left(\frac1L, \norm{(u,\eta)}_{H^1 \times H^1}\right)$ contains all small quadratic and higher degree terms, which are controlled by at least 
\[
\frac{1}{L}\int \psi'(u_x^2 + \eta_x^2 +u^2 + \eta^2) \quad \mbox{or} \quad \norm{(u,\eta)}_{H^1 \times H^1} \int \psi'(u_x^2 + \eta_x^2 +u^2 + \eta^2).
\]
These are again bounded by 
\[
\tilde c_0\int \psi'(u_x^2 + \eta_x^2 +u^2 + \eta^2),
\]
for a fixed (small) constant $\tilde c_0 > 0$, by taking appropriate $L \gg 1$ and $\norm{(u,\eta)}_{H^1 \times H^1} \ll 1$. See the proof of Proposition 7.1 in \cite{KMPP2018} for more details. Consequently, the $O\left(\frac1L, \norm{(u,\eta)}_{H^1 \times H^1}\right)$ terms will be discarded in the analysis below.

\medskip

We denote by $Q_{\sigma}(t)$ the right-hand side of \eqref{Loc_Energy}, except for $O\left(\frac1L, \norm{(u,\eta)}_{H^1 \times H^1}\right)$. Our aim in what follows is to show the bound
\begin{equation}\label{Coercivity0}
Q_{\sigma}(t) \le -c_0\int \psi'(u_x^2 + \eta_x^2 +u^2 + \eta^2),
\end{equation}
for some $c_0>\tilde c_0 > 0$. To complete the proof of Theorem \ref{TH3}, we will take 
\be\label{psi_final}
\psi := \frac12(1+\tanh)
\ee
in order to guarantee the positivity of the local energy. See Section \ref{sec:proof exterior} for more details. 

\medskip

On the other hand, the proof of decay in the left (negative) exterior region is analogous, while we replace $\sigma$ by $-\sigma$ and take $\psi = \frac12(1-\tanh)$. In what follows, we only focus on the {\bf right exterior region} ($\sigma > 0$ in \eqref{Loc_Energy}) for the analysis.



\subsection{A coercivity property} The key element for the proof of Theorem \ref{TH3} is the following coercivity result. 

\begin{lemma}\label{lem:Coercivity}
Let $(a,b,c)$ satisfy \eqref{Conds}.  Assume that 
\begin{equation}\label{Sigma}
\sigma > \max\left(1, \frac{15b-2}{3\sqrt{(2b-a)(2b-c)}}, \frac{12b^2-2b+9ac}{3b\sqrt{(b-2a)(b-2c)}}, \frac{3\sqrt{ac}}{b} \right).
\end{equation}
Then, for a sufficiently large $L \gg 1$ and small $\norm{(u,\eta)}_{H^1 \times H^1}$, we have \eqref{Coercivity0}, and consequently, the following coercivity property for \eqref{Loc_Energy}
\begin{equation}\label{Coercivity}
\begin{aligned}
2L  \frac{d}{dt}E_{loc}(t) \le&~{} - d_0 \int \psi' \left(u_x^2 + \eta_x^2 + u^2 + \eta^2\right).
\end{aligned}
\end{equation}
\end{lemma}

\begin{remark}
Theorem \ref{TH3} will be in part a related consequence of Lemma \ref{lem:Coercivity}, and a careful understanding of the condition \eqref{Sigma}.
\end{remark}

\begin{proof}[Proof of Lemma \ref{lem:Coercivity}]
We first deal with $f_{xxx}$ and $g_{xxx}$ terms in \eqref{Loc_Energy}. The Cauchy-Schwarz inequality in $\frac{6ac}{b^2}\int \psi' f_{xxx}g_{xxx}$ yields
\[\int \psi' \left|\frac{6ab}{b^2}f_{xxx}g_{xxx}\right| \le \int \psi' \left(\Lambda_{f,4} f_{xxx}^2 + \Lambda_{g,4} g_{xxx}^2\right),\]
where $\Lambda_{f,4}>0$ and $\Lambda_{g,4}>0$ satisfy
\begin{equation}\label{Lambda_fg_4}
\sqrt{\Lambda_{f,4} \Lambda_{g,4}} = \frac{3ac}{b^2} \quad \mbox{and} \quad \Lambda_{f,4} c = \Lambda_{g,4} a.
\end{equation}
Then, the terms containing $f_{xxx}$ and $g_{xxx}$ are reduced to
\[\begin{aligned}
&\frac{\sigma a}{b} \int \psi' f_{xxx}^2 + \frac{\sigma c}{b} \int \psi' g_{xxx}^2 +\frac{6ac}{b^2}\int \psi' f_{xxx}g_{xxx}\\
 &~{} \qquad \le \left(-\sigma \left(-\frac{a}{b} \right) + \Lambda_{f,4} \right) \int \psi' f_{xxx}^2 + \left(-\sigma \left(-\frac{c}{b} \right) + \Lambda_{g,4} \right) \int \psi' g_{xxx}^2.
\end{aligned}\]
In order to obtain \eqref{Coercivity0}, $\sigma$ should satisfy
\[\sigma > -\frac{\Lambda_{f,4} b}{a} = -\frac{\Lambda_{g,4} b}{c}.\]
Solving \eqref{Lambda_fg_4}, one finds
\[\Lambda_{f,4} = \frac{3(-a)^{\frac32}(-c)^{\frac12}}{b^2} \quad \mbox{and} \quad \Lambda_{g,4} = \frac{3(-a)^{\frac12}(-c)^{\frac32}}{b^2},\]
which implies
\[\sigma > \frac{3\sqrt{ac}}{b}.\]

\medskip

We use an analogous argument to deal with the other terms. The Cauchy-Schwarz inequality in $\frac{2(15b-2)}{3b}\int \psi' f_{x}g_{x}$ yields
\[\int \psi' \left|\frac{2(15b-2)}{3b} f_{x}g_{x}\right| \le \frac12 \int \psi' \left(\Lambda_{f,2} f_{x}^2 + \Lambda_{g,2} g_{x}^2\right),\]
where $\Lambda_{f,2}>0$ and $\Lambda_{g,2}>0$ satisfy
\begin{equation}\label{Lambda_fg_2}
\sqrt{\Lambda_{f,2} \Lambda_{g,2}} = \frac{15b-2}{3b} \quad \mbox{and} \quad \Lambda_{f,2} \left(2-\frac{c}{b} \right) = \Lambda_{g,2} \left(2-\frac{a}{b} \right).
\end{equation}
Solving \eqref{Lambda_fg_2}, one has
\[\Lambda_{f,2} = \frac{15b-2}{3b}\sqrt{\frac{2b-a}{2b-c}} \quad \mbox{and} \quad \Lambda_{g,2} = \frac{15b-2}{3b}\sqrt{\frac{2b-c}{2b-a}}.\]
The condition
\[\sigma > \Lambda_{f,2}\left(\frac{b}{2b-a}\right) = \Lambda_{g,2}\left(\frac{b}{2b-c}\right) = \frac{15b-2}{3\sqrt{(2b-a)(2b-c)}}\]
implies that the $f_x$, $g_x$ portions in \eqref{Loc_Energy} satisfies \eqref{Coercivity}.

\medskip

For the rest, the same argument ensures that if $\sigma$ satisfies the condition
\[\sigma > \Lambda_{f,3}\left(\frac{b}{b-2a}\right) = \Lambda_{g,3}\left(\frac{b}{b-2c}\right) = \frac{12b^2-2b+9ac}{3b\sqrt{(b-2a)(b-2c)}},\] 
one conclude \eqref{Coercivity} for $f_{xx}$, $g_{xx}$ portions, where $\Lambda_{f,3}$ and $\Lambda_{g,3}$ satisfy 
\[\sqrt{\Lambda_{f,2} \Lambda_{g,2}} = \frac{12b^2-2b+9ac}{3b^2} \quad \mbox{and} \quad \Lambda_{f,3} \left(1-\frac{2c}{b} \right) = \Lambda_{g,2} \left(1-\frac{2a}{b} \right).\]
The explicit formulas of $\Lambda_{f,3}$ and $\Lambda_{g,3}$ are given as
\[\Lambda_{f,3} = \frac{12b^2-2b+9ac}{3b^2} \sqrt{\frac{b-2a}{b-2c}} \quad \mbox{and} \quad \Lambda_{g,3} = \frac{12b^2-2b+9ac}{3b^2} \sqrt{\frac{b-2c}{b-2a}}.\] 
Therefore, we complete the proof.
\end{proof}

\subsection{Decay in the bounded region}
We take $\psi = \sech^2$ in \eqref{Loc_Energy}. Then, Lemma \ref{lem:Coercivity}  ensures the key property
\begin{proposition}[Integrability and decay along a sequence in time]\label{prop:Exterior Region1}
Let $(a,b,c)$ be parameters satisfying \eqref{Conds}. Let $\sigma > 0$ satisfy \eqref{Sigma}. Then, there exist $0 < \varepsilon_0 = \varepsilon_0(a,b,c,\sigma) \ll 1$ and $L_0 = L_0(a,b,c,\sigma) \gg 1$ such that for all $H^1 \times H^1$ global solutions $(u,\eta)(t)$ to \eqref{boussinesq0000} satisfying 
\begin{equation}\label{Smallness_Ext}
\norm{(u,\eta)(t=0)}_{H^1 \times H^1} < \varepsilon,
\end{equation}
for $0 < \varepsilon \le \varepsilon_0$ and for fixed $L \ge L_0$, we have 
\begin{equation}\label{Ext_Decay1}
\int_2^{\infty} \int \sech^2 \left(\frac{x-\sigma t}{L}\right) \left(u^2 + (\px u)^2 + \eta^2 + (\px \eta)^2 \right)(t,x) \, dx\,dt \lesssim L\ve^2.
\end{equation}
As an immediate consequence, there exists an increasing sequence of times $\{t_n\}$ $(t_n \to \infty$ as $n \to \infty)$ such that
\begin{equation}\label{Ext_Decay2}
\int \sech^2 \left(\frac{x - \sigma t_n}{L}\right) \left(u^2 + (\px u)^2 + \eta^2 + (\px \eta)^2 \right)(t_n,x) \; dx \longrightarrow 0\quad \mbox{  as }\quad n \to \infty.
\end{equation}
\end{proposition}

A similar decay property can be found in other fluid models, see e.g. \cite{KM2018,KMPP2018}, and the decay property does not depend on the power of the nonlinearity. 

\medskip

Moreover, an analogous argument as in the proofs of Propositions 7.1 and 7.2 in \cite{KMPP2018} (also Proposition \ref{prop:energy2} above) ensures decay in any compact interval along the line $|x| = \sigma|t|$, for $\sigma>0$ satisfying \eqref{Sigma}.

\begin{proposition}[Strong decay in compact intervals along lines $|x|=\sigma |t|$]\label{prop:Exterior Region2}
Let $(a,b,c)$ be parameters satisfying \eqref{Conds}. Let $\sigma > 0$ satisfy \eqref{Sigma}. Then, there exist $0 < \varepsilon_0 = \varepsilon_0(a,b,c,\sigma) \ll 1$ and $L_0 = L_0(a,b,c,\sigma) \gg 1$ such that for $H^1 \times H^1$ global solutions $(u,\eta)(t)$ to \eqref{boussinesq0000} satisfying \eqref{Smallness_Ext} for $0 < \varepsilon \le \varepsilon_0$ and for fixed $L \ge L_0$, we have 
\begin{equation}\label{Ext_Decay3}
\lim_{t \to \infty} \int \sech^4 \left(\frac{x-\sigma t}{L}\right) \left(u^2 + (\px u)^2 + \eta^2 + (\px \eta)^2\right)(t,x) \; dx = 0.
\end{equation}
\end{proposition}

However, Proposition \ref{prop:Exterior Region2} is not enough to ensure a complete proof of Theorem \ref{TH3}, because the intervals of decay in the former result do not cover the unbounded region claimed in the latter result. Consequently, we need a further result.

\subsection{Complete decay in exterior regions. Main argument}\label{sec:proof exterior}
The proof is based on the ideas in \cite{MM2} (see also \cite{KM2018} for a more recent application). 

\medskip

We want to prove \eqref{Conclusion_Ext} for the positive region.  Let $\sigma = 1+ \delta$ for given $\delta>0$ and define $\widetilde{\sigma} = 1 + \frac{\delta}{2}$. Then we know $1 < \widetilde \sigma < \sigma$. We choose $\psi := \frac12(1+\tanh)$ in \eqref{Loc_Energy}. For $2 < t \le t_0$ and large $L \gg 1$, which, in addition to the smallness condition \eqref{Smallness_Ext}, ensures Lemma \ref{lem:Coercivity}, we define the localized energy functional $\mathcal E_{t_0}(t)$ by
\[
\mathcal E_{t_0}(t) := \frac12 \int \psi \left(\frac{x - \sigma t_0 + \widetilde \sigma (t_0 -t)}{L} \right) \left(- \frac abu_x^2 - \frac cb\eta_x^2 + u^2 + \eta^2 + u^2\eta\right)(t,x)dx.
\]
Note that the smallness condition \eqref{Smallness_Ext} ensures the positivity of the localized energy functional $\mathcal E_{t_0}(t)$, i.e.,
\[0 \le \frac12 \int \psi \left(\frac{x - \sigma t_0 + \widetilde \sigma (t_0 -t)}{L} \right) \left(u_x^2 + \eta_x^2 + u^2 + \eta^2\right)(t,x)dx \lesssim \mathcal E_{t_0}(t).\]
An analogous argument to the proof of Lemma \ref{lem:Coercivity} yields
\[
\begin{aligned}
\frac{d}{dt} \mathcal E_{t_0}(t) \lesssim_{\widetilde \sigma, L} &~  {}  - \int \sech^2 \left(\frac{x - \sigma t_0 + \widetilde \sigma (t_0 -t)}{L}\right) (u^2 + (\px u)^2 + \eta^2 + (\px \eta)^2) \le 0,
\end{aligned}
\]
provided \eqref{Sigma} holds. This reveals that the localized energy functional $\mathcal E_{t_0} (t)$ is decreasing on $[2,t_0]$. 

\medskip

On the other hand, from the fact that $\lim_{x \to -\infty} \psi(x) =0$, we have
\begin{equation}\label{limsup}
\limsup_{t \to \infty}\int \psi \left(\frac{x -\beta t - \gamma }{L} \right) \left(- \frac abu_x^2 - \frac cb\eta_x^2 + u^2 + \eta^2 + u^2\eta \right)(\delta, x)dx =0,
\end{equation}
for any fixed $\beta, \gamma, \delta >0$. Together with all above, for any $2 < t_0$, we have
\[\begin{aligned}
0 &\le \int \psi \left(\frac{x - \sigma t_0}{L} \right) \left(u_x^2 + \eta_x^2 + u^2 + \eta^2\right)(t_0,x)dx\\
&\le  \int \psi \left(\frac{x -(\delta/2) t_0 - 2(1+ (\delta/2))}{L} \right) \left(- \frac abu_x^2 - \frac cb\eta_x^2 + u^2 + \eta^2 + u^2\eta\right)(2,x)dx,
\end{aligned}\]
which, in addition to \eqref{limsup}, implies
\[\lim_{t \to \infty} \int \psi \left(\frac{x - \sigma t}{L} \right) \left(u_x^2 + \eta_x^2 + u^2 + \eta^2\right)(t,x)dx = 0,\]
which completes the proof.

\begin{remark}
Except for a deeper understanding of the condition \eqref{Sigma}, Theorem \ref{TH3} is already completely proved.
\end{remark}
%

\subsection{Understanding the minimal speed condition \eqref{Sigma}} Our main interests now can be described as follows: 
\smallskip
 \ben
 \item Subsubsection \ref{SSS1}: to find the parameter conditions on $(a,b,c)$ for which energy decay is valid in $(-\infty, -(1+\epsilon) t) \cup ((1+\delta) t, \infty)$ for any $\epsilon, \delta >0$.  (Item (1) in Theorem \ref{TH3}.)
 \smallskip
 \item  Subsubsection \ref{SSS2}: to find the essential exterior region (strictly smaller than the interval $(-\infty, -(1+\epsilon) t) \cup ((1+\delta) t, \infty)$, or $|x| \gg |t|$), and parameter conditions on $(a,b,c)$ for which energy decay holds in that region. (Item (2) in Theorem \ref{TH3}.)
 \smallskip
 \item Subsubsection \ref{SSS3}: in the last paragraph, we will classify the parameter conditions on $(a,b,c)$ associated to the exterior regions in terms of $b$, when specified $a = c = \frac16 -b$. (Item (3) in Theorem \ref{TH3}.)
 \een 

\subsubsection{Exterior region I: the case $|x| > |t|$}\label{SSS1}
We first investigate the set of parameters $(a,b,c)$, for which $H^1$-decay of $(u,\eta)$ is valid inside the region $(-\infty, -t) \cup (t, \infty)$. This case is motivated by the observation made in Section \ref{sec:velocity}, where the group velocity of linear waves in the $a=c$ case was considered in some detail. The maximum group velocity for linear waves established in Section \ref{sec:velocity} was exactly $1$. Hence, this formal result naturally poses the energy decay problem outside the light cone.

\medskip

We shall use the alternative expression for $(a,b,c)$ in terms of $(\nu,b)$, the parametrization introduced in \eqref{R0}. Then \eqref{Sigma} reads  
\begin{equation}\label{Sigma2}
\begin{aligned}
\sigma > \max\Bigg(&1, \frac{2(15b-2)}{\sqrt{3}\sqrt{-3\nu^2+2\nu+108b^2-12b}},\\
&\qquad  \frac{-9\nu^2+6\nu+84b^2-20b}{4\sqrt{3}b\sqrt{-3\nu^2+2\nu+27b^2-6b}}, \frac{\sqrt{3}\sqrt{-3\nu^2+2\nu+12b^2-4b}}{2b} \Bigg).
\end{aligned}
\end{equation}

\begin{proposition}\label{prop:exterior1}
Let $(\nu,b)$ be defined as in \eqref{R0}. If additionally the parameters $(\nu,b)$ belong to the interior region bounded by the ellipse
\begin{equation}\label{ellipse}
153b^2-54b+4 \le 9ac,
\end{equation}
then the RHS of \eqref{Sigma2} is 1 and we have \eqref{Coercivity} for any $\sigma > 1$.
\end{proposition}

\begin{remark}\label{ELLIPSE}
An alternative expression in $(\nu,b)$ of \eqref{ellipse} is as follows:
\[\frac{144}{49}\left(\nu - \frac13 \right)^2 + \frac{9216}{49}\left(b-\frac{17}{96}\right)^2 = 1.\]
As an immediate consequence, the corresponding maximum value for $b$ when $\nu=\frac13$, i.e. $a=c$, is $\frac14$. This, in a particular case $a=c$, will be improved in Proposition \ref{prop:classification1} via the improved virial estimate.
\end{remark}


\begin{proof}[Proof of Proposition \ref{prop:exterior1}]
Consider the first and (fourth) last term in the RHS of \eqref{Sigma2}. A computation for all $(\nu,b) \in \mathcal R_0$ yields that 
\begin{equation}\label{1}
\frac{\sqrt{3}\sqrt{-3\nu^2+2\nu+12b^2-4b}}{2b} \le 1 \quad  \Longleftrightarrow \quad \nu^2-\frac23\nu-\frac{32}{9}b^2 + \frac43 b \ge 0.
\end{equation}
Since the discriminant for the above quadratic polynomial satisfies
\[\frac49 -4\left(-\frac{32}{9}b^2 + \frac43 b\right) \le 0,\]
when $\frac18 \le b \le \frac14$, the right-hand side of \eqref{1} is valid for all $(\nu,b) \in \mathcal R_0$ with $b \le \frac14$, and so the left one. Consequently, for $b\in [\frac18,\frac14]$, $(\nu,b) \in \mathcal R_0$,
\[
\max\left(1,\frac{\sqrt{3}\sqrt{-3\nu^2+2\nu+12b^2-4b}}{2b} \right)=1.
\]
Now we compare 1 and $\frac{-9\nu^2+6\nu+84b^2-20b}{4\sqrt{3}b\sqrt{-3\nu^2+2\nu+27b^2-6b}}$. Let 
\begin{equation}\label{eq:C}
C_\nu:=\left(-\frac{\nu}{2}+\frac13\right)\frac{\nu}{2}. 
\end{equation}
A computation gives
\[\frac{1}{48} < C_\nu < \frac{1}{36}, \quad \mbox{if} \quad (\nu,b) \in \mathcal R_0, \quad \frac16 < b \le \frac14.\]
Then, it is known (similarly as \eqref{1}, but it is more complicated) that
\begin{equation}\label{2}
\begin{aligned}
&\frac{-9\nu^2+6\nu+84b^2-20b}{4\sqrt{3}b\sqrt{-3\nu^2+2\nu+27b^2-6b}} \le 1\\
\Longleftrightarrow \quad &~{} 360b^4 -192b^3 +\left(25+342C_\nu\right)b^2 - 90C_\nu b+81C_\nu^2 \le 0.
\end{aligned}
\end{equation}
Thus, it suffices to show that the right-hand side of \eqref{2} is valid for all $(\nu,b) \in \mathcal R_0$ with $\frac16 < b \le \frac14$. 

\medskip

Let $q(b) := 360b^4 -192b^3 +\left(25+342C_\nu\right)b^2 - 90C_\nu b+81C_\nu^2$ be the quartic polynomial in $b$. A straightforward computation gives
\[q'(b) = 1440b^3-576b^3+2\left(25+342C_\nu\right)b -90.\]
Moreover, we find that
\[
q\left(\frac16\right) = \frac{(324C_\nu-11)^2}{1296} - \frac{13}{1296} \le 0, \quad \mbox{if} \quad \frac{11-\sqrt{13}}{324} \le C_\nu \le \frac{11+\sqrt{13}}{324}
\]
and
\[
q\left(\frac14\right) = 81C_\nu^2 - \frac98C_\nu-\frac{1}{32} \le 0, \quad \mbox{if} \quad -\frac{1}{72} \le C_\nu \le \frac{1}{36}.
\]
Hence, we know that both $q(1/6)$  and $q(1/4)$ are nonpositive when $\frac{11-\sqrt{13}}{324} \le C_\nu \le \frac{1}{36}$. Moreover, one can  check that
\[q'(0) < 0, \quad q'\left(\frac1{10}\right) > 0, \quad q'\left(\frac16\right) < 0, \quad q'\left(\frac14\right) > 0,\]
which reveals that $q$ has only one local minimum value at a point in $(1/6, 1/4)$. Thus, we conclude that $q(b) \le 0$ for $b \in (1/6, 1/4)$, when $\frac{11-\sqrt{13}}{324} \le C_\nu \le \frac{1}{36}$.

\medskip

On the other hand, one can see that for $(\nu,b) \in \mathcal R_0$ with $\frac16 < b \le \frac14$, the condition $\frac{1}{48} < C_\nu < \frac{11-\sqrt{13}}{324}$ ensures $(\frac15 < ) \frac{3+\sqrt{\sqrt{13}-2}}{18} < b \le \frac14$. Thus, it suffices to show $q(b) \le 0$ for $\frac15 < b \le \frac14$, when $\frac{1}{48} < C_\nu < \frac{11-\sqrt{13}}{324}$. A not very complicated computation gives
\[
q\left(\frac15\right) = 81C_\nu^2 - \frac{108C_\nu}{25} + \frac{1}{25} \le 0, \quad \mbox{if} \quad \frac{6-\sqrt{11}}{225} \le C_\nu \le \frac{6+\sqrt{11}}{225}.
\]
It is  known that $\frac{6-\sqrt{11}}{225} < \frac{1}{48} < \frac{11-\sqrt{13}}{324} < \frac{6+\sqrt{11}}{225}$, thus, our claim is valid thanks to the inequalities
\[
q'(0) < 0, \quad q'\left(\frac1{10}\right) > 0, \quad q'\left(\frac15\right) < 0, \quad q'\left(\frac14\right) > 0.
\]
Lastly, the condition 
\begin{equation}\label{ellipse0000}
\frac{2(15b-2)}{\sqrt{3}\sqrt{-3\nu^2+2\nu+108b^2-12b}} \le 1
\end{equation}
gives the conclusion \eqref{ellipse} (after substituting $a = -\frac{\nu}{2} + \frac13 -b$ and $c = \frac{\nu}{2} - b$ into \eqref{ellipse0000}), thus completing the proof.
\end{proof}

\begin{remark}
One can show that
\[
\frac{2(15b-2)}{\sqrt{3}\sqrt{-3\nu^2+2\nu+108b^2-12b}}
\]
is always bigger than the others three members in the RHS of \eqref{Sigma} (except for $1$), when $(\nu,b) \in \mathcal R_0$ and $\frac16 < b \le \frac14$. However, we do not pursue it here, in order to avoid another bunch of complicated calculations. 
\end{remark}

\begin{remark}
As explained above, Proposition \ref{prop:exterior1} finally ensures full energy decay inside the interval $\big(-\infty, -(1+\epsilon)t \big) \cup \big((1+\delta)t, \infty\big)$ for any $\epsilon, \delta >0$ (Item (1) in Theorem \ref{TH3}).
\end{remark}


\begin{remark}\label{rem:maximum velocity}
%
%
The candidate of the maximum velocity, for which most of parameters $(a,b,c)$ in \eqref{Loc_Energy} ensure \eqref{Coercivity}, is motivated by the asymptotic limit in \eqref{Sigma2} as $b \to \infty$, since $\nu$ is bounded. Indeed, taking $b \to \infty$ in \eqref{Sigma2}, one has
\[\sigma > \max\left(1, \frac53, \frac73, 3\right) = 3.\]
Indeed, some standard computation give
\[\frac{2(15b-2)}{\sqrt{3}\sqrt{-3\nu^2+2\nu+108b^2-12b}} \le \frac53,\]
\[\frac{-9\nu^2+6\nu+84b^2-20b}{4\sqrt{3}b\sqrt{-3\nu^2+2\nu+27b^2-6b}} \le \frac73,
\]
and
\[\frac{\sqrt{3}\sqrt{-3\nu^2+2\nu+12b^2-4b}}{2b} \le 3,\]
for all $(\nu,b) \in \mathcal R_0$. Thus, one expects that for a sufficiently large $b \gg 1$, $\frac{3\sqrt{ac}}{b}$ is the maximum value in the right-hand side of \eqref{Sigma} for all $(a,b,c)$. Consequently, Proposition \ref{prop:exterior2} below deals with the parameter conditions for which $\frac{\sqrt{ac}}{b}$ is the maximum value in the right-hand side of \eqref{Sigma}.
\end{remark}

\subsubsection{Exterior region II: the case $|x| \gg |t|$}\label{SSS2} Following Remark \ref{rem:maximum velocity}, we have

\begin{proposition}\label{prop:exterior2}
Let $(\nu,b)$ be defined as in \eqref{R0}. Assume $b \ge \frac14$. If, additionally, parameters $(\nu,b)$ belong to the region bounded below by the hyperbola-type curve given by
\begin{equation}\label{hyperbola}
-3\nu^2+2\nu \ge \frac{2b}{9}\left(20-75b+\sqrt{2169b^2-660b+152}\right),
\end{equation}
we have \eqref{Coercivity} for 
\[\sigma > \frac{\sqrt{3}\sqrt{-3\nu^2+2\nu+12b^2-4b}}{2b} = \frac{3\sqrt{ac}}{b}.\]
As a consequence, the corresponding minimum value of $b$ (when $\nu=\frac13$, i.e., $a=c$) is $\frac14$.
\end{proposition}

\begin{remark}\label{HYPERBOLA}
Note that a computation yields the expression \eqref{hyperbola} is equivalent to \eqref{hyperbola0}.
\end{remark}

\begin{proof}[Proof of Proposition \ref{prop:exterior2}]
The claim is to show
\begin{equation}\label{eq:claim}
\begin{aligned}
&\frac{\sqrt{3}\sqrt{-3\nu^2+2\nu+12b^2-4b}}{2b} \ge \frac{-9\nu^2+6\nu+84b^2-20b}{4\sqrt{3}b\sqrt{-3\nu^2+2\nu+27b^2-6b}}\\
\Longleftrightarrow&~{} \frac{\sqrt{3}\sqrt{-3\nu^2+2\nu+12b^2-4b}}{2b} \ge 1 \\
&~{}\mbox{and} \quad \frac{\sqrt{3}\sqrt{-3\nu^2+2\nu+12b^2-4b}}{2b} \ge \frac{2(15b-2)}{\sqrt{3}\sqrt{-3\nu^2+2\nu+108b^2-12b}}.
\end{aligned}
\end{equation}
We again use $C_\nu$ as in \eqref{eq:C}. Note that $C_\nu \ge -\frac1{12}$. Then, the left-hand side of \eqref{eq:claim} is equivalent to 
\begin{equation}\label{eq:claim1}
243C_\nu^2 + 90b(15b-4)C_\nu + 4b^2(288b^2-195b+29) \ge 0.
\end{equation}
Since 
\[
\frac{2b}{54}\left(20-75b-\sqrt{2169b^2-660b+152}\right) < -\frac1{12}, \quad \mbox{for} \quad b \ge \frac14,
\]
\eqref{eq:claim1} is equivalent to the condition
\[
C_\nu \ge \frac{2b}{54}\left(20-75b+\sqrt{2169b^2-660b+152}\right),
\]
for $(\nu,b) \in \mathcal R_0$, which is exactly \eqref{hyperbola}. Thus the claim \eqref{eq:claim}  implies Proposition \ref{prop:exterior2}. 

\medskip

An analogous computation in the right-hand side of \eqref{eq:claim} gives
\[\frac{\sqrt{3}\sqrt{-3\nu^2+2\nu+12b^2-4b}}{2b} \ge 1 \Longleftrightarrow C_\nu \ge \frac{2b}{9}\left(3-8b\right)\]
and
\begin{equation}\label{eq:claim2}
\begin{aligned}
&\frac{\sqrt{3}\sqrt{-3\nu^2+2\nu+12b^2-4b}}{2b} \ge \frac{2(15b-2)}{\sqrt{3}\sqrt{-3\nu^2+2\nu+108b^2-12b}}\\
\Longleftrightarrow \quad &~{} 81C_\nu^2 + 108b(15b-2)C_\nu +4b^2(504b^2-264b+23) \ge 0.
\end{aligned}
\end{equation}
Since 
\[
\frac{2b}{9}\left(6-45b-\sqrt{1521b^2-276b+13}\right)\ <\ -\frac{1}{12}, \quad \mbox{for} \quad b \ge \frac14,
\]
the right-hand side of \eqref{eq:claim2} is equivalent to 
\[
C_\nu \ge \frac{2b}{9}\left(6-45b+\sqrt{1521b^2-276b+13}\right).
\]
It is not difficult to show that
\[ 
\begin{aligned}
\frac{2b}{9}\left(3-8b\right) \le &~ \frac{2b}{9}\left(6-45b+\sqrt{1521b^2-276b+13}\right) \\
\le &~ {} \frac{2b}{54}\left(20-75b+\sqrt{2169b^2-660b+152}\right),
\end{aligned}
\]
for $b \ge \frac14$. Hence, we prove the claim \eqref{eq:claim}, and so Proposition \ref{prop:exterior2}.
\end{proof}

\begin{remark}
Proposition \ref{prop:exterior2} finally shows energy decay to zero in the set \eqref{I_ext2}, and proves Item (2) in Theorem \ref{TH3}.
\end{remark}

The only remaining case to prove now is Item (3) in Theorem \ref{TH3}.

\subsubsection{Analysis in the special case $a=c$}\label{SSS3}
Assuming $a = c = \frac16 -b$, one has the following direct consequence of Propositions \ref{prop:exterior1} and \ref{prop:exterior2}:

\begin{proposition}[Preliminary energy decay conditions]\label{prop:classification}
Let $a = c = \frac16 - b$, for $b > \frac16$. Then, we have \eqref{Coercivity} for
\begin{enumerate}
\item \emph{(Full energy decay)} $\sigma > 1$, if $\frac16 < b \le \frac14$.
\smallskip
\item \emph{(Growing speed as $b\to +\infty$)} $\sigma > 3 - \frac{1}{2b}>1$, if $b \ge \frac14$.
\end{enumerate}
\end{proposition}
This last result is not optimal (see item (3) in Theorem \ref{TH3}), so we will try to improve it. Indeed, Lemma \ref{lem:Coercivity} is not optimal even in the special case $a=c$, thus neither Proposition \ref{prop:classification}, which is based in Lemma \ref{lem:Coercivity}. 

\medskip

The remaining part of this section will be devoted to improve Proposition \ref{prop:classification} in essentially two directions\footnote{One can try to improve Lemma \ref{lem:Coercivity} in the general case, i.e., $a \neq c$, but we will not pursue this path due essentially to very complicated calculations.}: to increase the range of $b$ for which decay holds for $\sigma > 1$, and to prove decay in a wider exterior region (see Fig. \ref{fig:classification} for the description of Proposition \ref{prop:classification1} below).

\begin{proposition}[Improvement of Proposition \ref{prop:classification}]\label{prop:classification1}
Let $a = c = \frac16 - b$, for $b > \frac16$. Then, we have \eqref{Coercivity} for
\begin{enumerate}
\item $\sigma > 1$, if $\frac16 < b \le \frac{3+\sqrt{3}}{12}$.
\smallskip
\item $\sigma > \frac{2(b-\frac16)(b-\frac18)}{b(b-\frac1{12})} 
$, if $b \ge \frac{3+\sqrt{3}}{12}$.
\end{enumerate}
\end{proposition}

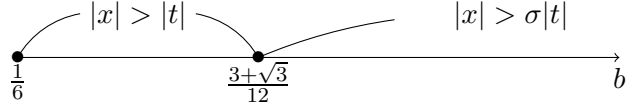
\begin{figure}[h!]
\begin{center}
\begin{tikzpicture}[scale=0.8]
\draw[->] (-5,3) -- (5,3) node[below] {$b$};
\node at (-5,3){$\bullet$};
\node at (-1,3){$\bullet$};
\node at (-5,2.6){$\frac16$};
\node at (-1,2.6){$\frac{3+\sqrt{3}}{12}$};
\draw (-5,3) arc (150:100:1.5);
\node at (-3,3.7){$|x| > |t|$};
\draw (-1,3) arc (30:80:1.5);
\draw (-1,3) arc (113:98:9);
\node at (3.2,3.7){$|x| > \sigma|t|$};
\end{tikzpicture}
\end{center}
\caption{Schematic representation of Proposition \ref{prop:classification1}. The exterior region $I_{ext}(t)$ in item (3), Theorem \ref{TH3}, for which the energy decay property holds, is determined by the corresponding value of $b$ as above depicted, when $a=c$. Note that the threshold value $b = \frac{3+\sqrt{3}}{12}$ (corresponding to $(a,c) = (-\frac{1+\sqrt{3}}{12},-\frac{1+\sqrt{3}}{12})$) is the maximum value found by us for which the energy decay property holds in whole exterior region outside the light cone. Above $b = \frac{3+\sqrt{3}}{12}$ the exterior region outside the light cone depends on $b$, i.e., $|x| > \sigma(b) |t|$, where $\sigma$ is given in Item (2) in Proposition \ref{prop:classification1}.} \label{fig:classification}
\end{figure}

\begin{remark}
Proposition \ref{prop:classification1} ends the proof of Item (3) in Theorem \ref{TH3}.
\end{remark}

\begin{proof}[Proof of Proposition \ref{prop:classification1}]
Recall the functional $\mathcal I(t)$ defined in \eqref{I}. Consider this functional but now with a weight $\varphi(t,x) = \psi(t,x)$, see \eqref{psi(t,x)} and \eqref{psi_final}. Then, a computation with $|\tau| < 1$\footnote{The restriction $|\tau| < 1$ is due to the positivity of $E_{loc}(t) + \tau\mathcal I(t)$.} gives
\[\begin{aligned}
2L\frac{d}{dt}(&E_{loc}(t) + \tau\mathcal I(t)) \\
=&~{} \int \psi' \left( \widetilde A_1(f^2+ g^2) + \widetilde A_2 (f_x^2 + g_x^2) + \widetilde A_3(f_{xx}^2+g_{xx}^2) + \widetilde A_4(f_{xxx}^2+g_{xxx}^2) \right)\\
&+\int \psi'\left( 2\widetilde B_1 fg + 2 \widetilde B_2 f_xg_x +2 \widetilde B_3 f_{xx}g_{xx} +2 \widetilde B_4 f_{xxx}g_{xxx}\right)\\
&+O\left(\frac1L, \norm{(u,\eta)}_{H^1 \times H^1}\right),
\end{aligned}\]
where
\[\begin{aligned}
&\widetilde A_1 = -\sigma + \tau, \quad \widetilde A_2 = -\sigma\left(3-\frac{1}{6b}\right)+\tau \left(3-\frac{1}{2b}\right), \\
&\widetilde A_3 = -\sigma\left(3-\frac{1}{3b}\right) + \tau\left(3-\frac{2}{3b}\right), \quad \widetilde A_4 =-\sigma\left(1-\frac{1}{6b}\right) + \tau\left(1-\frac{1}{6b}\right)\\
&\widetilde B_1 =  1-\sigma \tau, \quad \widetilde B_2 = \left(5-\frac{2}{3b}\right)  - 3\sigma\tau,\\
&\widetilde B_3 = \left(7-\frac{5}{3b}+\frac{1}{12b^2}\right) - 3\sigma \tau, \quad \widetilde B_4 = \left(3 - \frac{1}{b} + \frac{1}{12b^2}\right)-\sigma \tau.
\end{aligned}\]
Thanks to Proposition \ref{prop:classification}, we may assume $b > \frac14$.

\medskip

\noindent
\begin{lemma}\label{claim_final} For each $b > \frac14$, there exist $\tau = \tau(b)$ and $\sigma = \sigma(b)$ satisfying $|\tau|<1$ and $\sigma > 0$, respectively, such that
\begin{equation}\label{Ext_AB}
\widetilde A_i + |\widetilde B_i| < 0, \quad i=1,2,3,4.
\end{equation}
\end{lemma}
This lemma immediately implies
\begin{equation}\label{Coer}
2L \frac{d}{dt}\left(E_{loc}(t) + \tau\mathcal I(t)\right) \le -c_0\int \psi' (u_x^2 +u^2 + \eta_x^2 + \eta^2),
\end{equation}
for a set of $\sigma > 0$, thanks to the Cauchy-Schwarz inequality.

\begin{proof}[Proof of Lemma \ref{claim_final}]
We first consider $\widetilde A_1 + |\widetilde B_1|$. A computation gives
\[
-\sigma + \tau + |1-\sigma \tau| = \tau (1-\sigma) + (1-\sigma), \qquad \tau \le \frac1{\sigma},
\]
and
\[
-\sigma + \tau + |1-\sigma \tau| = \tau(1+\sigma)-(1+\sigma), \qquad \tau > \frac1{\sigma}.
\]
Thus, we conclude
\begin{equation}\label{11}
\widetilde A_1 + |\widetilde B_1| < 0 \Longleftrightarrow -1 < \tau < 1 \quad \sigma > 1.
\end{equation}

\medskip

Now we focus on proving that $\widetilde A_4 + |\widetilde B_4|<0$. 

\medskip

Let $\widetilde b = 1 - \frac{1}{6b}$ be as in Section \ref{sec:velocity}. Similarly as before, we have
\begin{equation}\label{Ext_A4B4_1}
-\sigma\widetilde b + \tau \widetilde b + |3\widetilde b^2-\sigma \tau| = \tau (\widetilde b-\sigma) -\sigma\widetilde b + 3\widetilde b^2, \quad \tau \le \frac{3\widetilde b^2}{\sigma}
\end{equation}
and
\begin{equation}\label{Ext_A4B4_2}
-\sigma\widetilde b + \tau \widetilde b + |3\widetilde b^2-\sigma \tau| = \tau (\widetilde b+\sigma) -\sigma\widetilde b - 3\widetilde b^2, \quad \tau > \frac{3\widetilde b^2}{\sigma}.
\end{equation}
Then, a computation in \eqref{Ext_A4B4_1}  gives that
\[\tau (\widetilde b-\sigma) -\sigma\widetilde b + 3\widetilde b^2 < 0 \quad \Longleftrightarrow \quad  \frac{\sigma\widetilde b - 3\widetilde b^2}{\widetilde b-\sigma} < \tau  \quad  \left(\le \frac{3\widetilde b^2}{\sigma}\right),\]
under the condition $\widetilde b < \sigma$. Note that 
\[\tau < \frac{\sigma\widetilde b - 3\widetilde b^2}{\widetilde b-\sigma}, \quad 0 < \sigma < \widetilde b\]
and $-1 < \tau$ have no intersection with respect to $\tau$, when $b > \frac14$ ($\Leftrightarrow \widetilde b > \frac13$). Similarly, we have in \eqref{Ext_A4B4_2} that
\[\tau (\widetilde b+\sigma) -\sigma\widetilde b - 3\widetilde b^2 < 0 \Longleftrightarrow   \frac{3\widetilde b^2}{\sigma} < \tau  < \frac{\sigma\widetilde b + 3\widetilde b^2}{\widetilde b + \sigma},\]
under the condition $\sigma>0$. Collecting all these, one has 
\begin{equation}\label{44}
\widetilde A_4 + |\widetilde B_4| < 0 \quad \Longleftrightarrow \quad \max\left(-1, \frac{\sigma\widetilde b - 3\widetilde b^2}{\widetilde b-\sigma}\right) < \tau < \min\left(1, \frac{\sigma\widetilde b + 3\widetilde b^2}{\widetilde b + \sigma}\right), \quad \sigma > \widetilde b.
\end{equation}
Note that $- 1 < \frac{\sigma\widetilde b - 3\widetilde b^2}{\widetilde b-\sigma} $ for $b > \frac14$ and $\sigma > 0$. Thus, we can choose $|\tau| < 1$ such that $\widetilde A_4 + |\widetilde B_4| <0$, if 
\begin{equation}\label{Sigma1111}
\sigma > \max\left(\frac{3 \widetilde b^2 + \widetilde b}{1+\widetilde b}, \sqrt{3}\widetilde b \right) = \max\left( \frac{48b^2 -14b +1}{24b^2-2b}, \sqrt{3}\left(1-\frac{1}{6b}\right) \right).
\end{equation}
Remark that $\sqrt{3}\widetilde b$ is the maximum in \eqref{Sigma1111}, when $\widetilde b \le \frac{1}{\sqrt{3}} $ ($\Leftrightarrow b \le \frac{3+\sqrt{3}}{12}$), and, in the same region, $\sqrt{3}\widetilde b \le 1$. However, when $\widetilde b > \frac{1}{\sqrt{3}} $ ($\Leftrightarrow b > \frac{3+\sqrt{3}}{12}$), we have 
\[\max\left(\sqrt{3}\left(1-\frac{1}{6b}\right), \frac{48b^2 -14b +1}{24b^2-2b} \right) =  \frac{48b^2 -14b +1}{24b^2-2b} > 1.\]
This is the essential part of Proposition \ref{prop:classification1}.

\medskip

The rest of the proof is divided in two parts: one step is to prove that there exists $|\tau|<1$ such that \eqref{Ext_AB} holds for $\sigma >1$, when $\frac14 < b \le \frac{3+\sqrt{3}}{12}$ (for Item (1)). The other step consists in proving that $(\tau, \sigma)$ satisfying $\widetilde A_4 + |\widetilde B_4| < 0$ also satisfies $\widetilde A_i + |\widetilde B_i| < 0$, $i=2,3$ (for Item (2)).

\medskip

An analogous argument (omitting the calculations) yields\footnote{Each minimum value in \eqref{22} and \eqref{33} is always greater than $-1$, if $\sigma > 0$.}
\begin{equation}\label{22}
\begin{aligned}
&\widetilde A_2 + |\widetilde B_2| < 0 \\
\Longleftrightarrow&~{} \frac{\sigma(2+\widetilde b)-1 - 4\widetilde b}{3(\widetilde b-\sigma)} < \tau < \min\left(1, \frac{\sigma(2+\widetilde b)+1 + 4\widetilde b}{3(\widetilde b + \sigma)}\right),
\end{aligned}
\end{equation}
for $\sigma > \widetilde b$, and
\begin{equation}\label{33}
\begin{aligned}
&\widetilde A_3 + |\widetilde B_3| < 0 \\
\Longleftrightarrow&~{} \frac{\sigma(1 + 2\widetilde b)-3\widetilde b^2-4\widetilde b}{4\widetilde b-1-3\sigma} < \tau < \min\left(1, \frac{\sigma(1+2\widetilde b)+3\widetilde b^2+4\widetilde b}{4\widetilde b -1 +3\sigma}\right),
\end{aligned}
\end{equation}
for $\sigma > \frac{4\widetilde b-1}{3}$. Note that $\frac{4\widetilde b-1}{3} > \frac{1- 4\widetilde b}{3}$ for $\widetilde b > \frac13$ ($\Leftrightarrow b > \frac14$).

\medskip

We first address the regime $\frac14 < b \le \frac{3+\sqrt{3}}{12}$ ($\Leftrightarrow \frac13 < \widetilde b \le \frac{1}{\sqrt{3}}$). On one hand, a computation gives
\[\frac{\sigma\widetilde b + 3\widetilde b^2}{\widetilde b + \sigma} \le \frac{\sigma(2+\widetilde b)+1 + 4\widetilde b}{3(\widetilde b + \sigma)} \Longleftrightarrow \frac{9\widetilde b^2-4\widetilde b-1}{2(1-\widetilde b)} \le \sigma.\]
The right-hand side is always true for $\sigma > 0$. On the other hand, an analogous calculation gives
\[\begin{aligned}
\frac{\sigma\widetilde b + 3\widetilde b^2}{\widetilde b + \sigma} \le \frac{\sigma(1+2\widetilde b)+3\widetilde b^2+4\widetilde b}{4\widetilde b -1 +3\sigma}\Longleftrightarrow (1-\widetilde b)\sigma^2 + (6\widetilde b - 8\widetilde b^2)\sigma + (7\widetilde b^2-9\widetilde b^3) \ge 0.
\end{aligned}\]
The right-hand side is always true, if
\begin{equation}\label{<<<}
(6\widetilde b - 8\widetilde b^2)^2-4(1-\widetilde b)(7\widetilde b^2-9\widetilde b^3) \le 0 \Longleftrightarrow (\frac13 <)~ \frac{4-\sqrt{2}}{7} \le \widetilde b \le \frac{4+\sqrt{2}}{7}.
\end{equation}
Otherwise (if $\frac13 < b < \frac{4-\sqrt{2}}{7}$), the right-hand side is equivalent to
\[\sigma \ge \frac{-(6\widetilde b - 8\widetilde b^2) + \sqrt{(6\widetilde b - 8\widetilde b^2)^2-4(1-\widetilde b)(7\widetilde b^2-9\widetilde b^3)}}{2(1-\widetilde b)},\]
which is always true for $\sigma >0$. 

\medskip

Thus, for each $\frac13 < \widetilde b \le \frac1{\sqrt{3}}$, if 
\begin{equation}\label{55}
\max\left(\frac{\sigma(2+\widetilde b)-1 - 4\widetilde b}{3(\widetilde b-\sigma)}, \frac{\sigma(1 + 2\widetilde b)-3\widetilde b^2-4\widetilde b}{4\widetilde b-1-3\sigma}, \frac{\sigma\widetilde b - 3\widetilde b^2}{\widetilde b-\sigma}\right) < \min\left(1, \frac{\sigma\widetilde b + 3\widetilde b^2}{\widetilde b + \sigma}\right),
\end{equation}
we have \eqref{Coer}. Note that 
\[1\le \frac{\sigma\widetilde b + 3\widetilde b^2}{\widetilde b + \sigma} \Longleftrightarrow \sigma \le\frac{3\widetilde b^2 - \widetilde b}{1 - \widetilde b}.\]
Since $\sigma > \widetilde b$ and 
\[\frac{\sigma\widetilde b - 3\widetilde b^2}{\widetilde b-\sigma} < 1 \Longleftrightarrow \sigma > \frac{3 \widetilde b^2 + \widetilde b}{1+\widetilde b},\] 
over the region of $\widetilde b$, for which 
\[\widetilde b< \frac{3\widetilde b^2 - \widetilde b}{1 - \widetilde b}\]
holds true (indeed $(\frac12,\frac{1}{\sqrt{3}}]$), there is no $\sigma >0$ such that both
\[\widetilde b< \sigma \le \frac{3\widetilde b^2 - \widetilde b}{1 - \widetilde b} \quad \mbox{and} \quad \sigma > \frac{3 \widetilde b^2 + \widetilde b}{1+\widetilde b}\]
satisfy at the same time. Thus, \eqref{55} is further reduced as
\[\max\left(\frac{\sigma(2+\widetilde b)-1 - 4\widetilde b}{3(\widetilde b-\sigma)}, \frac{\sigma(1 + 2\widetilde b)-3\widetilde b^2-4\widetilde b}{4\widetilde b-1-3\sigma}, \frac{\sigma\widetilde b - 3\widetilde b^2}{\widetilde b-\sigma}\right) < \frac{\sigma\widetilde b + 3\widetilde b^2}{\widetilde b + \sigma},\]
under the condition
\[\sigma > \max\left(\widetilde b, \frac{3\widetilde b^2 - \widetilde b}{1 - \widetilde b}\right).\]
In this case, we obtain
\[\sigma > \max\left(\frac{-\left(7\widetilde b^2-2\widetilde b-1\right)+\sqrt{\left(7\widetilde b^2-2\widetilde b-1\right)^2+4\left(4\widetilde b+2\right)\left(\widetilde b^3+4\widetilde b^2+x\right)}}{2\left(4\widetilde b+2\right)}, \sqrt{3}\,\widetilde b\right).\]
It is not difficult to show that the right-hand side is less than $1$, when $\frac13 < \widetilde b < \frac{1}{\sqrt{3}}$, and hence we, in addition to \eqref{11}, complete the proof of Lemma \ref{claim_final} in the case of Proposition \ref{prop:classification1}, Item (1).

\medskip

In order to prove Item (2) of the same proposition, it suffices to assume that $b > \frac{3+\sqrt{3}}{12}$ ($\Leftrightarrow 1> \widetilde b > \frac{1}{\sqrt{3}}$) and $\sigma >1$. A straightforward calculation reveals
\[ \frac{\sigma(2+\widetilde b)-1 - 4\widetilde b}{3(\widetilde b-\sigma)} < \frac{\sigma\widetilde b - 3\widetilde b^2}{\widetilde b-\sigma} \quad \mbox{and} \quad \frac{\sigma(1 + 2\widetilde b)-3\widetilde b^2-4\widetilde b}{4\widetilde b-1-3\sigma} < \frac{\sigma\widetilde b - 3\widetilde b^2}{\widetilde b-\sigma},\]
for all $\sigma > 1$ and $\frac{1}{\sqrt{3}} < \widetilde b <1$. 

\medskip

In view of  \eqref{22}, \eqref{33} and \eqref{44}, once we prove
\begin{equation}\label{>1}
\frac{\sigma\widetilde b + 3\widetilde b^2}{\widetilde b + \sigma} \ge 1 \quad \Longrightarrow \quad \frac{\sigma(2+\widetilde b)+1 + 4\widetilde b}{3(\widetilde b + \sigma)}, \quad \frac{\sigma(1+2\widetilde b)+3\widetilde b^2+4\widetilde b}{4\widetilde b -1 +3\sigma} \ge 1
\end{equation} 
and
\begin{equation}\label{<1}
\frac{\sigma\widetilde b + 3\widetilde b^2}{\widetilde b + \sigma} < 1 \quad \Longrightarrow \quad \frac{\sigma\widetilde b + 3\widetilde b^2}{\widetilde b + \sigma} < \frac{\sigma(2+\widetilde b)+1 + 4\widetilde b}{3(\widetilde b + \sigma)},  \quad\frac{\sigma(1+2\widetilde b)+3\widetilde b^2+4\widetilde b}{4\widetilde b -1 +3\sigma},
\end{equation} 
under $\sigma > 1$ and $\frac{1}{\sqrt{3}} < \widetilde b <1$, we immediate prove that $(\tau, \sigma)$ satisfying $\widetilde A_4 + |\widetilde B_4| < 0$ also satisfies $\widetilde A_i + |\widetilde B_i| < 0$, $i=2,3$, which completes the proof for Item (2).

\medskip

A not very complicated computation gives
\[
\frac{\sigma\widetilde b + 3\widetilde b^2}{\widetilde b + \sigma} \ge 1 \Longleftrightarrow \sigma \le \frac{3\widetilde b^2 - \widetilde b}{1 - \widetilde b},
\]
\[
\frac{\sigma(2+\widetilde b)+1 + 4\widetilde b}{3(\widetilde b + \sigma)} \ge 1 \Longleftrightarrow \sigma \le \frac{\widetilde b+1}{1-\widetilde b},
\]
and
\[\frac{\sigma(1+2\widetilde b)+3\widetilde b^2+4\widetilde b}{4\widetilde b -1 +3\sigma} \ge 1 \Longleftrightarrow \sigma \le \frac{3 \widetilde b^2 +1}{2(1-\widetilde b)}.\]
When $-\frac13 < \widetilde b < 1$, we show
\[\frac{3\widetilde b^2 - \widetilde b}{1 - \widetilde b} ~\le~ \frac{\widetilde b+1}{1-\widetilde b},~ \frac{3 \widetilde b^2 +1}{2(1-\widetilde b)},\]
which implies \eqref{>1}.

\medskip

On one hand, an analogous computation yields
\[\frac{\sigma\widetilde b + 3\widetilde b^2}{\widetilde b + \sigma} < 1 \Longleftrightarrow \sigma > \frac{3\widetilde b^2 - \widetilde b}{1 - \widetilde b}\]
and
\[\frac{\sigma\widetilde b + 3\widetilde b^2}{\widetilde b + \sigma} < \frac{\sigma(2+\widetilde b)+1 + 4\widetilde b}{3(\widetilde b + \sigma)} \Longleftrightarrow \sigma > \frac{9\widetilde b^2 -4\widetilde b-1}{2(1-\widetilde b)}.\]
When $-\frac13 < \widetilde b < 1$, we show
\[\frac{3\widetilde b^2 - \widetilde b}{1 - \widetilde b} > \frac{9\widetilde b^2 -4\widetilde b-1}{2(1-\widetilde b)},\]
which implies 
\begin{equation}\label{part1}
\frac{\sigma\widetilde b + 3\widetilde b^2}{\widetilde b + \sigma} < 1 \Longrightarrow \frac{\sigma\widetilde b + 3\widetilde b^2}{\widetilde b + \sigma} < \frac{\sigma(2+\widetilde b)+1 + 4\widetilde b}{3(\widetilde b + \sigma)}.
\end{equation}
On the other hand, we also have
\[\frac{\sigma(1+2\widetilde b)+3\widetilde b^2+4\widetilde b}{4\widetilde b -1 +3\sigma} \ge 1 \Longleftrightarrow \sigma^2(1-\widetilde b) + \sigma(6\widetilde b-8\widetilde b^2) + \widetilde b^2(7-9\widetilde b) > 0.\]
From \eqref{<<<}, the right-hand side is always true for $\frac{1}{\sqrt{3}} < \widetilde b < \frac{4+\sqrt{2}}{7}$. Otherwise ($\frac{4+\sqrt{2}}{7} \le \widetilde b$), we obtain
\[\sigma > \frac{\widetilde b}{1-\widetilde b}\left((4\widetilde b-3) + \sqrt{(4\widetilde b-3)^2-(1-\widetilde b)(7-9\widetilde b)}\right).\]
A computation ensures 
\[\frac{3\widetilde b^2 - \widetilde b}{1 - \widetilde b} > \frac{\widetilde b}{1-\widetilde b}\left((4\widetilde b-3) + \sqrt{(4\widetilde b-3)^2-(1-\widetilde b)(7-9\widetilde b)}\right),\]
whenever $-\frac13 < \widetilde b < 1$, which, in addition to \eqref{part1}, implies \eqref{<1}. Thus, we complete the proof of Lemma \ref{claim_final} and Proposition \ref{prop:classification1}.
\end{proof}
\end{proof}

\end{document}